\documentclass[9pt]{amsart}
\usepackage{amsmath}
\usepackage{amsxtra}
\usepackage{amscd}
\usepackage{amsthm}
\usepackage{amsfonts}
\usepackage{amssymb}
\usepackage{eucal}
\usepackage[all]{xy}
\usepackage{graphicx}

\newtheorem{cor}[subsubsection]{Corollary}
\newtheorem{lem}[subsubsection]{Lemma}
\newtheorem{prop}[subsubsection]{Proposition}

\newtheorem{thm}[subsubsection]{Theorem}

\theoremstyle{definition}

\theoremstyle{remark}
\newtheorem{rem}[subsubsection]{Remark}

\newcommand{\thmref}[1]{Theorem~\ref{#1}}

\newcommand{\secref}[1]{Sect.~\ref{#1}}
\newcommand{\lemref}[1]{Lemma~\ref{#1}}
\newcommand{\propref}[1]{Proposition~\ref{#1}}
\newcommand{\corref}[1]{Corollary~\ref{#1}}

\numberwithin{equation}{section}

\newcommand{\nc}{\newcommand}
\nc{\renc}{\renewcommand}
\nc{\ssec}{\subsection}
\nc{\sssec}{\subsubsection}
\nc{\on}{\operatorname}

\nc\ol{\overline}
\nc\wt{\widetilde}
\nc\tboxtimes{\wt{\boxtimes}}
\nc\tstar{\wt{\star}}
\nc{\alp}{a}

\nc{\ZZ}{{\mathbb Z}}
\nc{\NN}{{\mathbb N}}
\nc{\OO}{{\mathbb O}}
\renc{\SS}{{\mathbb S}}
\nc{\DD}{{\mathbb D}}
\nc{\GG}{{\mathbb G}}

\nc{\Fq}{{\mathbb F}_q}
\nc{\Fqb}{\ol{{\mathbb F}_q}}
\nc{\Ql}{\ol{{\mathbb Q}_\ell}}
\nc{\id}{\text{id}}
\nc\X{\mathcal X}

\nc{\Hom}{\on{Hom}}
\nc{\Lie}{\on{Lie}}
\nc{\Loc}{\on{Loc}}
\nc{\Pic}{\on{Pic}}
\nc{\Bun}{\on{Bun}}
\nc{\IC}{\on{IC}}
\nc{\Fls}{\on{Fl}^{\frac{\infty}{2}}}
\nc{\ICs}{\on{IC}^{\frac{\infty}{2}}}
\nc{\ICsl}{\on{IC}^{\lambda+\frac{\infty}{2}}}
\nc{\ICslm}{\on{IC}^{\lambda+\frac{\infty}{2},-}}
\nc{\ICsm}{\on{IC}^{\frac{\infty}{2},-}}
\nc{\Aut}{\on{Aut}}
\nc{\rk}{\on{rk}}
\nc{\Sh}{\on{Sh}}
\nc{\Perv}{\on{Perv}}
\nc{\pos}{{\on{pos}}}
\nc{\Conv}{\on{Conv}}
\nc{\Sph}{\on{Sph}}
\nc{\Sat}{\on{Sat}}
\nc{\Sym}{\on{Sym}}
\nc{\BunBb}{\overline{\Bun}_B}
\nc{\BunNb}{\overline{\Bun}_N}
\nc{\BunTb}{\overline{\Bun}_T}
\nc{\BunBbm}{\overline{\Bun}_{B^-}}
\nc{\BunBbel}{\overline{\Bun}_{B,el}}
\nc{\BunBbmel}{\overline{\Bun}_{B^-,el}}
\nc{\Buno}{\overset{o}{\Bun}}
\nc{\BunPb}{{\overline{\Bun}_P}}
\nc{\BunBM}{\Bun_{B(M)}}
\nc{\BunBMb}{\overline{\Bun}_{B(M)}}
\nc{\BunPbw}{{\widetilde{\Bun}_P}}
\nc{\BunBP}{\widetilde{\Bun}_{B,P}}
\nc{\GUb}{\overline{G/U}}
\nc{\GUPb}{\overline{G/U(P)}}

\nc{\Hhom}{\underline{\on{Hom}}}
\nc\syminfty{\on{Sym}^{\infty}}
\nc\lal{\ol{\kappa_x}}
\nc\xl{\ol{x}}
\nc\thl{\ol{\theta}}
\nc\nul{\ol{\nu}}
\nc\mul{\ol{\mu}}
\nc\Sum\Sigma
\nc{\oX}{\overset{o}{X}{}}
\nc{\hl}{\overset{\leftarrow}h{}}
\nc{\hr}{\overset{\rightarrow}h{}}
\nc{\M}{{\mathcal M}}
\nc{\N}{{\mathcal N}}
\nc{\F}{{\mathcal F}}
\nc{\D}{{\mathcal D}}
\nc{\Q}{{\mathcal Q}}
\nc{\Y}{{\mathcal Y}}
\nc{\G}{{\mathcal G}}
\nc{\E}{{\mathcal E}}
\nc{\CalC}{{\mathcal C}}
\nc\Dh{\widehat{\D}}

\nc{\C}{{\mathcal C}}
\nc{\K}{{\mathcal K}}
\renewcommand{\H}{{\mathcal H}}

\nc{\T}{{\mathcal T}}
\nc{\V}{{\mathcal V}}
\renc{\P}{{\mathcal P}}
\nc{\A}{{\mathcal A}}
\nc{\B}{{\mathcal B}}
\nc{\U}{{\mathcal U}}

\nc{\Gr}{{\on{Gr}}}

\nc{\frn}{{\check{\mathfrak u}(P)}}

\nc{\fC}{\mathfrak C}
\nc{\p}{\mathfrak p}
\nc{\q}{\mathfrak q}
\nc\f{{\mathfrak f}}

\nc{\qo}{{\mathfrak q}}
\nc{\po}{{\mathfrak p}}
\nc{\s}{{\mathfrak s}}
\nc\w{\text{w}}

\renewcommand{\mod}{{\on{-mod}}}

\nc\mathi\iota
\nc\Spec{\on{Spec}}
\nc\Mod{\on{Mod}}
\nc{\tw}{\widetilde{\mathfrak t}}
\nc{\pw}{\widetilde{\mathfrak p}}
\nc{\qw}{\widetilde{\mathfrak q}}
\nc{\jw}{\widetilde j}

\nc{\grb}{\overline{\Gr}}
\nc{\I}{\mathcal I}

\nc{\kappach}{{\check\kappa_x}}
\nc{\Lambdach}{{\check\Lambda}{}}
\nc{\much}{{\check\mu}}
\nc{\omegach}{{\check\omega}}
\nc{\nuch}{{\check\nu}}
\nc{\etach}{{\check\eta}}
\nc{\alphach}{{\checka}}
\nc{\oblvtach}{{\check\oblvta}}
\nc{\pich}{{\check\pi}}
\nc{\ch}{{\check h}}

\nc{\Hb}{\overline{\H}}

\nc{\BA}{{\mathbb{A}}}
\nc{\BC}{{\mathbb{C}}}
\nc{\BE}{{\mathbb{E}}}
\nc{\BF}{{\mathbb{F}}}
\nc{\BG}{{\mathbb{G}}}
\nc{\BM}{{\mathbb{M}}}
\nc{\BO}{{\mathbb{O}}}
\nc{\BD}{{\mathbb{D}}}
\nc{\BN}{{\mathbb{N}}}
\nc{\BP}{{\mathbb{P}}}
\nc{\BQ}{{\mathbb{Q}}}
\nc{\BR}{{\mathbb{R}}}
\nc{\BZ}{{\mathbb{Z}}}
\nc{\BS}{{\mathbb{S}}}

\nc{\CA}{{\mathcal{A}}}
\nc{\CB}{{\mathcal{B}}}

\nc{\CE}{{\mathcal{E}}}
\nc{\CF}{{\mathcal{F}}}
\nc{\CG}{{\mathcal{G}}}
\nc{\CH}{{\mathcal{H}}}

\nc{\CL}{{\mathcal{L}}}
\nc{\CC}{{\mathcal{C}}}
\nc{\CM}{{\mathcal{M}}}
\nc{\CN}{{\mathcal{N}}}
\nc{\cCN}{\check{{\mathcal{N}}}}
\nc{\CK}{{\mathcal{K}}}
\nc{\CO}{{\mathcal{O}}}
\nc{\CP}{{\mathcal{P}}}
\nc{\CQ}{{\mathcal{Q}}}
\nc{\CR}{{\mathcal{R}}}
\nc{\CS}{{\mathcal{S}}}
\nc{\CT}{{\mathcal{T}}}
\nc{\CU}{{\mathcal{U}}}
\nc{\CV}{{\mathcal{V}}}
\nc{\CW}{{\mathcal{W}}}
\nc{\CX}{{\mathcal{X}}}
\nc{\CY}{{\mathcal{Y}}}
\nc{\CZ}{{\mathcal{Z}}}
\nc{\CI}{{\mathcal{I}}}
\nc{\CJ}{{\mathcal{J}}}

\nc{\csM}{{\check{\mathcal A}}{}}
\nc{\oM}{{\overset{\circ}{\mathcal M}}{}}
\nc{\obM}{{\overset{\circ}{\mathbf M}}{}}
\nc{\oCA}{{\overset{\circ}{\mathcal A}}{}}
\nc{\obA}{{\overset{\circ}{\mathbf A}}{}}
\nc{\ooM}{{\overset{\circ}{M}}{}}
\nc{\osM}{{\overset{\circ}{\mathsf M}}{}}
\nc{\vM}{{\overset{\bullet}{\mathcal M}}{}}
\nc{\nM}{{\underset{\bullet}{\mathcal M}}{}}
\nc{\oD}{{\overset{\circ}{\mathcal D}}{}}
\nc{\obD}{{\overset{\circ}{\mathbf D}}{}}
\nc{\oA}{{\overset{\circ}{\mathbb A}}{}}
\nc{\op}{{\overset{\bullet}{\mathbf p}}{}}
\nc{\cp}{{\overset{\circ}{\mathbf p}}{}}
\nc{\oU}{{\overset{\bullet}{\mathcal U}}{}}
\nc{\oZ}{{\overset{\circ}{\mathcal Z}}{}}
\nc{\ofZ}{{\overset{\circ}{\mathfrak Z}}{}}
\nc{\oF}{{\overset{\circ}{\fF}}}

\nc{\fa}{{\mathfrak{a}}}
\nc{\fb}{{\mathfrak{b}}}
\nc{\fd}{{\mathfrak{d}}}
\nc{\ff}{{\mathfrak{f}}}
\nc{\fg}{{\mathfrak{g}}}
\nc{\fgl}{{\mathfrak{gl}}}
\nc{\fh}{{\mathfrak{h}}}
\nc{\fj}{{\mathfrak{j}}}
\nc{\fl}{{\mathfrak{l}}}
\nc{\fm}{{\mathfrak{m}}}
\nc{\fn}{{\mathfrak{n}}}
\nc{\fu}{{\mathfrak{u}}}
\nc{\fp}{{\mathfrak{p}}}
\nc{\fr}{{\mathfrak{r}}}
\nc{\fs}{{\mathfrak{s}}}
\nc{\ft}{{\mathfrak{t}}}
\nc{\fz}{{\mathfrak{z}}}
\nc{\fsl}{{\mathfrak{sl}}}
\nc{\hsl}{{\widehat{\mathfrak{sl}}}}
\nc{\hgl}{{\widehat{\mathfrak{gl}}}}
\nc{\hg}{{\widehat{\mathfrak{g}}}}
\nc{\chg}{{\widehat{\mathfrak{g}}}{}^\vee}
\nc{\hn}{{\widehat{\mathfrak{n}}}}
\nc{\chn}{{\widehat{\mathfrak{n}}}{}^\vee}

\nc{\fA}{{\mathfrak{A}}}
\nc{\fB}{{\mathfrak{B}}}
\nc{\fD}{{\mathfrak{D}}}
\nc{\fE}{{\mathfrak{E}}}
\nc{\fF}{{\mathfrak{F}}}
\nc{\fG}{{\mathfrak{G}}}
\nc{\fK}{{\mathfrak{K}}}
\nc{\fL}{{\mathfrak{L}}}
\nc{\fM}{{\mathfrak{M}}}
\nc{\fN}{{\mathfrak{N}}}
\nc{\fP}{{\mathfrak{P}}}
\nc{\fU}{{\mathfrak{U}}}
\nc{\fV}{{\mathfrak{V}}}
\nc{\fZ}{{\mathfrak{Z}}}

\nc{\ba}{{\mathbf{a}}}
\nc{\bb}{{\mathbf{b}}}
\nc{\bc}{{\mathbf{c}}}
\nc{\bd}{{\mathbf{d}}}
\nc{\bbf}{{\mathbf{f}}}
\nc{\be}{{\mathbf{e}}}
\nc{\bi}{{\mathbf{i}}}
\nc{\bj}{{\mathbf{j}}}
\nc{\bn}{{\mathbf{n}}}
\nc{\bo}{{\mathbf{o}}}
\nc{\bp}{{\mathbf{p}}}
\nc{\bq}{{\mathbf{q}}}
\nc{\bu}{{\mathbf{u}}}
\nc{\bv}{{\mathbf{v}}}
\nc{\bx}{{\mathbf{x}}}
\nc{\bs}{{\mathbf{s}}}
\nc{\by}{{\mathbf{y}}}
\nc{\bw}{{\mathbf{w}}}
\nc{\bA}{{\mathbf{A}}}
\nc{\bK}{{\mathbf{K}}}
\nc{\bB}{{\mathbf{B}}}
\nc{\bF}{{\mathbf{F}}}
\nc{\bC}{{\mathbf{C}}}
\nc{\bG}{{\mathbf{G}}}
\nc{\bD}{{\mathbf{D}}}
\nc{\bE}{{\mathbf{E}}}
\nc{\bH}{{\mathbf{H}}}
\nc{\bI}{{\mathbf{I}}}
\nc{\bM}{{\mathbf{M}}}
\nc{\bN}{{\mathbf{N}}}
\nc{\bO}{{\mathbf{O}}}
\nc{\bV}{{\mathbf{V}}}
\nc{\bW}{{\mathbf{W}}}
\nc{\bX}{{\mathbf{X}}}
\nc{\bZ}{{\mathbf{Z}}}
\nc{\bS}{{\mathbf{S}}}

\nc{\sA}{{\mathsf{A}}}
\nc{\sB}{{\mathsf{B}}}
\nc{\sC}{{\mathsf{C}}}
\nc{\sD}{{\mathsf{D}}}
\nc{\sF}{{\mathsf{F}}}
\nc{\sK}{{\mathsf{K}}}
\nc{\sM}{{\mathsf{M}}}
\nc{\sO}{{\mathsf{O}}}
\nc{\sW}{{\mathsf{W}}}
\nc{\sQ}{{\mathsf{Q}}}
\nc{\sP}{{\mathsf{P}}}
\nc{\sZ}{{\mathsf{Z}}}
\nc{\sr}{{\mathsf{r}}}
\nc{\bk}{{\mathsf{k}}}
\nc{\sg}{{\mathsf{g}}}
\nc{\sff}{{\mathsf{f}}}
\nc{\sfe}{{\mathsf{e}}}
\nc{\sfj}{{\mathsf{j}}}
\nc{\sfb}{{\mathsf{b}}}
\nc{\sfc}{{\mathsf{c}}}
\nc{\sd}{{\mathsf{d}}}
\nc{\sv}{{\mathsf{v}}}

\nc{\BK}{{\bar{K}}}

\nc{\tA}{{\widetilde{\mathbf{A}}}}
\nc{\tB}{{\widetilde{\mathcal{B}}}}
\nc{\tg}{{\widetilde{\mathfrak{g}}}}
\nc{\tG}{{\widetilde{G}}}
\nc{\TM}{{\widetilde{\mathbb{M}}}{}}
\nc{\tO}{{\widetilde{\mathsf{O}}}{}}
\nc{\tU}{{\widetilde{\mathfrak{U}}}{}}
\nc{\TZ}{{\tilde{Z}}}
\nc{\tx}{{\tilde{x}}}
\nc{\tbv}{{\tilde{\bv}}}
\nc{\tfP}{{\widetilde{\mathfrak{P}}}{}}
\nc{\tz}{{\tilde{\zeta}}}
\nc{\tmu}{{\tilde{\mu}}}

\nc{\urho}{\underline{\pi}}
\nc{\uB}{\underline{B}}
\nc{\uC}{{\underline{\mathbb{C}}}}
\nc{\ui}{\underline{i}}
\nc{\uj}{\underline{j}}
\nc{\ofP}{{\overline{\mathfrak{P}}}}
\nc{\oB}{{\overline{\mathcal{B}}}}
\nc{\og}{{\overline{\mathfrak{g}}}}
\nc{\oI}{{\overline{I}}}

\nc{\eps}{\varepsilon}
\nc{\hrho}{{\hat{\pi}}}

\nc{\one}{{\mathbf{1}}}
\nc{\two}{{\mathbf{t}}}

\nc{\Rep}{{\mathop{\operatorname{\rm Rep}}}}
\nc{\Tot}{{\mathop{\operatorname{\rm Tot}}}}
\nc{\Ker}{{\mathop{\operatorname{\rm Ker}}}}
\nc{\Hilb}{{\mathop{\operatorname{\rm Hilb}}}}
\nc{\End}{{\mathop{\operatorname{\rm End}}}}
\nc{\Ext}{{\mathop{\operatorname{\rm Ext}}}}
\nc{\CHom}{{\mathop{\operatorname{{\mathcal{H}}\it om}}}}
\nc{\GL}{{\mathop{\operatorname{\rm GL}}}}
\nc{\gr}{{\mathop{\operatorname{\rm gr}}}}
\nc{\Id}{{\mathop{\operatorname{\rm Id}}}}
\nc{\de}{{\mathop{\operatorname{\rm def}}}}
\nc{\length}{{\mathop{\operatorname{\rm length}}}}
\nc{\supp}{{\mathop{\operatorname{\rm supp}}}}

\nc{\Cliff}{{\mathsf{Cliff}}}
\nc{\Fl}{\on{Fl}}
\nc{\Fib}{{\mathsf{Fib}}}
\nc{\Coh}{{\mathsf{Coh}}}
\nc{\QCoh}{{\on{QCoh}}}
\nc{\IndCoh}{{\on{IndCoh}}}
\nc{\FCoh}{{\mathsf{FCoh}}}

\nc{\reg}{{\text{\rm reg}}}

\nc{\cplus}{{\mathbf{C}_+}}
\nc{\cminus}{{\mathbf{C}_-}}
\nc{\cthree}{{\mathbf{C}_*}}
\nc{\Qbar}{{\bar{Q}}}
\nc\Eis{\on{Eis}}
\nc\Eisb{\ol\Eis{}}
\nc\Eisr{\on{Eis}^{rat}{}}
\nc\wh{\widehat}
\nc{\Def}{\on{Def_{\check{\fb}}(E)}}
\nc{\barZ}{\overline{Z}{}}
\nc{\barbarZ}{\overline{\barZ}{}}
\nc{\barpi}{\overline\iota}
\nc{\barbarpi}{\overline\barpi}
\nc{\barpip}{\overline\iota{}^+}
\nc{\barpim}{\overline\iota{}^-}

\nc{\fq}{\mathfrak q}

\nc{\fqb}{\ol{\fq}{}}
\nc{\fpb}{\ol{\fp}{}}
\nc{\fpr}{{\fp^{rat}}{}}
\nc{\fqr}{{\fq^{rat}}{}}

\nc{\hattimes}{\wh\otimes}

\nc{\bh}{{\bar{h}}}
\nc{\bOmega}{{\overline{\Omega(\check \fn)}}}

\nc{\seq}[1]{\stackrel{#1}{\sim}}

\nc{\cT}{{\check{T}}}
\nc{\cG}{{\check{G}}}
\nc{\cM}{{\check{M}}}
\nc{\cB}{{\check{B}}}
\nc{\cN}{{\check{N}}}

\nc{\ct}{{\check{\mathfrak t}}}
\nc{\cg}{{\check{\fg}}}
\nc{\cb}{{\check{\fb}}}
\nc{\cn}{{\check{\fn}}}

\nc{\cLambda}{{\check\Lambda}}

\nc{\cla}{{\check\kappa_x}}
\nc{\cmu}{{\check\mu}}
\nc{\clambda}{{\check\lambda}}
\nc{\cnu}{{\check\nu}}
\nc{\ceta}{{\check\eta}}

\nc{\DefbE}{{\on{Def}_{\cB}(E_\cT)}}

\nc{\imathb}{{\ol{\imath}}}
\nc{\rlr}{\overset{\longrightarrow}{\underset{\longrightarrow}\longleftarrow}}

\nc{\KG}{K\backslash G}
\nc{\comult}{{co\text{-}mult}}
\nc{\counit}{{co\text{-}unit}}
\nc{\uHom}{{\underline{\Maps}}}
\nc{\dgSch}{\on{Sch}}
\nc{\Sch}{\on{Sch}}
\nc{\affdgSch}{\on{Sch}^{\on{aff}}}
\nc{\affSch}{\on{Sch}^{\on{aff}}}
\nc{\Groupoids}{\on{Grpd}}
\nc{\inftygroup}{\on{Spc}}
\nc{\inftyCat}{\infty\on{-Cat}}
\nc{\StinftyCat}{\inftyCat^{\on{St}}}
\nc{\MoninftyCat}{\infty\on{-Cat}^{\on{Mon}}}
\nc{\SymMoninftyCat}{\infty\on{-Cat}^{\on{SymMon}}}
\nc{\SymMonStinftyCat}{\on{DGCat}^{\on{SymMon}}}
\nc{\MonStinftyCat}{\on{DGCat}^{\on{Mon}}}
\nc{\inftystack}{\on{Stk}}
\nc{\inftystackalg}{Stk^{1\text{-}alg}}
\nc{\inftyprestack}{\on{PreStk}}
\nc{\inftydgnearstack}{\on{NearStk}}
\nc{\inftydgstack}{\on{Stk}}
\nc{\inftydgstackalg}{DGStk^{1\text{-}alg}}
\nc{\inftydgprestack}{\on{PreStk}}
\nc{\dgindSch}{\on{indSch}}
\nc{\indSch}{{}^{\on{cl}}\!\on{indSch}}
\nc{\infSch}{\on{infSch}}
\nc{\dr}{{\on{dR}}}

\nc{\mmod}{{\on{-}\!{\mathbf{mod}}}}

\nc{\starr}{\text{\dh}}
\nc{\Spectra}{\on{Spectra}}
\nc{\Crys}{\on{Crys}}
\nc{\oblv}{{\mathbf{oblv}}}
\nc{\ind}{{\mathbf{ind}}}
\nc{\coind}{{\mathbf{coind}}}
\nc{\inv}{{\mathbf{inv}}}
\nc{\triv}{{\mathbf{triv}}}
\nc{\CMaps}{{\mathcal Maps}}
\nc{\Maps}{\on{Maps}}
\nc{\bMaps}{\mathbf{Maps}}
\nc{\BMaps}{\ul{\on{Maps}}}
\nc{\Grid}{\on{Grid}}
\nc{\hGrid}{\on{Grid}^{\geq\,\on{dgnl}}}
\nc{\Diag}{\on{Diag}}
\nc{\bDelta}{\mathbf{\Delta}}
\nc{\tCateg}{(\infty\on{-2)-Cat}}
\nc{\ul}{\underline}
\nc{\Seg}{\on{Seq}}
\nc{\biSeg}{\on{bi-Seq}}
\nc{\triSeg}{\on{tri-Seq}}
\nc{\quadSeg}{\on{quad-Seq}}
\nc{\nSeg}{\on{n-Seq}}
\nc{\Segm}{\on{Seg}^{\on{mkd}}}
\nc{\fLm}{\fL^{\on{mkd}}}
\nc{\inftyCatm}{\inftyCat^{\on{mkd}}}
\nc{\Blocks}{\mathbf{Blocks}}
\nc{\Snakes}{\mathbf{Snakes}}
\nc{\bifL}{\on{bi-}\!\fL}
\nc{\Sets}{\on{Sets}}
\nc{\Ran}{\on{Ran}}
\nc{\Vect}{\on{Vect}}
\nc{\Shv}{\on{Shv}}
\nc{\unn}{\mathbf{union}}
\nc{\Spc}{\on{Spc}}
\nc{\ppart}{(\!(t)\!)}
\nc{\qqart}{[\![t]\!]}
\nc{\Dmod}{\on{D-mod}}
\nc{\cD}{\mathcal D}
\nc{\ocD}{\overset{\circ}{\cD}}
\nc{\sfp}{\mathsf{p}}
\nc{\sfq}{\mathsf{q}}
\nc{\DGCat}{\on{DGCat}}
\renc{\det}{\on{det}}
\nc{\Conf}{\on{Conf}}
\nc{\Whit}{\on{Whit}}
\nc{\Reg}{\on{Reg}}
\nc{\Res}{\on{Res}}
\nc{\BunNbox}{(\overline\Bun_N^{\omega^\rho})_{\infty\cdot x}} 
\nc{\BunNmbox}{(\overline\Bun_{N^-}^{\omega^\rho})_{\infty\cdot x}}
\nc{\bHecke}{\overset{\bullet}{\on{Hecke}}}
\nc{\Hecke}{\on{Hecke}}
\nc{\bCZ}{\ol\CZ}
\nc{\oCZ}{\overset{\circ}\CZ} 
\nc{\boCZ}{\ol{\oCZ}}
\nc{\sotimes}{\overset{!}\otimes}
\nc{\semiinf}{\on{SI}}
\nc{\coInd}{\on{coInd}}
\nc{\bCM}{\overset{\bullet}\CM{}}
\nc{\bCF}{\overset{\bullet}\CF{}}
\nc{\SI}{\on{SI}}
\nc{\KL}{\on{KL}}

\begin{document}

\author{D. Gaitsgory}

\title[The semi-infinite IC sheaf]{The semi-infinite intersection cohomology sheaf}

\dedicatory{To David Kazhdan, with gratitude and affection} 

\begin{abstract}
We introduce the semi-infinite category of sheaves on the affine Grassmannian,
and construct a particular object in it, which we call the the semi-infinite intersection 
cohomology sheaf. We relate it to several other entities naturally appearing in the
geometric Langlands theory.
\end{abstract} 

\date{\today}

\maketitle

\tableofcontents

\section*{Introduction}

\ssec{Our goals}

\sssec{}

This paper deals with sheaves on infinite-dimensional algebro-geometric objects. Specifically, we
consider the \emph{affine Grassmannian} of a reductive group $G$,
$$\Gr_G:=G\ppart/G\qqart,$$ 
and we consider sheaves that are equivariant with respect to the action of the group $N\ppart$. 
We denote this category by $\SI(\Gr_G)$.
Since $N\ppart$-orbits on $\Gr_G$ are all infinite-dimensional, objects from $\SI(\Gr_G)$ necessarily
have infinite-dimensional support. 

\medskip

We refer the reader to \secref{sss:sheaves intro}, where we explain what we mean by sheaves 
on infinite-dimensional objects such as $\Gr_G$, so that the category of $N\ppart$-equivariant sheaves
makes sense. 
Let us mention, however, that in order to set up such a theory we need to work from the start with the derived category
of sheaves (or, more precisely, its DG model). There is no hope to develop a rich enough theory while staying within an abelian
category (of either sheaves or perverse sheaves). 

\medskip

The word ``semi-infinite" in the title of the paper refers to the fact that 
$N\ppart$-orbits on $\Gr_G$ have also an infinite codimension. We recall that they are parameterized by coweights of $G$:
$$S^\lambda=N\ppart\cdot t^\lambda.$$
The orbit $S^\mu$ lies in the closure $\ol{S}{}{}^\lambda$ of $S^\lambda$ if and only if $\lambda-\mu$ lies in the ``wide cone" $\Lambda^{\on{pos}}$,
i.e., is the sum of simple positive roots with non-negative integral coefficients. 

\sssec{}

The goal of this paper is to construct and describe a particular object in $\SI(\Gr_G)$ which we call the 
\emph{semi-infinite intersection cohomology sheaf}, and denote by $\ICs$. Notionally, $\ICs$ is the
intersection cohomology sheaf of $\ol{S}{}^0$, the closure of the $N\qqart$-orbit $S^0$. In what precise
sense our $\ICs$ really is an intersection cohomology sheaf will be discussed in this introduction.

\medskip

It is obviously not the case the we have randomly decided to construct one particular object in $\SI(\Gr_G)$
and call it $\ICs$. Our construction is motivated (or, rather, necessitated) by the role that $\ICs$ is
supposed to perform in and around the geometric Langlands theory. We will list a certain number
of such roles in \secref{ss:relations}. 

\sssec{}  \label{sss:Lus}

The initial cue for what $\ICs$ should be was provided by \cite{Lus}. Namely, we want 
our semi-infinite intersection cohomology sheaf to be such that its fibers are given by the \emph{periodic}
(or \emph{stable}) affine Kazhdan-Lusztig polynomials. 

\medskip

These Kazhdan-Lusztig polynomials are defined combinatorially, and the challenge was to find a geometry
(i.e., a space and a sheaf on it) that realizes it. In fact, what the space should be is dictated by the initial
setting of Lusztig's: it is the double quotient 
$$N\ppart\backslash G\ppart/G\qqart.$$

So, the sought-for sheaf should be an object of $\SI(\Gr_G)$, and our $\ICs$ is designed so that it has the
desired fibers. Namely, its (!)-restriction to the orbit $S^\mu$ equals the constant (more precisely, dualizing) sheaf
tensored with the cohomologically graded vector space
\begin{equation} \label{e:Lus answer}
\Sym(\cn^-[-2])(\mu),
\end{equation}
where $\cn^-$ is the unipotent radical of the (negative) Borel in the Langlands dual Lie algebra $\cg$. 

\sssec{}

However, !-fibers do not determine a sheaf, and one may want to seek further confirmation as to why
our $\ICs$ is the ``right thing".  Such further confirmation comes from the paper \cite{FFKM}. 

\medskip

Namely, the authors of {\it loc. cit.} considered a certain \emph{finite-dimensional} algebro-geometric
object, namely, the algebraic stack $\BunBb$ (the definition of an object, equivalent to $\BunBb$ from the point of
view of singularities, denoted $\BunNb$, is recalled in \secref{ss:BunNb}). 

\medskip

We choose a projective curve $X$, 
and consider the stacks $\Bun_B$ and $\Bun_G$ that classify $B$-bundles and $G$-bundles on $X$, respectively.
The stack $\BunBb$ is Drinfeld's relative compactification of $\Bun_B$ along the fibers of the projection
$\Bun_B\to \Bun_G$. 

\medskip

The stack $\BunBb$ also has strata numbered by elements of the negative wide cone $-\Lambda^{\on{pos}}$,
and the stratum with index $\mu$, denoted $(\BunBb)_{=\mu}$, is isomorphic to 
$$\Bun_B\times X^\mu,$$
where $X^\mu$ is the spaces of $(-\Lambda^{\on{pos}})$-valued divisors on $X$ of total degree $\mu$.

\medskip 

Let $\ICs_{\on{glob}}$ denote the intersection cohomology sheaf of $\BunBb$. The subscript
``glob" meant to signify that $\ICs_{\on{glob}}$ is global in nature, i.e., its definition involves 
a global curve $X$. Pick a point $x\in X$, and let us identify the completed local ring of $X$ at $x$
with the formal power series ring $k\qqart$ that we used in the definition of $G\qqart,G\ppart$, etc. 

\medskip

The main result of \cite{FFKM} can be interpreted as saying that if we take the !-restriction of 
$\ICs_{\on{glob}}$ to the stratum $(\BunBb)_{=\mu}$ \emph{and the further} !-restriction to
$$\Bun_B\simeq \Bun_B\times \{\mu\cdot x\}\subset \Bun_B\times X^\mu=(\BunBb)_{=\mu},$$
the result 
is isomorphic to the intersection cohomology sheaf of $\Bun_B$ (which is constant, because $\Bun_B$ is smooth)
tensored with \eqref{e:Lus answer}.

\sssec{}

Thus, the geometry of $\BunBb$ does reproduce Lusztig's answer, but with the following caveats:

\medskip

\noindent{(i)} $\BunBb$ is not the same as $N\ppart\backslash \ol{S}{}^0$;

\medskip

\noindent{(ii)} The strata in $\BunBb$ are ``bigger" than just copies of $\Bun_B$: we have all those 
floating points of $X$ that we need to assemble to $x$ in order to get a copy of $\Bun_B$.

\medskip

That said, we do have a map
$$\pi:\ol{S}{}^0\to \BunBb.$$

We prove that our $\ICs$ is isomorphic to the !-pullback of $\ICs_{\on{glob}}$. This provides another
piece of evidence that $\ICs$ that we define is a sensible object. 

\medskip

\begin{rem}
Of course, one could have simply defined $\ICs$ as the !-pullback of  $\ICs_{\on{glob}}$ along the map $\pi$. 
However, the drawback of this approach is that such a definition
is not intrinsically local: for the applications we have in mind we wish to define $\ICs$ purely in terms of $\Gr_G$ and 
the $N\ppart$-action on it. 
\end{rem} 

\sssec{}

Being of infinite dimension and infinite codimension, $N\ppart$-orbits on $\Gr_G$ do not have an intrinsically defined
dimension function. However, one can talk about their
relative dimension. Namely, we set the relative dimension of $S^\mu$ and $S^\lambda$ to be (the negative of) the 
relative dimension of the stabilizer subgroups of points on these orbits.  The latter works out to be
$$\on{dim.rel.}(\on{Ad}_{t^{-\lambda}}(N\qqart),\on{Ad}_{t^{-\mu}}(N\qqart))=\langle \lambda-\mu,2\check\rho\rangle.$$

\medskip

This notion of relative dimension leads to a t-structure on $\SI(\Gr_G)$. The inevitable question that one asks is the following: 
is our $\ICs$ \emph{the} minimal extension of the constant sheaf on $S^0$?  The answer is ``no, but..."

\medskip

It is true that $\ICs$ lies in the heart of this t-structure. However, as a curious feature of our situation (which reflects the fact
that it is substantially infinite-dimensional) is that the minimal extension of the constant sheaf on $S^0$ is \emph{the !-extension}. 

\medskip

So, the ``no" aspect of the answer is that our $\ICs$ is \emph{not} the minimal extension. 
The ``but" aspect is the following: 

\medskip

Instead of the affine Grassmannian $\Gr_G$ (considered as attached to a point of a curve $X$)
we can consider its version over the Ran space $\Ran(X)$ of $X$. Denote the corresponding objects by 
$\Gr_{G,\Ran}, S^0_{\Ran}$, etc. In this situation one can also introduce a t-structure and consider the 
minimal extension of the constant sheaf on $S^0_{\Ran}$; denote it by $\ICs_{\Ran}$.  Then it will be true that our $\ICs$
is the !-restriction of $\ICs_{\Ran}$ to 
$$\Gr_G=\{x\}\underset{\Ran(X)}\times \Gr_{G,\Ran}.$$

This will be performed in the forthcoming paper \cite{Ga1}. 

\medskip

This, we can say that our $\ICs$ is not the intersection cohomology sheaf in any t-structure, as long as we do not
allow the point $x\in X$ to move. However, it is the restriction of the IC sheaf of a more global object in at least 
two different contexts ($\BunBb$ or $\Ran(X)$), in both of which $x$ moves along $X$. A related fact is that in both of 
these models, the codimension of the $\mu$-stratum is $\langle \mu,\check\rho\rangle$, while on $\Gr_G$ itself, this codimension
is twice that amount. 

\ssec{Relation to the work of Bouthier-Kazhdan}

In the paper \cite{BK} another approach to the construction of the semi-infinite IC sheaf is taken. 

\sssec{}

In \cite{BK}, the authors consider the affine closure $\ol{G/N}$ of $G/N$ and the scheme of arcs 
$$\ol{G/N}\qqart.$$

Inside $\ol{G/N}\qqart$ one considers the ``good" open part $\overset{\circ}{\ol{G/N}}\qqart$, whose field valued-points are those maps
$$\Spec(k\qqart)\to \ol{G/N},$$
for which the composite
$$\Spec(k\ppart)\to \Spec(k\qqart)\to \ol{G/N}$$
lands in $G/N$.  

\sssec{}

Note that at the set-theoretic level, the quotients $N\ppart\backslash \ol{S}{}^0$ and $G\qqart\backslash \overset{\circ}{\ol{G/N}}\qqart$
are isomorphic. 

\medskip

The main difference of the two approaches is that our basic geometric object is $\ol{S}{}^0$ (which is the quotient of an
appropriate closed ind-subscheme of $G\ppart$ by $G\qqart$); it is an ind-scheme of \emph{ind-finite type} (i.e., a rising union
of schemes of finite type under closed embeddings), and we consider the category of sheaves on it, which is built 
out of categories of sheaves on schemes of finite type that comprise $\ol{S}{}^0$. 

\medskip 

The authors of \cite{BK} consider $\overset{\circ}{\ol{G/N}}\qqart$, which is a \emph{scheme}, but \emph{of infinite type}.
There are several ways to define the category of sheaves on arbitrary schemes, but in general this category will be 
quite ill-behaved; in particular, one would not be able to define the intersection cohomology sheaf of a subscheme. 

\medskip

However, the authors of \cite{BK} show, that in the case of $\overset{\circ}{\ol{G/N}}\qqart$ (or more generally,
for $\ol{G/N}$ replaced by an affine scheme $Z$ with a smooth open $\overset{\circ}Z\subset Z$), it can be 
approximated by finite-dimensional schemes in the pro-smooth topology so that one obtains 
a reasonably behaved category of sheaves, equipped with a Verdier auto-duality functor and a t-structure. 
In particular, one can consider the intersection cohomology sheaf of the open subscheme 
$$G/N\qqart\subset \overset{\circ}{\ol{G/N}}\qqart.$$

It follows from the results of \cite{BK} that the !-fibers of their intersection cohomology sheaf identify
with those of our $\ICs$. 

\sssec{}

At the moment, the author does not know what is the relationship between the $G\qqart$-equivariant part
of the Bouthier-Kazhdan category and the portion of our $\SI(\Gr_G)$ supported on $\ol{S}{}^0$. 

\medskip

Most probably, they are \emph{not} equivalent; further, it is possible that the Bouthier-Kazhdan category is
equivalent to the category $\Shv(\Fls)$ mentioned in \secref{sss:true semiinf} below.  

\sssec{}

Thus, the Bouthier-Kazhdan approach produces the right intersection cohomology sheaf, and probably
the entire category. However, for the needs of the geometric Langlands theory (see \secref{ss:relations}) 
we still the realization of these objects via the affine Grassmannian, as developed in the present paper. 

\ssec{What is done in this paper?}

\sssec{}

We want to define $\ICs$ so it its !-fibers are given by the stable affine Kashdan-Lusztig polynomials. By definition,
the value of the latter on $\mu\in -\Lambda^{\on{pos}}$ is the direct limit over $\lambda$ (ranging over the
poset of dominant coweights of $G$, see \secref{sss:bizarre order}) of the !-fibers of
$\IC_{\ol\Gr^\lambda_G}$ at $t^{\lambda+\mu}$, with the cohomological shift $[\langle \lambda+\mu,2\check\rho\rangle]$.

\medskip

With this in mind, we let $\ICs$ be the direct limit over $\lambda$ of the objects
$$t^{-\lambda}\cdot \IC_{\ol\Gr^\lambda_G}[\langle \lambda,2\check\rho\rangle],$$
where $g\cdot -$ means translation by an element $g\in G\ppart$. 

\medskip

It is almost immediate to see that $\ICs$ defined in this way is indeed $N\ppart$-equivariant, i.e., is an object of $\SI(\Gr_G)$.

\medskip

One advantage of the above definition of $\ICs$ is that it is easy to see that it 
\emph{factorizes} over the Ran space, see \secref{sss:factorizes}. 

\sssec{}

Next, we consider the map
$$\pi:\ol{S}{}^0\to \BunNb$$
and show that $\ICs$ is canonically isomorphic (up to a cohomological shift) to the !-pullback of the intersection
cohomology sheaf on $\BunNb$. 

\medskip

Among the rest, this isomorphism provides a geometric way to construct an isomorphism between the stable 
fibers of the perverse sheaves $\IC_{\ol\Gr^\lambda_G}$ and the fibers of the intersection cohomology sheaf on $\BunNb$,
something that the authors of \cite{BFGM} (including the author of the present paper) did not see how to do
previously. 

\sssec{}

Third, we reproduce an argument from \cite{Ras} showing that the category $\SI(\Gr_G)$ is canonically equivalent to 
a category of finite-dimensional nature, namely, $\Shv(\Gr)^{I^0}$, where $I^0$ is the unipotent radical of the
Iwahori subgroup.

\medskip 

In terms of this equivalence, our $\ICs$ corresponds to a remarkable object of $\Shv(\Gr)^{I^0}$, known as
the (dual) \emph{baby Verma object}.  This object appeared among the rest in \cite{ABBGM} and \cite{FG}
in the geometric descriptions of the (dual) baby Verma module over the small quantum group 
and the Wakimoto module at the critical level, respectively. 

\ssec{What do we need $\ICs$ for?}  \label{ss:relations}

This subsection plays a motivational role only, and many of the objects mentioned here do not
have proper references. 

\sssec{}  \label{sss:factorizes}

We should say right away that the object of crucial importance is not so much $\ICs$ itself, but rather
is \emph{factorizable} version. Namely, the ind-scheme $\Gr_G$ (resp. the group ind-scheme $N\ppart$), 
when viewed as attached to a point $x\in X$, admits a natural factorization structure, and so the category
$\SI(\Gr_G)$ has a natural structure of factorization category over the Ran space $\Ran(X)$.

\medskip

It is a key feature of $\ICs$ that it carries a canonical structure of \emph{factorization algebra} in 
$\SI(\Gr_G)_{\Ran}$; denote it by $\ICs_{\Ran}$. 

\medskip

Here are some of the roles that $\ICs$  and $\ICs_{\Ran}$ are supposed to perform.

\sssec{}  \label{sss:true semiinf}

It was a dream since the late 1980's, put forth by B.~Feigin and E.~Frenkel, that there should
exist a \emph{category of sheaves/D-modules on the semi-infinite flag space} 
$$\Fls:= G\ppart/N\ppart.$$ 
In the context of D-modules, it was expected that there should exist a functor
of global sections from $\Shv(\Fls)$ to the category of modules over the Kac-Moody algebra
at the critical level. 

\medskip

The problem is that $\Fls$ is ``badly'' infinite-dimensional and cannot be approximated by
finite-dimensional algebro-geometric objects. So at the time, it was not clear how to even define the
category $\Shv(\Fls)$. 

\medskip

However, nowadays one can give \emph{a} definition: namely, the category $\Shv(G\ppart)$
is well-defined, and one can formally take the corresponding equivariant category 
$$\Shv(G\ppart)^{N\ppart},$$
where we consider the $N\ppart$-action on $G\ppart$ by right translations. 

\medskip

For example, the category of spherical (i.e., $G\qqart$-equivariant) objects in $\Shv(G\ppart)^{N\ppart}$
identifies by definition with our category $\SI(\Gr_G)$.

\medskip

The problem is that the resulting category is \emph{not} the desired category, i.e., it has
different Hom spaces for some basic objects, than what is expected from \emph{the}
category $\Shv(\Fls)$. 

\medskip

For example, the category of Iwahori-equivariant objects in 
$\Shv(\Fls)$ is supposed to be equivalent to the category of ind-coherent sheaves
on 
$$(\wt\cCN\underset{\cg}\times \{0\})/\cT,$$
while the category of Iwahori-equivariant objects in $\Shv(G\ppart)^{N\ppart}$ is that on 
$$(\wt\cCN\underset{\cg}\times \wt\cg)/\cG,$$
where $\wt\cCN\to \cg$ and $\wt\cg\to \cg$ are the Spring and 
Grothendieck-Springer resolutions, respectively, for the Langlands dual Lie algebra.  

\medskip

Here is how our object $\ICs$ (or rather, its close relative) is supposed to provide a remedy. Instead of 
just $\Shv(G\ppart)^{N\ppart}$, which is a factorization module category over $\Shv(\Gr_G)^{N\ppart}=\SI(\Gr_G)$,
we consider the category of factorization $\ICs_{\Ran}$-modules in $\Shv(\Gr_G)^{N\ppart}$. 

\medskip

Calculations performed by the author indicate that the above category of factorization modules
may well be the desired category $\Shv(\Fls)$, and that in the case of D-modules, it does admit the desired functor 
of global sections to Kac-Moody modules at the critical level. 

\sssec{}

Another place where the object $\ICs$ appears is the paper \cite{Ga}. There
one considers the Kazhdan-Lusztig category $\KL(G)_{\kappa}$. This is the category of $G\qqart$-integrable
representations of the Kac-Moody Lie algebra attached to $\fg$ at a (negative integral) level $\kappa$.
According to Kazhdan and Lusztig, there is a canonical equivalence between $\KL(G)_{\kappa}$
and the category of modules over the quantum group $U_q(G)$, where $q$ is determined by $\kappa$.

\medskip

We wish to describe the functor 
\begin{equation} \label{e:KL functor}
\KL(G)_{\kappa} \to \KL(T)_{\kappa}
\end{equation}
that makes the following diagram commutative
\begin{equation} \label{e:KL diag}
\CD
\KL(G)_{\kappa}   @>{\sim}>>  U_q(G)\mod \\
@VVV @VVV  \\
\KL(T)_{\kappa}   @>{\sim}>>  U_q(T)\mod,
\endCD
\end{equation} 
where the right vertical arrow is the functor of cohomology with respect to $\fu_q(N^+)$, the positive 
part of the \emph{small} quantum group.

\medskip

Now, a general formalism says that objects of 
$$\Shv(\Gr_G)^{N\ppart\cdot T\qqart}$$
give rise to functors \eqref{e:KL functor} (namely, convolve and take semi-infinite cohomology with respect to $\fn\ppart$).
It turns out that our object $\ICs$ gives rise to the functor that we need for the commutativity of the diagram \eqref{e:KL diag}.

\sssec{}

There is an analogous situation to the one above, where instead of $\KL(G)_{\kappa}$ we consider the
\emph{metaplectic Whittaker category} for the dual group, denoted $\Whit_q(\Gr_{\cG})$. According to the
conjecture of J.~Lurie and the author, we still have an equivalence
$$\Whit_q(\Gr_{\cG})\simeq U_q(G)\mod.$$

We wish to have a functor
\begin{equation} \label{e:Whit functor}
\Whit_q(\Gr_{\cG})\to \Whit_q(\Gr_{\cT})
\end{equation} 
that makes the diagram
$$
\CD
\Whit_q(\Gr_{\cG})  @>{\sim}>>  U_q(G)\mod \\
@VVV @VVV  \\
\Whit_q(\Gr_{\cT})  @>{\sim}>>  U_q(T)\mod,
\endCD
$$
commutative, where the right vertical arrow is the same as in \eqref{e:KL diag}. 

\medskip

Again, any object in the metaplectic $\SI_q(\Gr_{\cG})$ version of $\SI(\Gr_{\cG})$ gives to a functor
in \eqref{e:Whit functor}, and the sought-for functor is given by the appropriate metaplectic version 
$\ICs_q\in \SI_q(\Gr_{\cG})$ of $\ICs$. This is the subject of the forthcoming paper \cite{GLys}. 

\ssec{Background and conventions}

\sssec{}

This paper unavoidably uses higher category theory. It appears in the very definition of our
basic object of study, the semi-infinite category on the affine Grassmannian.

\medskip

Thus, we will assume that the reader is familiar with the basics tenets of the theory. The fundamental
reference is \cite{Lu}, but shorter expositions (or user guides) exist as well, for example, the first
chapter of \cite{GR}. 

\sssec{}

Our algebraic geometry happens over an arbitrary algebraically closed ground field $k$. Our
geometric objects are classical, i.e., non-derived (see, however, \secref{sss:dual geom}). 

\medskip

By a prestack we mean an arbitrary (accessible) functor
\begin{equation} \label{e:prestack}
(\affSch)^{\on{op}}\to \on{Groupoids}
\end{equation}
(we de not need to consider higher groupoids). 

\medskip

We shall say that a prestack $\CY$ is locally of finite type if the corresponding functor \eqref{e:prestack} takes
filtered limits of affine schemes to colimits of groupoids. 
Equivalently, $\CY$ is locally of finite type if it is the left Kan extension of a functor
$$(\affSch_{\on{ft}})^{\on{op}}\to \on{Groupoids}$$
along the fully faithful embedding
$$(\affSch_{\on{ft}})^{\on{op}}\hookrightarrow (\affSch)^{\on{op}}.$$

Still equivalently, $\CY$ is locally of finite type if for any $S\in \affSch$ and $y:S\to \CY$,
the category of factorizations of $y$ as
$$S\to S'\to \CY,$$
where $S'\in \affSch_{\on{ft}}$, is \emph{contractible}. 

\medskip

We let $\on{PreStk}_{\on{lft}}$ denote the category of prestacks locally of finite type. We can identify it with the category 
of all functors
\begin{equation} 
(\affSch_{\on{ft}})^{\on{op}}\to \on{Groupoids}.
\end{equation}

\sssec{}

We let $G$ be a reductive group over $k$. We fix a Borel subgroup $B\subset G$ and the opposite Borel $B^-\subset G$.
Let $N\subset B$ and $N^-\subset B^-$ denote their respective unipotent radicals.

\medskip

Set $T=B\cap B^-$; this is a Cartan subgroup of $G$. We use it to identify the quotients
$$B/N \simeq T \simeq B^-/N^-.$$

\medskip

We let $\Lambda$ denote the coweight lattice of $G$, i.e., the lattice of cocharacters of $T$.
We let $\Lambda^{\on{pos}}\subset \Lambda$ denote the sub-monoid consisting of 
linear combinations of positive simple roots with non-negative integral coefficients. 
We let $\Lambda^+\subset \Lambda$ denote the sub-monoid of \emph{dominant coweights}.

\sssec{}

For an affine scheme $Z$, we let $\fL(Z)$ (resp., $\fL^+(Z)$) denote the scheme (resp., ind-scheme), whose $S$-points for $S=\Spec(A)$
are given by $\Hom(\Spec(A\ppart),Z)$ (resp., $\Hom(\Spec(A\qqart),Z)$).  Thus, we replace the notations
$$Z\ppart\rightsquigarrow \fL(Z), \,\, Z\qqart\rightsquigarrow \fL^+(Z)$$
from earlier in the introduction. 

\medskip

For $Z=G$, we consider the \'etale (equivalently, fppf) quotient $\Gr_G:=\fL(G)/\fL^+(G)$. This is the affine Grassmannian of $G$.
It is known to be an ind-scheme of ind-finite type. 

\sssec{}

While our geometry happens over a field $k$, the representation-theoretic categories
that we study are DG categories over another field, denoted $\sfe$ (assumed algebraically closed
and of characteristic $0$). 

\medskip

All our DG categories are assumed presentable. When considering functors,
we will only consider functors that preserve colimits. We denote the $\infty$-category of DG categories
by $\DGCat$.  It carries a symmetric monoidal structure (i.e., one can consider tensor products
of DG categories). The unit object is the DG category of complexes of $\sfe$-vector spaces,
denoted $\Vect$. 

\medskip

For a pair of objects $c_0,c_1$ in a DG category $\CC$ we will denote by $\CHom_\CC(c_0,c_1)$
the corresponding object of $\Vect$. 

\medskip

We will use the notion of t-structure on a DG category. Given a t-structure on $\CC$, we will denote by
$\CC^{\leq 0}$ the corresponding subcategory of cohomologically connective objects, and by $\CC^{>0}$
its right orthogonal. We let $\CC^\heartsuit$ denote the heart $\CC^{\leq 0}\cap \CC^{\geq 0}$. 

\sssec{}  \label{sss:sheaf theory}

The source of DG categories will be a \emph{sheaf theory}, which is a functor
$$\Shv:(\affSch_{\on{ft}})^{\on{op}}\to \DGCat, \quad S\mapsto \Shv(S).$$

For a morphism of affine schemes $f:S_0\to S_1$, the corresponding functor
$$\Shv(S_1)\to \Shv(S_0)$$
is the !-pullback $f^!$.  

\medskip

The main examples of sheaf theories are:

\medskip

\noindent{(i)} We take $\sfe=\ol\BQ_\ell$ and we take $\Shv(S)$ to be the ind-completion of the
(small) DG category of constructible $\ol\BQ_\ell$-sheaves.

\medskip

\noindent{(ii)} When $k=\BC$ and $\sfe$ arbitrary, we take $\Shv(S)$ to be the ind-completion of the
(small) DG category of constructible $\sfe$-sheaves on $S(\BC)$ in the analytic topology. 

\medskip

\noindent{(iii)} If $k$ has characteristic $0$, we take $\sfe=k$ and we take $\Shv(S)$ to be the DG category
of D-modules on $S$.

\medskip

In examples (i) and (ii), the functor $f^!$ always has a left adjoint, denoted $f_!$. In example (iii) this is not
the case. However, the partially defined left adjoint $f_!$ is defined on holonomic objects.  It is 
defined on the entire category if $f$ is proper. 

\sssec{Sheaves on prestacks}   \label{sss:sheaves intro} 

We apply the procedure of right Kan extension along the embedding
$$(\affSch_{\on{ft}})^{\on{op}}\hookrightarrow (\on{PreStk}_{\on{lft}})^{\on{op}}$$
to the functor $\Shv$, and thus obtain a functor (denoted by the same symbol)
$$\Shv:(\on{PreStk}_{\on{lft}})^{\on{op}}\to \DGCat.$$

By definition, for $\CY\in \on{PreStk}_{\on{lft}}$ we have
\begin{equation} \label{e:shv on prestack}
\Shv(\CY)=\underset{S\in \affSch_{\on{ft}},y:S\to \CY}{\on{lim}}\, \Shv(S),
\end{equation} 
where the transition functors in the formation of the limit are the 
!-pullbacks\footnote{Note that even though the index category (i.e., $(\affSch_{\on{ft}})_{/\CY}$) is ordinary, the above limit
is formed in the $\infty$-category $\DGCat$. This is how $\infty$-categories appear in this paper.}.

\medskip

For a map of prestacks $f:\CY_0\to \CY_1$ we thus have a well-defined pullback functor
$$f^!:\Shv(\CY_1)\to \Shv(\CY_0).$$

\medskip

We denote by $\omega_\CY$ the dualizing sheaf on $\CY$, i.e., the pullback of 
$$\sfe\in \Vect\simeq \Shv(\on{pt})$$
along the tautological map $\CY\to \on{pt}$. 

\sssec{}  \label{sss:dual geom}

This paper is closely related to the geometric Langlands theory, and the geometry of the Langlands dual group
$\cG$ makes it appearance. 

\medskip

By definition, $\cG$ is a reductive group over $\sfe$ and geometric objects constructed out of $\cG$ give
rise to $\sfe$-linear DG categories by considering quasi-coherent (resp., ind-coherent) sheaves on them.

\medskip

However, unlike our geometric objects over $k$, the objects of this dual geometry are often \emph{derived}. 
A typical example is the derived enhancement of the flag variety:
$$\wt\cg\underset{\cg}\times \on{pt}.$$

\ssec{Acknowledgements} This paper is dedicated to David Kazhdan, for his guidance and for being a source of 
inspiration for more than 20 years. 

\medskip

The author learned about the problem of semi-infinite flags from M.~Finkelberg back in 1995. 
I wish to thank him for igniting my interest in the subject.  

\medskip

I also wish to thank Sam Raskin for many stimulating discussions, and very helpful suggestions regarding this
paper. 

\section{The semi-infinite category on the affine Grassmannian}

\ssec{The category of sheaves on the affine Grassmannian}


\sssec{}  \label{sss:sheaves}

Recall that for any prestack $\CY$ locally of finite type, we define the category $\Shv(\CY)$ by \eqref{e:shv on prestack}. 

\medskip

Let now $\CY$ be an ind-scheme of ind-finite type, i.e., a prestack that can be written as
\begin{equation} \label{e:indsch}
\underset{i\in I}{\on{colim}}\, Y_i,
\end{equation}
where $I$ is a filtered index category, $Y_i$'s are schemes of finite type, and for every $(i\to i')\in I$, the corresponding map
$$Y_i \overset{f_{i,i'}}\longrightarrow Y_{i'}$$
is a closed embedding\footnote{In this case each $Y_i$ is a closed subfunctor of $\CY$, and $\CY$ admits a universal 
presentation as in \eqref{e:indsch}, where the index category is taken to be that of \emph{all} closed subfunctors of $\CY$
that are representable by schemes of finite type. This category is automatically filtered, and the initial category $I$ is cofinal
in the universal one.}. 

\medskip

In this case we have
$$\Shv(\CY)\simeq \underset{i\in I}{\on{lim}}\, \Shv(Y_i),$$
where for $(i\to i')\in I$, the corresponding functor $\Shv(Y_{i'})\to \Shv(Y_i)$ is $f_{i,i'}^!$. 

\sssec{} \label{sss:t ind-sch}

If $\CY$ is an ind-scheme, the category $\Shv(\CY)$ has a t-structure. It is uniquely characterized by the
following (equivalent) conditions:

\medskip

\noindent(i) An object $\CF\in \Shv(\CY)$ belongs to $(\Shv(\CY))^{\geq 0}$ if and only if its (!)-pullback
to any $Y_i$ belongs to $(\Shv(Y_i))^{\geq 0}$.

\medskip

\noindent(ii) The subcategory $(\Shv(\CY))^{\leq 0}$ is generated under filtered colimits by direct images
of objects $\CF_i\in (\Shv(Y_i))^{\leq 0}$.

\sssec{}

We apply the above discussion to $\CY=\Gr_G$, and thus obtain a well-defined category $\Shv(\Gr_G)$,
equipped with a t-structure.  

\ssec{Definition of the semi-infinite subcategory}  \label{ss:semiinf cat}

We note that the ind-scheme $\Gr_G$ is acted on by the group ind-scheme $\fL(N)$. We define the  
category 
$$\SI(\Gr_G):=\Shv(\Gr_G)^{\fL(N)}$$ as the full subcategory of objects $\Shv(\Gr_G)$ that are
$\fL(N)$-invariant.

\medskip

By definition, this means the following.

\sssec{}  \label{sss:N alpha}

First off, since $N$ is unipotent, $\fL(N)$ naturally comes from an ind-object in the category of group-schemes
\begin{equation} \label{e:loop N as colim}
\fL(N)\simeq \underset{\alpha\in A}{\on{colim}}\, N_\alpha,
\end{equation} 
where $A$ is a filtered category.

\medskip

We set
\begin{equation} \label{e:intersect alpha} 
\Shv(\Gr_G)^{\fL(N)}:=\underset{\alpha}{\on{lim}}\, \Shv(\Gr_G)^{N_\alpha},
\end{equation} 
where each $\Shv(\Gr_G)^{N_\alpha}$ is a full subcategory of $\Shv(\Gr_G)$ and for $(\alpha\to \alpha')\in A$, we have
$$\Shv(\Gr_G)^{N_{\alpha'}}\subset \Shv(\Gr_G)^{N_\alpha}$$
as full subcategories in $\Shv(\Gr_G)$. Note that the limit in \eqref{e:intersect alpha} amounts to the intersection
$$\underset{\alpha}\cap\, \Shv(\Gr_G)^{N_\alpha}$$
as full subcategories in $\Shv(\Gr_G)$.

\medskip

Let us now explain what the subcategories 
$$\Shv(\Gr_G)^{N_\alpha}\subset \Shv(\Gr_G)$$
are. 

\sssec{}

For a fixed $\alpha$, the prestack $\Gr_G$, when viewed as equipped with an action of $N_\alpha$,
is naturally an ind-object in the category of schemes equipped with an action of $N_\alpha$.

\medskip

I.e., we can represent $\Gr_G$ as 
$$\underset{i\in I}{\on{colim}}\, Y_i$$
where each $Y_i$ is a closed subscheme of $\Gr_G$, stable under the $N_\alpha$-action. 

\medskip

We set
$$\Shv(\Gr_G)^{N_\alpha}:=\underset{i\in I}{\on{lim}}\, \Shv(Y_i)^{N_\alpha},$$
viewed as a full subcategory of 
$$\Shv(\Gr_G)\simeq \underset{i\in I}{\on{lim}}\, \Shv(Y_i).$$

Thus, it remains to explain what we mean by
$$\Shv(Y_i)^{N_\alpha}\subset \Shv(Y_i)$$
for each $\alpha$ and $i$, and why for $(i\to i')$, the corresponding functor 
$$\Shv(Y_{i'}) \overset{f_{i,i'}^!}\longrightarrow \Shv(Y_i)$$
sends $\Shv(Y_{i'})^{N_\alpha}$ to $\Shv(Y_i)^{N_\alpha}$.

\sssec{}

The group-scheme $N_\alpha$, can be naturally written as
$$\underset{\beta\in B_{\alpha,i}}{\on{lim}}\, N_{\alpha,\beta},$$
where:

\medskip

\noindent(i) $B_{\alpha,i}$ is a filtered category;

\smallskip

\noindent(ii) Each $N_{\alpha,\beta}$ is a unipotent algebraic group (of finite type);

\smallskip

\noindent(iii) For every $(\beta\to \beta')\in B_{\alpha,i}$ the corresponding map
$N_{\alpha,\beta'}\to N_{\alpha,\beta}$
is surjective;

\smallskip

\noindent(iv) The action of $N_\alpha$ on $Y_i$ comes from a compatible family of actions of $N_{\alpha,\beta}$'s on $Y_i$.

\medskip

For any $\beta\in B_{\alpha,i}$, we can consider the corresponding equivariant category
$\Shv(Y_i)^{N_{\alpha,\beta}}$. Since $N_{\alpha,\beta}$ is unipotent, the forgetful functor
$$\Shv(Y_i)^{N_{\alpha,\beta}}\to \Shv(Y_i)$$ is fully faithful, and  
for every $(\beta\to \beta')\in B_{\alpha,i}$, we have the inclusion of subcategories
$$\Shv(Y_i)^{N_{\alpha,\beta}}=\Shv(Y_i)^{N_{\alpha,\beta'}}$$
as subcategories of $\Shv(Y_i)$. 

\medskip

We set $\Shv(Y_i)^{N_\alpha}\subset \Shv(Y_i)$ to be $\Shv(Y_i)^{N_{\alpha,\beta}}$ 
for some/any $\beta\in B_{\alpha,i}$. 

\sssec{}

Going back, it is clear that for a map $(i\to i')\in I$, the corresponding functor
$$\Shv(Y_{i'}) \overset{f_{i,i'}^!}\longrightarrow \Shv(Y_i)$$
sends $\Shv(Y_{i'})^{N_\alpha}$ to $\Shv(Y_i)^{N_\alpha}$.

\medskip

It is also clear that for a map $(\alpha\to \alpha')\in A$, we have
$$\Shv(\Gr)^{N_{\alpha'}}\subset \Shv(\Gr)^{N_\alpha}$$
as full subcategories of $\Shv(\Gr)$.

\sssec{} \label{sss:!-av}

Consider the forgetful functor
$$\SI(\Gr_G):=\Shv(\Gr_G)^{\fL(N)}\to \Shv(\Gr_G).$$

As any functor, it admits a partially defined left adjoint, to be denoted 
$\on{Av}^{\fL(N)}_!$.

\medskip

Explicitly, for $\fL(N)$ written as \eqref{e:loop N as colim}, we have
$$\on{Av}^{\fL(N)}_!:=\underset{\alpha}{\on{colim}}\, \on{Av}^{N_\alpha}_!,$$
where $\on{Av}^{N_\alpha}_!$ is the partially defined left adjoint to the forgetful functor
$$\Shv(\Gr_G)^{N_\alpha}\to \Shv(\Gr_G).$$

\sssec{}

Recall (see \secref{sss:t ind-sch}) that the category $\Shv(\Gr_G)$ is equipped with a t-structure. In 
\secref{sss:proof of invis} we will prove: 

\begin{prop} \label{p:invis}
Every $\CF\in \Shv(\Gr)^{\fL(N)}$ is \emph{infinitely connective}, i.e., lies in $(\Shv(\Gr_G))^{-\leq n}$ for every $n$.
\end{prop}

The above proposition says that objects of the category $\SI(\Gr_G)=\Shv(\Gr)^{\fL(N)}$ are invisible from the point of
view of the heart of $\Shv(\Gr_G)$. However, in \secref{ss:t} we will show that $\SI(\Gr_G)$ carries its own
t-structure with meaningful features. 

\ssec{Semi-infinite orbits and standard functors}

In this subsection we begin to investigate the structure of $\SI(\Gr_G)$. 

\sssec{}

For a coweight $\lambda$ let $S^\lambda$ denote the $\fL(N)$-orbit of the point $t^\lambda\in \Gr_G$, 
and let $\ol{S}{}^\lambda$ be its closure. We have
$$S^\mu\subset \ol{S}{}^\lambda\, \Leftrightarrow\, \lambda-\mu \in \Lambda^{\on{pos}}.$$

\medskip

The denote the corresponding maps by
$$S^\lambda \overset{\bj_\lambda}\hookrightarrow \ol{S}{}^\lambda,$$
$$\ol{S}{}^\lambda \overset{\ol\bi_\lambda}\hookrightarrow \Gr_G,$$
$$S^\lambda \overset{\bi_\lambda}\hookrightarrow \Gr_G,$$
so that $\bi_\lambda=\ol\bi_\lambda\circ \bj_\lambda$. 

\medskip

The map $\bj_\lambda$ (resp., $\ol\bi_\lambda$, $\bi_\lambda$) 
is an open (resp., closed, locally closed) embedding of prestacks, i.e., its base change by an affine scheme
yields a map of schemes with the corresponding property. In particular, $S^\lambda$ and $\ol{S}{}^\lambda$ are ind-schemes. 

\sssec{}

We have the corresponding functors
\begin{equation} \label{e:standard functors !}
(\ol\bi_\lambda)^!: \Shv(\Gr_G)\to \Shv(\ol{S}{}^\lambda),\quad (\bj_\lambda)^!: \Shv(\ol{S}{}^\lambda)\to \Shv(S^\lambda),\quad 
(\bi_\lambda)^!: \Shv(\ol{S}{}^\lambda)\to \Shv(S^\lambda)
\end{equation} 
and 
\begin{equation} \label{e:standard functors *}
(\ol\bi_\lambda)_*: \Shv(\ol{S}{}^\lambda)\to \Shv(\Gr_G),\quad (\bj_\lambda)_*:\Shv(S^\lambda)\to  \Shv(\ol{S}{}^\lambda),\quad 
(\bi_\lambda)_*: \Shv(S^\lambda)\to \Shv(\ol{S}{}^\lambda).
\end{equation} 

\medskip

Since the map $\bj_\lambda$ is an open embedding we will also use the notation $(\bj_\lambda)^*:=(\bj_\lambda)^!$. The functors
$((\bj_\lambda)^*,(\bj_\lambda)_*)$ form an adjoint pair with $(\bj_\lambda)_*$ being fully faithful. 

\medskip

Since the map $\ol\bi_\lambda$ is a closed embedding we will also use the notation $(\ol\bi_\lambda)_!:=(\ol\bi_\lambda)_*$; it is
fully faithful. The functors $((\ol\bi_\lambda)_!,(\ol\bi_\lambda)^!)$ form an adjoint pair. 

\medskip

In particular, the functor $(\bi_\lambda)_*\simeq (\ol\bi_\lambda)_*\circ (\bj_\lambda)_*$ is also fully faithful. 

\sssec{}

The indschemes $S^\lambda$ and $\ol{S}{}^\lambda$ are acted on by $\fL(N)$, 
and the construction in \secref{ss:semiinf cat} applies to them as well. We denote
$$\SI(\Gr_G)_{\leq \lambda}:=\Shv(\ol{S}{}^\lambda)^{\fL(N)}$$
and
$$\SI(\Gr_G)_{=\lambda}:=\Shv(S^\lambda)^{\fL(N)}.$$

\medskip

The functors of \eqref{e:standard functors !} and \eqref{e:standard functors *} send the corresponding subcategories
$$\SI(\Gr_G)_{\leq \lambda}\subset \Shv(\ol{S}{}^\lambda), \quad \SI(\Gr_G)_{=\lambda}\subset \Shv(S^\lambda), \quad
\SI(\Gr_G)\subset \Shv(\Gr_G)$$
to one another and the same adjunctions hold. 

\sssec{}

We claim:

\begin{prop} \label{p:on orbit}
The functor of taking the !-fiber at $t^\lambda$ defines an equivalence
$$\SI(\Gr_G)_{=\lambda}\simeq\Vect.$$
\end{prop}  

\begin{proof}

Let us write $\fL(N)$ as in \eqref{e:loop N as colim}. For each $\alpha$, set the index category $I$ to be equal to $A_{\alpha/}$,
and for $\alpha'\in A_{\alpha/}$ set 
$$Y_{\alpha'}=N_{\alpha'}\cdot t^\lambda.$$

Then $\Shv(S^\lambda)^{\fL(N)}$ is the limit
$$\underset{\alpha\to \alpha'}{\on{lim}}\, \Shv(Y_{\alpha'})^{N_\alpha}.$$

But cofinal in the above index category is the subcategory where $\alpha'=\alpha$. Thus, the above limit
maps isomorphically to
$$\underset{\alpha}{\on{lim}}\, \Shv(Y_{\alpha})^{N_\alpha}.$$

Now, for every $\alpha$, the functor of taking the fiber at $t^\lambda$ defines an equivalence
$$\Shv(Y_{\alpha})^{N_\alpha}\to \Vect;$$
since $\on{Stab}_{N_\alpha}(t^\lambda)$ is unipotent.

\end{proof} 

\begin{cor} \label{c:omega}
The inverse equivalence to that of \propref{p:on orbit} sends $\sfe\in \Vect$ to the dualizing object $\omega_{S^\lambda}\in \Shv(S^\lambda)$. 
\end{cor}

\begin{proof}
It is clear that $\omega_{S^\lambda}$ belongs to the subcategory $\Shv(S^\lambda)^{\fL(N)}$. Now the assertion
follows from the fact that its !-fiber at $t^\lambda$ identifies with $\sfe$. 
\end{proof}

\ssec{Structure of the semi-infinite category}

In this subsection we will show that the functors introduced in the preceding section admit all the expected
adjoints. 

\sssec{}

We claim:

\begin{prop} \label{p:! exists}
The partially defined left adjoint $(\bj_\lambda)_!$ of 
$(\bj_\lambda)^!:\Shv(\ol{S}{}^\lambda)\to \Shv(S^\lambda)$
is defined on the full subcategory $\SI(\Gr_G)_{=\lambda}\subset S^\lambda$
and takes values in $\SI(\Gr_G)_{\leq \lambda}\subset \ol{S}{}^\lambda$. 
Moreover, \begin{equation} \label{e:Av Delta}
(\bi_\lambda)_!(\omega_{S^\lambda})\simeq \on{Av}^{\fL(N)}_!(t^\lambda).
\end{equation} 
\end{prop}

\begin{proof}

By \corref{c:omega}, to prove the existence of $(\bj_\lambda)_!$ it is enough to show that $(\bj_\lambda)_!$
is defined on the object $\omega_{S^\lambda}$, and that the result is $N_\alpha$-equivariant for any $\alpha$.

\medskip

Let $\overset{\circ}\CY\overset{\bj}\longrightarrow \CY$ be an arbitrary open embedding of ind-schemes. 
Write $\CY$ as in \eqref{e:indsch}, and set $\overset{\circ}Y_i:=\overset{\circ}\CY\underset{\CY}\times Y_i$
$$\overset{\circ}{Y}_i \overset{\bj_i}\longrightarrow Y_i.$$

Then 
$$\omega_{\overset{\circ}\CY}\simeq \underset{i}{\on{colim}}\, \omega_{\overset{\circ}{Y}_i}$$
where we think of $\omega_{\overset{\circ}{Y}_i}$ as an object of $\Shv(\overset{\circ}{\CY})$ via the fully faithful embedding 
$\Shv(\overset{\circ}{Y}_i)\hookrightarrow \Shv(\overset{\circ}{\CY})$, and 
$$(\bj_!)(\omega_{\overset{\circ}\CY})\simeq \underset{i}{\on{colim}}\, (\bj_i)_!(\omega_{\overset{\circ}{Y}_i}).$$

\medskip

This implies the existence of $(\bj_\lambda)_!(\omega_{S^\lambda})$. The $N_\alpha$-equivariance
follows from the above description: for a given $\alpha$, take $Y_i$ to be $N_\alpha$-invariant subschemes of $\ol{S}{}^\lambda$. 

\medskip

The isomorphism \eqref{e:Av Delta} follows from the fact that $\CHom$ from either side to an object of $\SI(\Gr_G)$ amounts to taking the !-fiber of
that object at $t^\lambda$. 

\end{proof}

\begin{lem}
The functor $(\bj_\lambda)_!:\SI(\Gr_G)_{=\lambda}\to \SI(\Gr_G)_{\leq \lambda}$ is fully faithful.
\end{lem}

\begin{proof}

Since the right adjoint of $(\bj_\lambda)^!\simeq (\bj_\lambda)^*$, i.e., $(\bj_\lambda)_*$, is fully faithful, 
it formally follows that so is the left adjoint, i.e., $(\bj_\lambda)_!$. 

\end{proof} 

\begin{cor}
The functor $(\bi_\lambda)^!:\SI(\Gr_G)\to \SI(\Gr_G)_{=\lambda}$ admits a 
\emph{left} adjoint, to be denoted $(\bi_\lambda)_!$. The functor $(\bi_\lambda)_!$
is fully faithful. 
\end{cor}

\begin{proof}
The left adjoint in question is given by $(\bi_\lambda)_!:=(\ol\bi_\lambda)_!\circ (\bj_\lambda)_!$. 
\end{proof}

We also note:

\begin{lem} \label{l:invis}
The objects $(\bi_\lambda)_!(\omega_{S^\lambda})$ are infinitely connective in the t-structure on $\Shv(\Gr_G)$.
\end{lem}

\begin{proof}

Write $S^\lambda$ as the colimit of $N_{\alpha}\cdot t^\lambda$, as in the proof of \propref{p:on orbit}. Then 
$(\bi_\lambda)_!(\omega_{S^\lambda})$ is the colimit of !-extensions of the objects $\omega_{N_{\alpha}\cdot t^\lambda}$
under the locally closed embeddings $N_{\alpha}\cdot t^\lambda\hookrightarrow \Gr_G$. 

\medskip

To prove the lemma it suffices to show
that for every $n$, there exists an index $\alpha$ such that for all $\alpha\to \alpha'$, the object
$\omega_{N_{\alpha'}\cdot t^\lambda}\in \Shv(N_{\alpha'}\cdot t^\lambda)$ belongs to $(\Shv(N_{\alpha'}\cdot t^\lambda))^{\leq -n}$.

\medskip

However, this does indeed happen as soon as $\dim(N_{\alpha}\cdot t^\lambda)\geq n$.

\end{proof}

\sssec{}     \label{sss:proof of invis}

For $\lambda\in \Lambda$, let $\Delta^\lambda$ and $\nabla^\lambda$ be the objects of $\SI(\Gr_G)$
equal to
$$(\bi_\lambda)_!(\sfe)[-\langle \lambda,2\check\rho\rangle] \text{ and } (\bi_\lambda)_*(\sfe)[-\langle \lambda,2\check\rho\rangle],$$
respectively, where we think of $\sfe$ as an object of $\SI(\Gr_G)_{=\lambda}$ via the equivalence of \propref{p:on orbit}.
The cohomological shift by $[-\langle \lambda,2\check\rho\rangle]$ in our normalization will be explained later. 

\medskip

By construction, the objects $\Delta^\lambda$ are compact in $\SI(\Gr_G)$ (but of course not in $\Shv(\Gr_G)$). 

\begin{lem} \label{l:! gener}
The category $\SI(\Gr_G)$ is generated by the objects $\Delta^\lambda$. 
\end{lem}

\begin{proof}
The is equivalent to the fact that for $\CF\in \SI(\Gr_G)$,
$$\CHom(\Delta^\lambda,\CF)=0,\,\forall \lambda\, \Rightarrow \CF=0.$$

However, 
$$\CHom(\Delta^\lambda,\CF)=0 \Leftrightarrow (\bi_\lambda)^!(\CF)=0.$$

Since for every closed subscheme $Y\subset \Gr_G$, the intersections $Y\cap S^\lambda$ define a decomposition
into locally closed subsets, we obtain that 
$$(\bi_\lambda)^!(\CF)=0,\,\forall \lambda\, \Rightarrow \CF=0.$$

\end{proof} 

By \lemref{l:invis}, the objects $\Delta^\lambda$ are infinitely connective in the t-structure on $\Shv(\Gr_G)$.
Combined with \lemref{l:! gener}, this implies \propref{p:invis}.

\sssec{}

We claim:

\begin{lem}
For every $\lambda\in \Lambda$, the partially defined left adjoint $(\bi_\lambda)^*$ of
$(\bi_\lambda)_*:\Shv(S^\lambda)\to \Shv(\Gr_G)$ is defined on the full subcategory $\SI(\Gr_G)\subset \Shv(\Gr_G)$,
and takes values in $\SI(\Gr_G)_{=\lambda}\subset \Shv(S^\lambda)$.
\end{lem}

\begin{proof}

By \lemref{l:! gener}, it is enough to show that $(\bi_\lambda)^*$ is defined on every $\Delta^{\lambda'}$, and the result
belongs to $\SI(\Gr_G)_{=\lambda}$. However, from the explicit description in the proof of \propref{p:! exists} it follows
that $(\bi_\lambda)^*(\Delta^{\lambda'})$ equals $\omega_{S^\lambda}[-\langle \lambda,2\check\rho\rangle]$
for $\lambda'=\lambda$, and $0$ otherwise. 

\end{proof}

By definition, we obtain:

\begin{lem} 
An object $\CF\in \SI(\Gr_G)$ is compact if and only if there exist at most finitely many elements
$\lambda\in \Lambda$ such that $(\bi_\lambda)^*(\CF)\neq 0$, and for every such $\lambda$, 
the corresponding object of $\SI(\Gr_G)_{=\lambda}$ is compact\footnote{I.e., under the equivalence
with $\Vect$ corresponds to a complex with finitely many cohomologies such that each of these 
cohomologies is finite-dimensional}.
\end{lem}

As we shall see in the sequel, the objects $\nabla^\lambda$ are \emph{not} compact. 

\ssec{The t-structure on the semi-infinite category}  \label{ss:t}

\sssec{} \label{sss:t}

We define the t-structure on $\SI(\Gr_G)$ in either of the following (tautologically equivalent) ways:

\medskip

\noindent(i) The subcategory $(\SI(\Gr_G))^{\leq 0}$ is generated under filtered colimits by the objects $\Delta^\lambda[n]$, $n\geq 0$.

\medskip

\noindent(ii) An object $\CF\in \SI(\Gr_G)$ belongs to $(\SI(\Gr_G))^{>0}$ if and only if $\Maps(\Delta^\lambda[n],\CF)=*$ for $n\geq 0$.

\sssec{}

Define a t-structure on 
$$\SI(\Gr_G)_{=\lambda}\simeq \Vect$$
to be the shift of the standard t-structure on $\Vect$ by $[-\langle \lambda,2\check\rho\rangle]$, i.e., the object
$\sfe[-\langle \lambda,2\check\rho\rangle]$ is in the heart.

\medskip

Then the above t-structure on $\SI(\Gr_G)$ can also be characterized by:

\medskip

\noindent(ii') An object $\CF\in \SI(\Gr_G)$ belongs to $(\SI(\Gr_G))^{\geq 0}$ if and only if $(\bi_\lambda)^!(\CF)\in \SI(\Gr_G)_{=\lambda}$
belongs to $(\SI(\Gr_G)_{=\lambda})^{\geq 0}$ for all $\lambda$. 

\medskip

\noindent(iii) An object $\CF\in \SI(\Gr_G)$, which is supported on $\ol{S}^\mu$ for some $\mu$, 
belongs to $(\SI(\Gr_G))^{\leq 0}$ if and only if $(\bi_\lambda)^*(\CF)\in \SI(\Gr_G)_{=\lambda}$
belongs to $(\SI(\Gr_G)_{=\lambda})^{\leq 0}$ for all $\lambda$.

\begin{rem}
The above t-structure on $\SI(\Gr_G)$ (and in particular, the corresponding truncations functors) exists for abstract reasons.
Ultimately, all the objects involved are colimits of compact objects of $\Shv(\Gr_G)$, i.e., compact objects in $\Shv(Y_i)$,
where $Y_i$ are closed subschemes of $\Gr_G$. 
However, one should not delude oneself by thinking that the above
t-structure can be modeled by a finite-dimensional situation, as is manifested by \thmref{t:j is perv} below. 
\end{rem} 

\sssec{}

In \secref{ss:!-fibers via global} we will prove:

\begin{thm}  \label{t:j is perv}  For $\lambda\neq 0$, the object 
$(\bi_{\lambda})^!(\Delta^0)$ belongs to $(\SI(\Gr_G)_{\lambda})^{>0}$.
\end{thm}

The meaning of \thmref{t:j is perv} is that $\Delta^0$ belongs to $(\SI(\Gr_G))^\heartsuit$
and is \emph{the intermediate extension} of $\omega_{S^0}$, the latter being the unique irreducible object in $\SI(\Gr_G)_{=0}$. 
In particular, we obtain:

\begin{cor} \label{c:j is perv} 
The object $\Delta^0\in (\SI(\Gr_G))^\heartsuit$ is irreducible.
\end{cor}

In \secref{ss:proof of ab} we will prove:

\begin{prop} \label{p:abelian categ}
The abelian category $(\SI(\Gr_G))^\heartsuit$ is equivalent to $\Rep(\cB^-)$,
where $\cB^-$ is the (negative) Borel in the Langlands dual $\cG$ of $G$. Under this equivalence,
the object $\Delta^\lambda$ correspond to the 1-dimensional representation of $\cB^-$ with
character $w_0(\lambda)$. 
\end{prop}

\sssec{}

However, $\Delta^0$ is \emph{not} the object that we are after. In the next section we will construct 
another object 
$$\ICs\in (\SI(\Gr_G))^\heartsuit$$
that we will call the \emph{semi-infinite IC sheaf}, and study its properties.  We will relate $\ICs$ to two
other naturally appearing objects:

\medskip

\noindent(i) In \secref{s:Drinf} we will relate $\ICs$ to the IC sheaf of Drinfeld's compactification. 

\medskip

\noindent(ii) In \secref{s:baby Verma} we will express it as the (dual) baby Verma object.

\medskip

We will show:

\begin{prop} \label{p:semiinf as reg}
In terms of the equivalence of \propref{p:abelian categ}, the object $\ICs$ corresponds to $\CO(\cB^-/\cT)$.
\end{prop}

\begin{rem}
One may wonder whether the object $\ICs$ that we will construct can be expressed as the intermediate
expression in some t-structure.  This is indeed possible via the following procedure:

\medskip

Instead of the usual affine Grassmannian we consider its version over the Ran space $\Ran(X)$ of
a curve $X$; denote it $\Gr_{G,\Ran}$.  It is equipped with an action of the Ran version of the
loop group $\fL(G)_{\Ran}$, and in particular $\fL(N)_{\Ran}$. We define the category
$$\SI(\Gr_{G,\Ran}):=(\Shv(\Gr_{G,\Ran}))^{\fL(N)_{\Ran}},$$
and it can be equipped with a t-structure.  Let $S^0_{\Ran}\subset \Gr_{G,\Ran}$ be the locally closed subfunctor
equal to the orbit of unit section with respect to $\fL(N)_{\Ran}$. Let $\ICs_{\Ran}$ be the intermediate 
extension of the dualizing sheaf on $S^0_{\Ran}$. One can show (and this will be done in \cite{Ga1})
that our $\ICs$ identifies with the !-restriction of $\ICs_{\Ran}$ to
$$\Gr_G=\{x\}\underset{\Ran(X)}\times \Gr_{G,\Ran}.$$

\end{rem}

\section{Construction of the semi-infinite IC sheaf as a colimit}  \label{s:colimit}

\ssec{The Langlands dual group}

In this subsection we will fix some conventions pertaining to the Lannglands dual of $G$,
to be used in this section and later in the paper. 

\sssec{}

Let $\cG$ be the Langlands dual group of $G$.  We regard as equipped with chosen Borel and Cartan subgroups:
$$
\cT\subset \cB.$$

In particular, we have a well-defined negative Borel $\cB^-$, so that $\cT=\cB\cap \cB^-$. 

\medskip

We regard $\cT$ as a common torus quotient of $\cB$ and $\cB^-$ and use characters of $\cT$ to produce characters
of each. Note, however, that a character that is dominant from the point of view of $\cB$ is anti-dominant from
the point of view of $\cB^-$. 


\medskip

We regard $\Rep(\cG)$ the \emph{abelian} symmetric monoidal category of $\cG$-representations. 
For $V\in \Rep(\cG)$ and $\nu\in \Lambda$ (note that $\Lambda$ is the weight lattice of $\cT$)
we let $V(\nu)$ denote the corresponding weight space. 

\sssec{}

We fix representatives of irreducible highest weight representations $V^\lambda$, $\lambda\in \Lambda^+$ equipped
with \emph{trivializations} of highest weight lines $V^\lambda(\lambda)=(V^\lambda)^{\cN}$, i.e., with
chosen highest weight vectors $$v^\lambda\in V^\lambda(\lambda).$$ Let
$(v^\lambda)^*$ denote the functional $V^\lambda\to \sfe$ such that $\langle (v^\lambda)^*,v^\lambda\rangle=1$
and $(v^\lambda)^*|_{V(\mu)}=0$ for $\mu\neq \lambda$. We can regard $(v^\lambda)^*$ as a highest weight vector
with respect to $\cB^-$ in $(V^\lambda)^*$.  

\sssec{}

In what follows let $\sfe^\lambda$ denote the 1-dimensional representation of $\cT$ on $\sfe$, given by the character $\lambda$.

\medskip

When we think of $V^\lambda$ as endowed with the distinguished highest weight vector $v^\lambda$, we identify it with
$$\on{Ind}^\cG_\cB(\be^\lambda),$$ where $\be^\lambda$ is regarded as a representation of $\cB$ via $\cB\to \cT$, and
$$\on{Ind}^\cG_\cB:\Rep(\cB)\to \Rep(\cG)$$
is the functor \emph{left adjoint} to the restriction functor $\Res^\cG_\cB$. 

\medskip

When we think of $V^\lambda$ as endowed with the distinguished covector $(v^\lambda)^*$, we identify it with
$$\coInd^\cG_{\cB^-}(\be^\lambda),$$ where $\be^\lambda$ is regarded as a representation of $\cB^-$ via $\cB^-\to \cT$, and
$$\coInd^\cG_{\cB^-}:\Rep(\cB^-)\to \Rep(\cG)$$
is the functor \emph{right adjoint} to the restriction functor $\Res^\cG_{\cB^-}$. 

\sssec{}

For each pair $\lambda_1,\lambda_2\in \Lambda$, we have a \emph{canonical} map
\begin{equation} \label{e:Plucker map}
V^{\lambda_1}\otimes V^{\lambda_2}\to V^{\lambda_1+\lambda_2},
\end{equation} 
determined by the condition that the diagram
$$
\CD
V^{\lambda_1}\otimes V^{\lambda_2}  @>>> V^{\lambda_1+\lambda_2} \\
@V{(v^{\lambda_1)^*}\otimes (v^{\lambda_2})^*}VV   @VV{(v^{\lambda_1+\lambda_2})^*}V   \\
\sfe\otimes \sfe  @>{=}>>  \sfe
\endCD
$$
commutes. 

\medskip

We also have a canonical map 
\begin{equation} \label{e:Plucker map dual}
V^{\lambda_1+\lambda_2}\to V^{\lambda_1}\otimes V^{\lambda_2},
\end{equation} 
determined by the condition that the diagram
$$
\CD
V^{\lambda_1+\lambda_2} @>>> V^{\lambda_1}\otimes V^{\lambda_2}  \\
@A{v^{\lambda_1+\lambda_2}}AA  @AA{v^{\lambda_1}\otimes v^{\lambda_2}}A     \\
\sfe @>{=}>>   \sfe\otimes \sfe 
\endCD
$$
commutes.

\sssec{}

For future reference we note that for a fixed finite-dimensional $V\in \Rep(\cG)$ and $\lambda$ deep enough
in the dominant cone (i.e., $\lambda\in \lambda_0+\Lambda^+$ for some $\lambda_0$) there exists a \emph{canonical}
isomorphism in $\Rep(\cG)$
\begin{equation} \label{e:razval initial}
V^\lambda\otimes V\simeq \underset{\mu\in \Lambda}\oplus\, V^{\lambda+\mu} \otimes V(\mu).
\end{equation}

It is uniquely characterized by the property that for a given $v\in V(\mu)$, the image of
$$v^\lambda\otimes v\in V^\lambda(\lambda)\otimes V(\mu) \subset V^\lambda\otimes V$$
lies in
$$\underset{\mu'\geq \mu}\oplus\, V^{\lambda+\mu'} \otimes V(\mu')\subset \underset{\mu\in \Lambda}\oplus\, V^{\lambda+\mu} \otimes V(\mu),$$
and its projection onto the $V^{\lambda+\mu} \otimes V(\mu)$ factor equals $v^{\lambda+\mu}\otimes v$. 

\ssec{Recollections on geometric Satake}

\sssec{} 

We consider $\Sph(G):=\Shv(\Gr_G)^{\fL^+(G)}$ as a monoidal category with respect to convolution. As such it acts
on $\Shv(\Gr_G)$ by right convolutions; denote this action by
$$\CF\in \Shv(\Gr_G),\,\, \CF'\in \Sph(G)\mapsto \CF\star \CF'.$$

\sssec{} \label{sss:Satake}

We will regard Geometric Satake as a monoidal functor
$$\Sat:\Rep(\cG)\to \Sph(G).$$

We will use the following of its properties:

\medskip

\noindent(i) For $\lambda\in \Lambda^+$, we have $\Sat(V^\lambda)=\IC_{\ol\Gr^\lambda_G}$,
where $\Gr^\lambda_G=\fL^+(G)\cdot t^\lambda$, and $\ol\Gr^\lambda_G$ is its closure.

\medskip

\noindent(ii) For $V\in \Rep(\cG)$ and $\mu\in \Lambda$, the weight space $V(\mu)$ identifies canonically
with
$$H_c(S^\mu,(\bi_\mu)^*(\Sat(V)))[\langle \mu,2\check\rho\rangle];$$
in particular, the above cohomology sits in a single degree equal to $\langle \mu,2\check\rho\rangle$. 

\medskip

\noindent(iii) For $\lambda\in \Lambda^+$ the trivialization of the line $V^\lambda(\lambda)$ corresponds
to the identification
$$H_c(S^\lambda,(\bi_\lambda)^*(\Sat(V^\lambda)))[\langle \lambda,2\check\rho\rangle]
\simeq H_c(\fL^+(N)\cdot t^\lambda,\sfe)[2\langle \lambda,2\check\rho\rangle]\simeq \sfe,$$
where we are using the fact that
$$S^\lambda\cap \ol\Gr^\lambda_G=\fL^+(N)\cdot t^\lambda\simeq \fL^+(N)/\on{Ad}_{t^\lambda}(\fL^+(G)),$$
and $\dim(\fL^+(N)/\on{Ad}_{t^\lambda}(\fL^+(G)))=\langle\lambda,2\check\rho\rangle$. 

\sssec{}

Denote $$S^{-,\mu}:=\fL(N^-)\cdot t^\mu\overset{\bi_{-,\mu}}\hookrightarrow \Gr_G.$$
Let $\bk_\mu$ (resp., $\bk_{-,\mu}$) denote the embeddings of the point
$$S^\mu \overset{\bk_\mu}\hookleftarrow \on{pt} \overset{\bk_{-,\mu}}\hookrightarrow S^{-,\mu}.$$

\medskip

We have the following assertion, which is a hallmark application on Braden's theorem from \cite{Bra}:

\begin{lem}  \label{l:Braden}  \hfill

\medskip

\noindent{\em(a)} For $\CF\in \Shv(S^\mu)$ equivariant with respect to the action of $T$ we have canonical isomorphisms
$$H_c(S^\mu,\CF)\simeq (\bk_\mu)^!(\CF),\quad H(S^\mu,\CF)\simeq (\bk_\mu)^*(\CF),$$
and similarly for $(S^{-,\mu},\bk_{-,\mu})$.  

\medskip

\noindent{\em(b)} For $\CF\in \Shv(\Gr_G)$ equivariant with respect to the action of $T$, we have a canonical isomorphism 
$$(\bk_\mu)^!\circ (\bi_\mu)^*(\CF)\simeq  (\bk_{-,\mu})^*\circ (\bi_{-,\mu})^!(\CF)$$
\end{lem}

We will apply the isomorphism of \lemref{l:Braden} to objects of the form $\Sat(V)$ and thus
obtain several alternative expressions for $V(\mu)$. 

\ssec{Definition of the semi-infinite IC sheaf}

From now on we will assume that $G$ is semi-simple and simply 
connected\footnote{This is not a substantial restriction: for a general reductive $G$, the semi-infinite IC
sheaf we are about to define is supported on the neutral connected component of $\Gr_G$, while the latter 
is the same for $G$ and the simply-connected cover of its derived group.}.

\sssec{}  \label{sss:bizarre order}

Consider the set $\Lambda^+$, endowed the above set with the following (non-standard!) order relation
$$\lambda_1\leq \lambda_2\, \Leftrightarrow\, \lambda_2-\lambda_1\in \Lambda^+.$$

%

Note that this poset if filtered (due to the assumption that $G$ is semi-simple and simply connected). 
We will view $(\Lambda^+,\leq)$ as a category. 

\sssec{}

We are going to construct a functor
$$
(\Lambda^+,\leq) \to \Shv(\Gr_G)$$
that on the level of objects sends
$$\lambda\in \Lambda^+ \mapsto  t^{-\lambda}\cdot \Sat(V^\lambda)[\langle \lambda,2\check\rho\rangle].$$

\sssec{}

For an individual pair $\lambda_1\leq \Lambda_2$, the corresponding map
\begin{equation} \label{e:transition map}
t^{-\lambda_1}\cdot \Sat(V^{\lambda_1})[\langle \lambda_1,2\check\rho\rangle] \to 
t^{-\lambda_2}\cdot \Sat(V^{\lambda_2})[\langle \lambda_2,2\check\rho\rangle] 
\end{equation}
is defined as follows.

\medskip

Write $\lambda_2=\lambda_1+\lambda$ with $\lambda\in \Lambda^+$. We claim that we have a canonically defined map
\begin{equation} \label{e:transition map initial}
\delta_{1,\Gr_G}\to t^{-\lambda}\cdot \Sat(V^\lambda)[\langle \lambda,2\check\rho\rangle].
\end{equation}

Indeed, the datum of a map \eqref{e:transition map initial} is equivalent to that of a map
\begin{equation} \label{e:transition map initial 1}
\delta_{t^\lambda,\Gr_G}\to  \Sat(V^\lambda)[\langle \lambda,2\check\rho\rangle],
\end{equation}
and the latter amounts to a vector in the !-fiber of $\Sat(V^\lambda)[\langle \lambda,2\check\rho\rangle]$ at $t^\lambda$.

\medskip

However, this !-fiber of $\Sat(V^\lambda)$ identifies canonically with $\sfe$, as $t^\lambda$ belongs to
the smooth locus $\Gr^\lambda_G\subset \ol\Gr^\lambda_G$ and $\langle \lambda,2\check\rho\rangle=\dim(\Gr^\lambda_G)$. 

\medskip

We now let \eqref{e:transition map} be the map
\begin{multline*} 
t^{-\lambda_1}\cdot \Sat(V^{\lambda_1})[\langle \lambda_1,2\check\rho\rangle] \simeq 
t^{-\lambda_1}\cdot \delta_{1,\Gr_G} \star \Sat(V^{\lambda_1})[\langle \lambda_1,2\check\rho\rangle]  
\overset{\text{\eqref{e:transition map initial 1}}}\longrightarrow \\
\to t^{-\lambda_1}\cdot t^{-\lambda}\cdot \Sat(V^{\lambda})[\langle \lambda,2\check\rho\rangle]  
\star \Sat(V^{\lambda_1})[\langle \lambda_1,2\check\rho\rangle] \simeq 
t^{-\lambda_2}\cdot \Sat(V^{\lambda}\otimes V^{\lambda_1})[\langle \lambda_2,2\check\rho\rangle] \overset{\text{\eqref{e:Plucker map}}}\longrightarrow \\
\to t^{-\lambda_2}\cdot \Sat(V^{\lambda_2})[\langle \lambda_2,2\check\rho\rangle]
\end{multline*} 

\sssec{}

It is easy to check that the above assignment defines a functor from $(\Lambda^+,\leq)$ to the \emph{homotopy category} of $\Shv(\Gr_G)$,
i.e., that for $\lambda_1\leq \lambda_2\leq \lambda_3$, the corresponding two maps
$$t^{-\lambda_1}\cdot \Sat(V^{\lambda_1})[\langle \lambda_1,2\check\rho\rangle] \rightrightarrows
t^{-\lambda_3}\cdot \Sat(V^{\lambda_3})[\langle \lambda_3,2\check\rho\rangle]$$
are equal in the homotopy category. 

\medskip

In \secref{ss:infty} below we will show that the above assignment uniquely extends to a functor of $\infty$-categories
\begin{equation} \label{e:construction}
(\Lambda^+,\leq) \to \Shv(\Gr_G), \quad \lambda\mapsto t^{-\lambda}\cdot \Sat(V^\lambda)[\langle \lambda,2\check\rho\rangle].
\end{equation}

\medskip

Assuming the existence of the functor \eqref{e:construction}, we define an object $\ICs\in \Shv(\Gr_G)$ by 

\begin{equation} \label{e:IC as colimit}
\ICs:=\underset{\lambda\in (\Lambda^+,\leq)}{\on{colim}}\, t^{-\lambda}\cdot \Sat(V^\lambda)[\langle \lambda,2\check\rho\rangle]
\end{equation} 

\begin{rem}
A conceptual framework for the formation of the colimit \eqref{e:IC as colimit}, suggested to us by S.~Raskin, will
be explained in \secref{ss:Drinfeld-Plucker}.
\end{rem}

\sssec{}

Here are the first properties of $\ICs$:

\begin{prop}  \label{p:properties IC} \hfill

\smallskip

\noindent{\em(a)} $\ICs$ belongs to $\SI(\Gr_G):=\Shv(\Gr_G)^{\fL(N)}$;

\smallskip

\noindent{\em(b)} $\ICs$ is supported on $\ol{S}{}^0$;

\smallskip

\noindent{\em(c)} $(\bi_0)^!(\ICs)\simeq \omega_{S^0}\simeq (\bi_0)^*(\ICs)$;

\smallskip

\noindent{\em(d)} $\ICs$ belongs to $\SI(\Gr_G)^\heartsuit$.

\end{prop} 

\ssec{Proof of \propref{p:properties IC}}

\sssec{}

For point (a), fix a subgroup $N_\alpha$ as in \secref{sss:N alpha}. We need to show that
$\ICs$ is $N_\alpha$-equivariant.

\medskip

Consider the subset of $\Lambda^+$ consisting of those
elements $\lambda$ such that $\on{Ad}_{t^{\lambda}}(N_\alpha)\subset \fL^+(N)$. Clearly, this subset is cofinal.

\medskip

It is clear that for every such $\lambda$, the object
$$t^{-\lambda}\cdot \CF$$
is $N_\alpha$-equivariant for any $\CF$ that is $\fL^+(N)$-equivariant.

\medskip

This implies that all the terms in the colimit \eqref{e:IC as colimit} corresponding to $\lambda$'s from this subset are
$N_\alpha$-equivariant. 

\sssec{}

For point (b) we need to show that if $S^\mu$ is in the support of $\ICs$, then $\mu\in -\Lambda^{\on{pos}}$.
However, this follows from the fact that
$$S^{\mu}\cap \ol\Gr^\lambda_G\neq \emptyset\, \Rightarrow\, \lambda-\mu\in \Lambda^{\on{pos}}.$$

\sssec{}

In view of (b), we have $(\bi_0)^!(\ICs)\simeq (\bi_0)^*(\ICs)$.  Therefore, using \propref{p:on orbit} and point (a), 
in order to prove (c), we need to show that
$$(\bk_0)^!\circ (\bi_0)^!(\ICs)\simeq \sfe.$$

\medskip

The left-hand side is a colimit over $\lambda \in \Lambda^+$ of the !-fibers of 
$$t^{-\lambda}\cdot \Sat(V^\lambda)[\langle \lambda,2\check\rho\rangle]$$
at the point $1\in \Gr_G$.

\medskip

Each term in this colimit is canonically isomorphic to $\sfe$. And it is easy to see from the definition
of the maps \eqref{e:transition map}, that the maps in this system are the identity maps.

\sssec{}

For point (d), let us first prove that $\ICs\in \SI(\Gr_G)^{\geq 0}$. We will show that for $0\neq \mu\in -\Lambda^{\on{pos}}$,
we have
$$(\bi_\mu)^!(\ICs)\in (\SI_{=\mu})^{>0}.$$

\medskip

By \propref{p:on orbit} and point (a), we need to show that the colimit over $\lambda\in \Lambda^+$ of the !-fibers of 
$$t^{-\lambda}\cdot \Sat(V^\lambda)[\langle \lambda,2\check\rho\rangle]$$
at $t^\mu$ lives in degrees $>\langle \mu,2\check\rho \rangle$.  

\medskip

For every fixed $\mu$, this is the same as the !-fiber of $\Sat(V^\lambda)[\langle \lambda,2\check\rho\rangle]$
at $t^{\lambda+\mu}$. Now the desired estimate follows from the fact that the !-restriction of $\Sat(V^\lambda)$ to
$\Gr^{\lambda+\mu}_G$ is concentrated in strictly positive perverse degrees, while
$$\dim(\Gr^{\lambda+\mu}_G)=\langle \lambda+\mu,2\check\rho \rangle.$$

\sssec{}  \label{sss:* estim}

Let us now show that $\ICs\in \SI(\Gr_G)^{\leq 0}$. By \secref{sss:t}(iii), 
we need to show that for $\mu\in -\Lambda^{\on{pos}}$,
we have
$$(\bi_\mu)^*(\ICs)\in (\SI_{=\mu})^{\leq 0}.$$

In fact, we will see that $(\bi_\mu)^*(\ICs)$ belongs to $(\SI_{=\mu})^{\heartsuit}$ for all $\mu\in -\Lambda^{\on{pos}}$.
(We will also see in \propref{p:* fibers of IC} that these objects are non-zero, which implies that $\ICs$ is \emph{not} the minimal 
extension from $S^0$.)

\medskip

By \propref{p:on orbit} and point (a), we need to show that the colimit over $\lambda\in \Lambda^+$ of 
$$(\bk_\mu)^!\circ (\bi_\mu)^*\left(t^{-\lambda}\cdot \Sat(V^\lambda)[\langle \lambda,2\check\rho\rangle]\right)$$
lives degree $\langle \mu,2\check\rho \rangle$.  

\medskip

We claim this happens for each term in the colimit. Indeed, using \lemref{l:Braden}(a), we rewrite
$$(\bk_\mu)^!\circ (\bi_\mu)^*(t^{-\lambda}\cdot \Sat(V^\lambda))\simeq
H_c(S^{\mu+\lambda},(\bi_{\mu+\lambda})^*(\Sat(V^\lambda))).$$

\medskip

Now, by \secref{sss:Satake}(ii), the latter cohomology indeed lives in single degree equal to 
$\langle\mu+\lambda,2\check\rho\rangle$. 

\qed[\propref{p:properties IC}]

\ssec{Fibers of the semi-infinite IC sheaf}  \label{ss:fibers of IC}

\sssec{}

We will now derive some information on the shape
of $(\bi_\mu)^!(\ICs)$ and $(\bi_\mu)^*(\ICs)$.

\medskip

We claim: 

\begin{prop}  \label{p:! fibers of IC}
For $\mu\in -\Lambda^{\on{pos}}$, we have
$$(\bi_\mu)^!(\ICs) \simeq \omega_{S^\mu}[-\langle \mu,2\check\rho \rangle]\otimes \Sym(\cg/\cb^-[-2])(-\mu).$$
\end{prop}

\begin{proof}

We need to show that the colimit over $\lambda\in \Lambda^+$ of !-fibers at $t^{\lambda+\mu}$ of 
$\Sat(V^\lambda)[\langle \lambda+\mu,2\check\rho\rangle]$, under the maps \eqref{e:transition map}
identifies with $\Sym((\cg/\cb^-)[-2])(-\mu)$. This calculation will be performed in \secref{sss:fibers on Gr}.

\end{proof} 

\sssec{}

We now claim:

\begin{prop}  \label{p:* fibers of IC}
For $\mu\in -\Lambda^{\on{pos}}$ there is a canonical isomorphism
$$(\bi_\mu)^*(\ICs)\simeq \omega_{S^\mu}[-\langle \mu,2\check\rho \rangle]\otimes \CO(\cN)(\mu),$$
where $\CO_\cN$ denotes the algebra of regular functions on $\cN$, viewed as a $T$-representation
via the adjoint action.
\end{prop}

\begin{proof}

We need to show that the colimit over $\lambda\in \Lambda^+$ of 
$$(\bk_\mu)^!\circ (\bi_\mu)^*\left(t^{-\lambda}\cdot \Sat(V^\lambda)[\langle \lambda+\mu,2\check\rho\rangle]\right)$$
identifies canonically with $\CO(\cN)(\mu)$.

\medskip

By \secref{sss:Satake}(ii), each term in the colimit identifies canonically with $V^\lambda(\lambda+\mu)$. Furthermore, by 
\secref{ss:mult hw} below, the transition maps are given as follows: 

\medskip

For $\lambda_2=\lambda_1+\lambda$, the corresponding map $V^{\lambda_1}(\lambda_1+\mu)\to V^{\lambda_2}(\lambda_2+\mu)$ is
given by
$$
V^{\lambda_1}(\lambda_1+\mu) \simeq V^\lambda(\lambda)\otimes V^{\lambda_1}(\lambda_1+\mu)\to 
(V^\lambda\otimes V^{\lambda_1})(\lambda_2+\mu)
\overset{\text{\eqref{e:Plucker map}}}\longrightarrow V^{\lambda_2}(\lambda_2+\mu).$$

\medskip

However, we claim that this colimit indeed identifies with $\CO(\cN)(\mu)$. Namely, we map each $V^\lambda(\lambda+\mu)$ to 
$\CO(\cN)(\mu)$ by the matrix coefficient map, which sends $v\in V^\lambda(\lambda+\mu)$ to the function
$$n\mapsto \langle (v^\lambda)^*,n\cdot v\rangle.$$

\end{proof}

\ssec{Tensor product maps in terms of geometric Satake}  \label{ss:mult hw} 

In this subsection we will explain the geometry that stands behind the claim made in the course
of the proof of \propref{p:* fibers of IC} above. 

\sssec{}

Let $V$ be a representation of $\cG$, and $\lambda\in \Lambda^+$ and $\mu\in \Lambda$. Consider the map
\begin{equation} \label{e:mult by hw}
V(\mu) \simeq 
V^\lambda(\lambda)\otimes V(\mu)\to  (V^\lambda\otimes V)(\lambda+\mu).
\end{equation}

In this subsection we will describe it in terms of geometric Satake.

\sssec{}

Let $\Conv_G$ denote the convolution diagram, i.e.,
$$\Conv_G=\fL(G)\overset{\fL^+(G)}\times \Gr_G,$$
where $\overset{\fL^+G}\times$ means taking the quotient with respect to the diagonal action of $\fL^+(G)$.
We will think of $\Conv_G$ as fibered over $\Gr_G$ via 
$$\on{pr}:\fL(G)\overset{\fL^+(G)}\times \Gr_G\to \fL(G)/\fL^+(G)=\Gr_G$$
with typical fiber $\Gr_G$. We will write $\Conv_G$ as  a ``twisted product" 
$$\Gr_G\wt\times \Gr_G,$$
where the first factor is the ``base copy" $\Gr_G$ and the second factor is the ``fiber copy" of $\Gr_G$. 

\medskip

For any $\CF\in \Shv(\Gr_G)$ and $\CF'\in \Shv(\Gr_G)^{\fL^+(G)}=\Sph(G)$, we can form their twisted external product
$$\CF\wt\boxtimes \CF'\in \Shv(\Conv_G).$$

\medskip

The action of $\fL(G)$ on $\Gr_G$ defines yet another map
$$\on{act}:\Conv_G\to \Gr_G.$$
This map is (ind)-proper. For $\CF,\CF'$ as above, their convolution product $\CF\star \CF'$ is by definition
$$\on{act}_!(\CF\wt\boxtimes \CF')\simeq \on{act}_*(\CF\wt\boxtimes \CF').$$

\sssec{}

Using \secref{sss:Satake} and \lemref{l:Braden}(b), for $V$, $\lambda$ and $\mu$ as above, we need to describe the map
\begin{multline}  \label{e:mult by hw geom}
H(S^{-,\mu},(\bi_{-,\mu})^!(\Sat(V)))[\langle \mu,2\check\rho \rangle]
\to H(S^{-,\mu+\lambda},(\bi_{-,\mu+\lambda})^!(\IC_{\ol\Gr_G^\lambda}\star \Sat(V)))[\langle \mu+\lambda,2\check\rho \rangle]
\simeq \\
\simeq H(\on{act}^{-1}(S^{-,\mu+\lambda}),(\wt\bi_{-,\mu+\lambda})^!(\IC_{\ol\Gr_G^\lambda}\wt\boxtimes \Sat(V)))
[\langle \mu+\lambda,2\check\rho \rangle],
\end{multline} 
where $\wt\bi_{-,\mu+\lambda}$ denotes the locally closed embedding
$$\on{act}^{-1}(S^{-,\mu+\lambda})\hookrightarrow \Gr_G\wt\times \Gr_G=\Conv_G.$$

\medskip

We now notice that 
$$\on{act}^{-1}(S^{-,\mu+\lambda})\cap (\ol\Gr^\lambda_G\wt\times \Gr_G)$$
contains as a closed subscheme the preimage with respect to $\on{pr}$ of
$$S^{-,\lambda}\cap \ol\Gr^\lambda_G=\{t^\lambda\}.$$

This subscheme is isomorphic to a copy of $S^{-,\mu}$, and the !-pullback of 
$$(\wt\bi_{-,\mu+\lambda})^!(\IC_{\ol\Gr_G^\lambda}\wt\boxtimes \Sat(V))[\langle \mu+\lambda,2\check\rho \rangle]$$ to it identifies 
with $(\bi_{-,\mu})^!(\Sat(V))[\langle \mu,2\check\rho \rangle]$. 

\medskip

We claim that the resulting map
$$H(S^{-,\mu},(\bi_{-,\mu})^!(\Sat(V)))[\langle \mu,2\check\rho \rangle]\to 
H(\on{act}^{-1}(S^{-,\mu+\lambda}),(\wt\bi_{-,\mu+\lambda})^!(\IC_{\ol\Gr_G^\lambda}\wt\boxtimes \Sat(V)))
[\langle \mu+\lambda,2\check\rho \rangle]$$
is the desired map \eqref{e:mult by hw geom}.

\medskip

Indeed, this follows from the construction of the monoidal structure on the functor $\Sat$.

\ssec{Construction at the level of $\infty$-categories}  \label{ss:infty}

In this subsection we carry out the construction of $\ICs$ at the level of $\infty$-categories. 

\sssec{}  \label{sss:central}

Consider the following paradigm. Let $\CA$ be a monoidal $\infty$-category, and let $\CC$ be a right-lax bi-module
 $\infty$-category, which means that we have the maps
$$a_1\star (a_2\star c)\to (a_1\star a_2)\star c \text{ and }
(c\star a_2)\star a_1 \to c\star (a_2\star a_1),\quad a_1,a_2\in \CA,\,\,c\in \CC$$
(satisfying a coherent system of higher compatibilities) but these maps are not necessarily isomorphisms. 

\medskip

Let $c\in \CC$  be a lax central object. By this we mean that we are given a family of maps
$$a\star c\overset{\phi(a,c)}\longrightarrow c\star a, \quad a\in A$$
that make the diagrams
\begin{equation} \label{e:weak central}
\CD
a_1\star (a_2\star c)   @>{\phi(a_2,c)}>>   a_1\star (c\star a_2)  \\
@VVV   @VV{\sim}V   \\
(a_1\star a_2)\star c  & & (a_1\star c)\star a_2 \\
@V{\phi(a_1\star a_2,c)}VV   @VV{\phi(a_1,c)}V   \\
c\star (a_1\star a_2)  @<<<  (c\star a_1) \star a_2
\endCD
\end{equation}
commute, along with a coherent system of higher compatibilities. 

\sssec{}  \label{sss:ordinary}

Assume for a moment that the monoidal category $\CA$ is ordinary (i.e., the mapping spaces are discrete). 
Assume also that for any
$a_1,...,a_n\in \CA$, the space
$$\Maps\left(a_1\star (a_2\star...(a_n\star c)...)),c\star (a_1\star....a_n)\right)$$
is discrete. In this case, the datum of the
maps $\phi(a,c)$ and the commutativity of the diagrams \eqref{e:weak central} at the level of 
homotopy categories uniquely extends to the $\infty$-level. 

\sssec{}

Let $\CA^\sim$ be the $\infty$-category that has the same objects as $\CA$, but where we set
$\Maps_{\CA^\sim}(a_1,a_2)$ to be the $\infty$-groupoid underlying the category
$$(a\in \CA, a_2\simeq a\star a_1),$$
and where the composition is given by
$$(a''\in \CA, a_3\simeq a''\star a_2)\circ (a'\in \CA, a_2\simeq a'\star a_1)=(a''\star a',a_3\simeq (a''\star a')\star a_1).$$

\medskip

Suppose now that the left action of $\CA$ is strict (i.e., non-lax), and that each of the functors
$$c\mapsto a\star c$$
admits a right adjoint. We will denote it symbolically by 
$$c\mapsto a^\vee \star c.$$

\medskip

In this case, we can consider the functor 
$$\CA^\sim\to \CC$$
that at the level of objects sends $a\in \CA$ to 
$$a^\vee\star (c\star a)$$
and at the level of morphisms sends $(a\in \CA, a_2\simeq a\star a_1)$ to
\begin{multline*} 
a_1^\vee\star (c\star a_1) \to 
a_1^\vee\star (a^\vee \star (a \star (c\star a_1))) \simeq 
a_1^\vee\star (a^\vee \star ((a \star c)\star a_1))) \overset{\phi(a,c)}\longrightarrow \\
\to a_1^\vee\star (a^\vee \star ((c \star a)\star a_1)))  \to 
 a_1^\vee\star (a^\vee \star (c \star (a\star a_1))) \overset{\text{left action is strict}}\simeq \\
 \simeq (a\star a_1)^\vee \star (c \star (a\star a_1))) \simeq a_2^\vee \star (c\star a_2).
\end{multline*}

\sssec{}

We apply the above paradigm to $\CA=\Lambda^+$, viewed as a monoidal structure given by the structure of
semi-group, and $\CC=\Shv(\Gr_G)$.  

\medskip

The left action of $\Lambda$ on $\Shv(\Gr_G)$ is given by 
$$\lambda\star \CF=t^\lambda\cdot \CF.$$

The right action is obtained by composing the right-lax functor
$$\Lambda^+\to \Rep(\cG), \quad \lambda\to V^\lambda[\langle \lambda,2\check\rho\rangle]$$
(with the right-lax monoidal structure given by the maps \eqref{e:Plucker map}) with
the right convolution action of $\Rep(\cG)$ on $\Shv(\Gr_G)$ obtained from
Geometric Satake. 

\medskip

We take $c=\delta_{1,\Gr_G}$. The map
\begin{equation} \label{e:central on delta}
t^\lambda\cdot \delta_{1,\Gr_G}=\delta_{t^\lambda,\Gr_G}\to \IC_{\ol\Gr^\lambda_G}[\langle \lambda,2\check\rho\rangle]=
\delta_{1,\Gr_G}\star \Sat(V^\lambda)[\langle \lambda,2\check\rho\rangle]
\end{equation} 
is given by the identification
$$H_{\{t^\lambda\}}(\Gr_G,\IC_{\ol\Gr^\lambda_G})[\langle \lambda,2\check\rho\rangle]\simeq 
H_{\{t^\lambda\}}(\Gr^\lambda_G,\IC_{\Gr^\lambda_G})[\langle \lambda,2\check\rho\rangle]
\simeq H_{\{t^\lambda\}}(\Gr^\lambda_G,\omega_{\Gr^\lambda_G})\simeq \sfe.$$

\medskip

We claim that $c$ indeed has a unique weak central structure which at the level of objects is \eqref{e:central on delta}.
To prove this, we will show that we are in the situation of \secref{sss:ordinary}.

\medskip 

Indeed, we need to check that
$$\Maps(\delta_{t^\lambda,\Gr_G},\IC_{\ol\Gr^\lambda_G}[\langle \lambda,2\check\rho\rangle])$$
is discrete. However, we have just seen the above mapping space identifies with $\sfe$. 

\ssec{Another way to view $\ICs$}

In this subsection we will write $\ICs$ as a colimit in the abelian category $(\SI(\Gr_G))^\heartsuit$.

\sssec{}

First, we make the following observation: 
 
\begin{prop}  \label{p:conv exact}
The action of $\Sph(G)^\heartsuit$ on $\SI(\Gr_G)$ is t-exact.
\end{prop}

\begin{proof}

Let $\CF$ be an object of $(\SI(\Gr_G))^\heartsuit$. With no restriction of generality, we can assume that 
$\CF$ is compact. To prove that the functor $-\star \CF$ is right t-exact on $\SI(\Gr_G)$, by \secref{sss:t}(iii), 
it is enough to show that the objects $\Delta^\lambda\star \CF$ are connective.  I.e.,
we need to check that $(\bi_\mu)^*(\Delta^\lambda\star \CF)$ is a connective object of $(\SI(\Gr_G))_{=\mu}$
for every $\mu$.

\medskip

In fact, we claim that 
$$(\bi_\mu)^*(\Delta^\lambda\star \CF)\simeq 
H_c(S^{\mu-\lambda},(\bi_{\mu-\lambda})^*(\CF))[\langle \mu-\lambda,2\check\rho\rangle]
\otimes \omega_{S^\mu}[\langle -\mu,2\check\rho\rangle],$$
or equivalently
$$(\bi_\mu)^*((\bi_\lambda)_!(\omega_{S^\lambda})\star \CF)\simeq 
H_c(S^{\mu-\lambda},(\bi_{\mu-\lambda})^*(\CF))
\otimes \omega_{S^\mu}.$$

Indeed, $(\bi_\mu)^*((\bi_\lambda)_!(\omega_{S^\lambda})\star \CF)$ ia given as a !-direct image under the map
$$\on{act}:S^\lambda\wt\boxtimes S^{\mu-\lambda}\to S^\mu$$
of $\omega_{S^\lambda}\wt\boxtimes (\bi_{\mu-\lambda})^*(\CF)$
and is isomorphic to $\omega_{S^\mu}\otimes V$ for some $V\in \Vect$. In order to calculate $V$ we compute
\begin{multline*}
V\simeq H_c(S^\mu,\omega_{S^\mu}\otimes V)\simeq\
H_c\left(S^\mu,\on{act}_!(\omega_{S^\lambda}\wt\boxtimes (\bi_{\mu-\lambda})^*(\CF))\right)\simeq \\
\simeq 
H_c(S^\lambda\wt\boxtimes S^{\mu-\lambda},\omega_{S^\lambda}\wt\boxtimes (\bi_{\mu-\lambda})^*(\CF))\simeq
H_c(S^\lambda,\omega_{S^\lambda})\otimes H_c(S^{\mu-\lambda},(\bi_{\mu-\lambda})^*(\CF))
\simeq H_c(S^{\mu-\lambda},(\bi_{\mu-\lambda})^*(\CF)),
\end{multline*}
as required.

\medskip

In order to show that $-\star \CF$ is left t-exact we either swap ! and * in the above argument, or argue as follows:

\medskip

We note that for $\CF\in \Sph(G)$, both the left (and also right) adjoint of $-\star \CF$ is given by $-\star \BD(\CF^\tau)$,
where $\BD$ denotes Verdier duality, and $\tau$ is the anti-involution on $\Sph(G)$, induced by the inversion on $\fL(G)$.
Now, since the left adjoint is right t-exact (by the above), the right adjoint is left t-exact.

\end{proof}

\sssec{}

Consider the objects
$$\Delta^{-\lambda}\star \Sat(V^\lambda)\in \SI(\Gr_G).$$
By \propref{p:conv exact} and assuming \thmref{t:j is perv}, we obtain that they all lie in $(\SI(\Gr_G))^\heartsuit$. 

\medskip

Note that by \eqref{e:Av Delta}, we have
$$\on{Av}_!^{\fL(N)}(t^{-\lambda}\star \Sat(V^\lambda))[-\langle \lambda,2\check\rho\rangle]
\simeq \on{Av}_!^{\fL(N)}(\delta_{t^{-\lambda}})[-\langle \lambda,2\check\rho\rangle]\star \Sat(V^\lambda)
\simeq \Delta^{-\lambda}\star \Sat(V^\lambda)\in \SI(\Gr_G),$$
where the first isomorphism is due to the fact that the action of $\Sph(G)$ by right
convolutions is given by proper maps.

\medskip

In particular, applying $\on{Av}_!^{\fL(N)}$ to the maps \eqref{e:transition map}, we obtain a functor 
$$(\Lambda^+,\leq)\to (\SI(\Gr_G))^\heartsuit, \quad \lambda\mapsto \Delta^{-\lambda}\star \Sat(V^\lambda).$$

\medskip

We claim:

\begin{prop}
There exists a canonical isomorphism
$$\ICs\simeq \underset{\lambda\in (\Lambda^+,\leq)}{\on{colim}}\, \Delta^{-\lambda}\star \Sat(V^\lambda)\in \SI(\Gr_G).$$
\end{prop}

\begin{proof}
Since $\ICs$ is already $\fL(N)$-equivariant, we have
$$\on{Av}_!^{\fL(N)}(\ICs)\simeq \ICs.$$

Now, the functor $\on{Av}_!^{\fL(N)}$, being a left adjoint, commutes with colimits, and hence
$$\on{Av}_!^{\fL(N)}\left(\underset{\lambda\in (\Lambda^+,\leq)}{\on{colim}}\, t^{-\lambda}\cdot \Sat(V^\lambda)\right)
\simeq \underset{\lambda\in (\Lambda^+,\leq)}{\on{colim}}\, \Delta^{-\lambda}\star \Sat(V^\lambda)\in \SI(\Gr_G).$$

\end{proof}

\section{Relation to the IC sheaf of Drinfeld's compactification} \label{s:Drinf}

In this section we let $X$ be a smooth and complete curve with a marked point $x\in X$.

\ssec{Recollections on Drinfeld's compactification}  \label{ss:BunNb}

\sssec{}

Let $\BunNb$ denote Drinfeld's relative compactification of the stack $\Bun_N$ along the fibers of the map
$\Bun_N\to \Bun_G\times \on{pt}/T$.

\medskip

Let us recall the definition of $\BunNb$. By definition, this is the algebraic stack 
that classifies triples $(\CP_G,\kappa)$, where:

\medskip

\noindent(i) $\CP_G$ is a $G$-bundle on $X$;

\medskip

\noindent(ii) $\kappa$ is a \emph{Pl\"ucker} data, i.e., a system of non-zero maps 
$$\kappa^{\check\lambda}:\CO_X\to \CV^{\check\lambda}_{\CP_G},$$
(here $\CV^{\check\lambda}$ denotes the Weyl module with highest weight $\check\lambda\in \check\Lambda^+$)
that satisfy Pl\"ucker relations, i.e., for $\check\lambda_1$ and $\check\lambda_2$ the diagram
$$
\CD
\CO_X\otimes \CO_X  @>{\kappa^{\check\lambda_1}\otimes \kappa^{\check\lambda_2}}>> 
\CV^{\check\lambda_1}_{\CP_G}\otimes \CV^{\check\lambda_2}_{\CP_G}  \\
@A{\sim}AA   @AAA  \\
\CO_X @>{\kappa^{\check\lambda_1+\check\lambda_2}}>> \CV^{\check\lambda_1\check\lambda_2}_{\CP_G} 
\endCD
$$
must commute. 

%

\sssec{}

We let $(\BunNb)_{\infty \cdot x}$ an ind-algebraic stack, which is a version of $\BunNb$, where we allow the maps 
$\kappa^{\check\lambda}$ to have poles at $x$.



\medskip

For each $\lambda\in \Lambda$, we let
$$(\BunNb)_{\leq \lambda\cdot x}\overset{\ol\imath_\lambda}\hookrightarrow (\BunNb)_{\infty \cdot x}$$
be the closed substack, where we bound the order of pole of $\kappa^{\check\lambda}$ by the integer $\langle \lambda,\check\lambda\rangle$. 
In particular,
$$\BunNb=(\BunNb)_{\leq 0\cdot x}.$$



\sssec{}

We let 
$$(\BunNb)_{=\lambda\cdot x} \overset{\jmath_\lambda}\hookrightarrow (\BunNb)_{\leq \lambda\cdot x}$$
be the open substack, where we require that $\kappa^{\check\lambda}$ have a pole of order exactly 
$\langle \lambda,\check\lambda\rangle$ and where we forbid it to have zeroes on $X-x$.
For $\lambda=0$ we have
$$(\BunNb)_{=0\cdot x} =\Bun_N;$$
denote the corresponding map
$\Bun_N\hookrightarrow \BunNb$ simply by $\jmath$. 

\medskip

The stack $(\BunNb)_{=\lambda\cdot x}$ is smooth for any $\lambda$; in fact
$$(\BunNb)_{=\lambda\cdot x}\simeq \Bun_B\underset{\Bun_T}\times \{\CP^0_T(-\lambda\cdot x)\},$$
where $\CP^0_T$ denotes the trivial $T$-bundle on $X$. 

\medskip

Denote
$$\imath_\lambda=\ol\imath_\lambda\circ \jmath_\lambda.$$



\sssec{}

Set
$$\Delta^\lambda_{\on{glob}}:=(\imath_\lambda)_!(\omega_{(\BunNb)_{=\lambda\cdot x}})[-\dim((\BunNb)_{=\lambda\cdot x})]$$
and 
$$\nabla^\lambda_{\on{glob}}:=(\imath_\lambda)_*(\omega_{(\BunNb)_{=\lambda\cdot x}})[-\dim((\BunNb)_{=\lambda\cdot x})].$$

In the above formula,
$$\dim((\BunNb)_{=\lambda\cdot x})=(g-1)\cdot \dim(N)+\langle \lambda,2\check\rho\rangle.$$

Thus, $\omega_{(\BunNb)_{=\lambda\cdot x}}[-\dim((\BunNb)_{=\lambda\cdot x})]$
is the IC sheaf on $(\BunNb)_{=\lambda\cdot x}$. 

\medskip

We let $\IC_{\on{glob}}$ denote the intersection cohomology sheaf on $\BunNb$, viewed as an object of
$\Shv((\BunNb)_{\infty \cdot x})$. 

\sssec{}

Hecke modifications of the underlying $G$-bundle define a (right) action of the monoidal category
$\Sph(G)$ on $\Shv((\BunNb)_{\infty \cdot x})$.
We denote this action by
$$\CF\in \Shv((\BunNb)_{\infty \cdot x}),\,\, \CF'\in \Sph(G)\,\,\mapsto \,\, \CF\star \CF'.$$

\sssec{}

We have a tautological map
$$\pi:\Gr_G\to (\BunNb)_{\infty \cdot x}.$$

Indeed, if $\CP_G$ is a $G$-bundle on $X$, equipped with a trivialization on $X-x$, the tautological
reduction to $N$ of the trivial bundle defines a Pl\"ucker data on $\CP_G$. 

\medskip

Note that the preimage of $(\BunNb)_{\leq \lambda\cdot x}$ (resp., $(\BunNb)_{=\lambda\cdot x}$) under $\pi$ equals 
$\ol{S}{}^\lambda$ (resp., $S^\lambda$). 

\sssec{}

Thus, we have the pullback functor 
$$\pi^!:\Shv((\BunNb)_{\infty \cdot x})\to \Shv(\Gr_G)$$
and its partially defined left adjoint $\pi_!$.

\medskip

Both these functors intertwine the actions of $\Sph_G$ on $\Shv(\Gr_G)$ and $\Shv((\BunNb)_{\infty \cdot x})$.
This is because the convolution action is given by proper maps. 

\ssec{Statement of the result}

\sssec{}

The main result of this paper is the following:

\begin{thm} \label{t:main}
There exists a canonical isomorphism 
$$\ICs\simeq \pi^!(\IC_{\on{glob}})[(g-1)\cdot \dim(N)].$$
\end{thm} 

\sssec{}

It is clear from the definitions that for $\lambda\in \Lambda$, we have a canonical isomorphism
$$\nabla^\lambda\simeq \pi^!(\nabla^\lambda_{\on{glob}})[(g-1)\cdot \dim(N)].$$

\medskip

In addition, by adjunction, we obtain a map 
\begin{equation} \label{e:j map}
\Delta^\lambda\to  \pi^!(\Delta^\lambda_{\on{glob}})[(g-1)\cdot \dim(N)].
\end{equation}

\medskip

Along with \thmref{t:main}, we will prove:
\begin{thm} \label{t:main j}
The map \eqref{e:j map} is an isomorphism.
\end{thm}

\ssec{Description of fibers}  \label{ss:!-fibers via global} 

\sssec{}

First, let us note that \thmref{t:main} gives an answer to the long-standing question of how to see geometrically an isomorphism
between 
$$\underset{\lambda\in \Lambda^+}{\on{colim}}\, 
H_{\{t^{\lambda+\mu}\}}(\Gr_G,\IC_{\ol\Gr^\lambda_G}[\langle \lambda,2\check\rho\rangle])$$
and !-fibers of $(\imath_\mu)^!(\IC_{\on{glob}})$.

\medskip

Indeed, the two sides identify with the !-fiber at $t^\mu$ of the two sides of the isomorphism of 
\thmref{t:main}.

\sssec{}

From Theorems \ref{t:main} and \ref{t:main j} one can obtain explicit descriptions of the objects
$$(\bi_\mu)^!(\ICs) \text{ and } (\bi_\mu)^!(\Delta^0).$$

\sssec{}  

The following was established in \cite[Theorem 1.12]{BFGM}:

\begin{prop}  \label{p:global IC}
For $\mu\in -\Lambda^{\on{pos}}$, there exists a canonical isomorphism
$$(\imath_\mu)^!(\IC_{\on{glob}})\simeq \omega_{(\BunNb)_{=\mu\cdot x}}[-\dim((\BunNb)_{=\mu\cdot x})]\otimes \Sym(\cn^-[-2])(\mu).$$
\end{prop}

Combining this with \thmref{t:main}, we obtain:

\begin{cor}
There exists a canonical isomorphism 
$$(\bi_\mu)^!(\ICs) \simeq \omega_{S^\mu}[-\langle \mu,2\check\rho \rangle]\otimes \Sym(\cn^-[-2])(\mu).$$
\end{cor}

\begin{rem}
Note that \propref{p:! fibers of IC} gives another description of $(\bi_\mu)^!(\ICs)$, namely as
$$\omega_{S^\mu}[-\langle \mu,2\check\rho \rangle]\otimes \Sym(\cg/\cb^-[-2])(-\mu).$$

Presumably, the two identifications are related by isomorphism
$$\cn^-\simeq \cn\simeq \cg/\cb^-,$$
where the first arrow is the Cartan involution. However, the author was not able to prove it.
\end{rem}

\sssec{}

The next assertion follows as a combination of \propref{p:global IC} and \cite[Theorem 6.6]{BG}:

\begin{prop} \label{p:! of j glob}
For $\mu\in -\Lambda^{\on{pos}}$, the object
$$(\imath_\mu)^!(\Delta^0_{\on{glob}})\in \Shv((\BunNb)_{=\mu\cdot x})$$
has a canonical filtration with subquotients of the form
$$\omega_{(\BunNb)_{=\mu\cdot x}}[-\dim((\BunNb)_{=\mu\cdot x})]\otimes \Sym(\cn^-[-2])(\mu_1)\otimes \on{C}^\cdot(\cn)(\mu_2), \quad
\mu_1+\mu_2=\mu.$$
\end{prop}

In the above proposition, $\on{C}^\cdot(\cn)$ denotes the cohomological Chevalley complex of the Lie algebra $\cn$.

\sssec{}

Note that \propref{p:! of j glob} combined with \thmref{t:main j}, immediately imply \thmref{t:j is perv}. Indeed, for $\mu\neq 0$, the complexes
$$\Sym(\cn^-[-2])(\mu_1)\otimes \on{C}^\cdot(\cn)(\mu_2), \quad \mu_1+\mu_2=\mu$$
are concentrated is strictly positive degrees. 

\ssec{Construction of the map}

In this subsection we begin the proof of \thmref{t:main} by constructing the map $\ICs\to \pi^!(\IC_{\on{glob}})[(g-1)\cdot \dim(N)]$.

\sssec{}

By adjunction, the datum of a map
$$\ICs\to \pi^!(\IC_{\on{glob}})[(g-1)\cdot \dim(N)]$$
is equivalent to that of a map
\begin{equation} \label{e:map one direction}
\pi_!(\ICs)\to \IC_{\on{glob}}[(g-1)\cdot \dim(N)].
\end{equation}

\medskip

We will first construct a map 
\begin{equation} \label{e:map one direction lambda}
\pi_!(t^{-\lambda}\cdot \Sat(V^\lambda))[\langle \lambda,2\check\rho \rangle]\to \IC_{\on{glob}}[(g-1)\cdot \dim(N)]
\end{equation}
for an individual $\lambda\in \Lambda^+$.  

\sssec{}

Note that for $\CF'\in \Sph(G)$, the functor adjoint to $-\star \CF'$ is given by 
$$-\star \BD((\CF')^\tau),$$
where $\BD$ denotes Verdier duality on $\Sph(G)=\Shv(\Gr_G)^{\fL^+(G)}$ and $\tau$ is the anti-involution on $\Sph(G)$
induced by the inversion on $\fL(G)$. 

\medskip

Thus, the functor adjoint to convolution with $\Sat(V^\lambda)=\IC_{\ol\Gr^\lambda_G}$ is given by convolution with
$\IC_{\ol\Gr^{-\lambda}_G}$. Therefore, by adjunction, the datum of a map \eqref{e:map one direction lambda} is
equivalent to that of a map
$$\delta_{\pi(t^{-\lambda})}[\langle \lambda,2\check\rho \rangle]\to \IC_{\on{glob}}\star \IC_{\ol\Gr^{-\lambda}_G}[(g-1)\cdot \dim(N)],$$
or which is the same, of a vector in the !-fiber of
$$\IC_{\on{glob}}\star \IC_{\ol\Gr^{-\lambda}_G}[(g-1)\cdot \dim(N)-\langle \lambda,2\check\rho \rangle]$$
at the point $\pi(t^{-\lambda})\in (\BunNb)_{\infty\cdot x}$. 

\medskip

We claim that the fiber in question identifies canonically with $\sfe$. 

\sssec{}

Indeed, consider the convolution morphism
$$
\CD
(\BunNb)_{\infty\cdot x}\wt\times  \Gr_G @>{\on{pr}}>> (\BunNb)_{\infty\cdot x} \\
@V{\on{act}}VV  \\
(\BunNb)_{\infty\cdot x}.
\endCD
$$

It is easy to see that the intersection 
$$\on{act}^{-1}(\pi(t^{-\lambda}))\cap (\BunNb\wt\times  \ol\Gr^{-\lambda}_G)$$
coincides with its open subset 
\begin{equation} \label{e:simple fiber of conv}
\on{act}^{-1}(\pi(t^{-\lambda}))\cap (\Bun_N\wt\times  \Gr^{-\lambda}_G)
\end{equation}
and identifies canonically with $S^\lambda\cap \Gr^\lambda_G$.

\medskip

The restriction of 
$$\IC_{\on{glob}}[(g-1)\cdot \dim(N)]\wt\boxtimes \IC_{\ol\Gr^{-\lambda}_G}[\langle \lambda,2\check\rho \rangle]$$
to $\Bun_N\wt\times  \Gr^{-\lambda}_G$ identifies with $\omega_{\Bun_N\wt\times  \Gr^{-\lambda}_G}$, and hence
its further !-restriction to \eqref{e:simple fiber of conv} identifies with
$$\omega_{S^\lambda\cap \Gr^\lambda_G}.$$

\medskip

Thus, the !-fiber of $$\IC_{\on{glob}}\star \IC_{\ol\Gr^{-\lambda}_G}[(g-1)\cdot \dim(N)-\langle \lambda,2\check\rho \rangle]$$
at $\pi(t^{-\lambda})\in (\BunNb)_{\infty\cdot x}$ identifies with
$$H(S^\lambda\cap \Gr^\lambda_G,\omega_{S^\lambda\cap \Gr^\lambda_G})[-2\langle \lambda,2\check\rho \rangle],$$
and the latter is indeed canonically $\sfe$, since 
$$S^\lambda\cap \Gr^\lambda_G=\fL^+(N)/\on{Ad}_{t^\lambda}(\fL^+(N)),$$
and is isomorphic to an affine space of dimension $\langle \lambda,2\check\rho \rangle$. 

\sssec{}

It is easy to see from the construction that for $\lambda_2=\lambda_1+\lambda$, the corresponding maps
$$\pi_!(t^{-\lambda_1}\cdot \Sat(V^{\lambda_1}))[\langle \lambda_1,2\check\rho \rangle]\to \IC_{\on{glob}}[(g-1)\cdot \dim(N)]$$
and
$$\pi_!(t^{-\lambda_2}\cdot \Sat(V^{\lambda_2}))[\langle \lambda_2,2\check\rho \rangle]\to \IC_{\on{glob}}[(g-1)\cdot \dim(N)]$$
are compatible with the morphism \eqref{e:transition map}.

\medskip

We claim that these maps combine to give rise to a uniquely defined map
$$\pi_!(\ICs)=\underset{\lambda\in \Lambda^+}{\on{colim}}\, 
\pi_!(t^{-\lambda}\cdot \Sat(V^\lambda))[\langle \lambda,2\check\rho \rangle]\to \IC_{\on{glob}}[(g-1)\cdot \dim(N)].$$

This follows from the next general observation.

\sssec{}

Let $\CC$ be an $\infty$-category, and
$$I\to \CC, \quad i\mapsto c_i$$
a diagram of objects, where $I$ is some index $\infty$-category. 

\medskip

Let $c'$ be another object of $\CC$. Let us be given a system of maps $\phi_i:c_i\to c$, so that for every arrow $i_1\to i_2$ in $I$,
the two maps
$$c_{i_1}\to c_{i_2}\overset{\phi_{i_2}}\to c' \text{ and } c_{i_1}\overset{\phi_{i_1}}\to c'$$
agree on the homotopy category of $\CC$. 

\medskip

We have:

\begin{lem}
Suppose that for every $i$, the space $\Maps(c_i,c')$ is discrete. Then the maps $\phi_i$ uniquely combine to a map
$$\underset{i\in I}{\on{colim}}\, c_i\to c'.$$
\end{lem} 

\sssec{}

The conditions of the above lemma are applicable in our situation since we have identified 
$$\CMaps_{\Shv((\BunNb)_{\infty\cdot x})}(\pi_!(t^{-\lambda}\cdot \Sat(V^\lambda))[\langle \lambda,2\check\rho \rangle],
\IC_{\on{glob}}[(g-1)\cdot \dim(N)])$$
with $\sfe$. 

\sssec{}  \label{sss:map to nabla}

For future use, we will need to following description of the composite map
$$t^{-\lambda}\cdot \Sat(V^\lambda)[\langle \lambda,2\check\rho \rangle]\to \pi^!(\IC_{\on{glob}})[(g-1)\cdot \dim(N)]\to
\pi^!(\nabla^0_{\on{glob}})[(g-1)\cdot \dim(N)]\simeq (\bj_0)_*(\omega_{S^0}).$$

Namely, unwinding the definitions, we obtain that this map equals
$$t^{-\lambda}\cdot \Sat(V^\lambda)[\langle \lambda,2\check\rho \rangle] \to 
(\bj_0)_*\circ (\bj_0)^*(t^{-\lambda}\cdot \Sat(V^\lambda)[\langle \lambda,2\check\rho \rangle]) \to
(\bj_0)_*(\omega_{S^0}),$$
where the first arrow is the unit of the $((\bj_0)^*,(\bj_0)_*)$-adjunction, and the second comes from the identification
$$(\bj_0)^*(t^{-\lambda}\cdot \Sat(V^\lambda)[\langle \lambda,2\check\rho \rangle])\simeq  
(\bj_0)^!(t^{-\lambda}\cdot \Sat(V^\lambda)[\langle \lambda,2\check\rho \rangle])\simeq \omega_{S^0\cap t^{-\lambda}\cdot \Gr^\lambda_G},$$
where $S^0\cap t^{-\lambda}\cdot \Gr^\lambda_G$ is a closed subscheme of $S^0$ and we regard 
$\omega_{S^0\cap t^{-\lambda}\cdot \Gr^\lambda_G}$ as an object of $\Shv(S^0)$.

\ssec{Strategy of proofs of Theorems \ref{t:main} and \ref{t:main j}}

\sssec{}

In \secref{ss:proof of equiv} we will prove:

\begin{prop} \label{p:equiv}
The objects $\pi^!(\IC_{\on{glob}})$ and $\pi^!(\Delta^0_{\on{glob}})$ belong to $\Shv(\Gr_G)^{\fL(N)}=:\SI(\Gr_G)$.
\end{prop} 

Let us assume this proposition and proceed with the proofs of Theorems \ref{t:main} and \ref{t:main j}. 

\sssec{}

In order to prove \thmref{t:main} it is sufficient to show that for any $\mu\in -\Lambda^{\on{pos}}$,
the map
$$(\bi_\mu)^*(\ICs)\to (\bi_\mu)^*\circ \pi^!(\IC_{\on{glob}}[(g-1)\cdot \dim(N)]),$$
induced by \eqref{e:map one direction} is an isomorphism.

\medskip

In order to prove \thmref{t:main j} (for $\lambda=0$), it suffices to show that 
$$(\bi_\mu)^*\circ \pi^!(\Delta^0_{\on{glob}})=0$$
for $\mu\neq 0$. The proof for a general $\lambda$ will be completely analogous. 

\sssec{}

Taking into account \propref{p:on orbit} and using \lemref{l:Braden}, we obtain that Theorems \ref{t:main} and \ref{t:main j}
would follow once we establish the next assertion:

\begin{prop}  \label{p:comp fibers}  For any $\mu\in -\Lambda^{\on{pos}}$ we have:

\smallskip

\noindent{\em(a)}
The map
$$H(S^{-,\mu},(\bi_{-,\mu})^!(\ICs))\to H(S^{-,\mu},(\bi_{-,\mu})^!\circ \pi^!(\IC_{\on{glob}}[(g-1)\cdot \dim(N)])),$$
induced by \eqref{e:map one direction}, is an isomorphism.

\smallskip

\noindent{\em(b)} $H(S^{-,\mu},(\bi_{-,\mu})^!\circ \pi^!(\Delta^0_{\on{glob}}))=0$ for $\mu\neq 0$. 

\end{prop}

\ssec{Recollections about the Zastava space}

\sssec{}

Consider the fiber product
\begin{equation} \label{e:pre-Zastava}
\BunNb\underset{\Bun_G}\times \Bun^{-\mu}_{B^-}.
\end{equation} 

Here $\Bun^{-\mu}_{B^-}$ denotes the connected component of $\Bun_{B^-}$ equal to
$$\Bun_{B^-}\underset{\Bun_T}\times \Bun_T^{-\mu},$$
where $\Bun_T^{-\mu}$ corresponds to $T$-bundles of degree $-\mu$. 

\medskip

The Zastava space $\CZ^\mu$ is by definition the open subset in \eqref{e:pre-Zastava}, corresponding to the
condition that (generic) reduction of $\CP_G$ to $N$ and the (genuine) reduction of $\CP_G$ to $B^-$ are
transversal at the generic point of $X$. By construction, $\CZ^\mu$ is an algebraic stack, but it is in fact
a quasi-projective scheme. 

\medskip

Let $\fq$ denote the forgetful map
$$\CZ^\mu\to \BunNb.$$

\medskip

We denote by $\oCZ{}^\mu\subset \CZ^\mu$ the open subset equal to the preimage 
of $\Bun_N\subset \BunNb$ under $\fq$. By a slight abuse of notation we denote by
$\jmath$ the corresponding open embedding. The scheme $\oCZ{}^\mu$ is smooth. 

\sssec{}

Let us think of a point of $\Bun_{B^-}$ as a triple $(\CP_G,\CP_T,\kappa^-)$, where $\CP_G$ is a $G$-bundle on $X$, 
$\CP_T$ is a $T$-bundle on $X$ of degree $\mu$, and 
$\kappa^-$ is a non-degenerate Pl\"ucker data, i.e., surjective maps
$$\kappa^{-,\check\lambda}:\wt\CV_{\CP_G}^{\check\lambda}\to \check\lambda(\CP_T),$$
satisfying Pl\"ucker relations. In the above formula $\wt\CV^{\check\lambda}$ is the dual Weyl module with highest weight 
$\check\lambda$.

\medskip

A point of \eqref{e:pre-Zastava} belongs to $\CZ^\mu$ if and only if for every $\check\lambda$, the composite map
$$\CO_X \overset{\kappa^{\check\lambda}}\to 
\CV_{\CP_G}^{\check\lambda}\to \wt\CV_{\CP_G}^{\check\lambda}\overset{\kappa^{-,\check\lambda}}\longrightarrow 
\check\lambda(\CP_T^f)$$
is non-zero.

\medskip

In this case, the datum of zeroes of the above composite maps is encoded by a point of $X^\mu$. Here for
an element $\mu\in -\Lambda^{\on{pos}}$ equal to $\underset{i}\Sigma\, -n_i\cdot \alpha_i$ we let
$$X^\mu=\underset{i}\prod\, X^{(n_i)}.$$

\medskip

We denote by $\fs$ the resulting map
$$\CZ^\mu\to X^\mu.$$ 

\sssec{}  \label{sss:central fiber}

Denote by $\fF^\mu$ the ``central fiber" of $\CZ^\mu$ over $X^\mu$, i.e., the preimage of the point $\mu\cdot x\in X^\mu$.
Set 
$$\oF{}^\mu:=\fF^\mu\cap \oCZ{}^\mu.$$

\medskip

According to \cite[Proposition 2.6]{BFGM}, we have a canonical isomorphism
$$\fF^\mu\simeq \ol{S}{}^0\cap S^{-,\mu}$$ so that
$\oF{}^\mu$ corresponds to $S^0\cap S^{-,\mu}$. 

\medskip

It follows from the construction that under this identification the composite map
$$\fF^\mu\hookrightarrow \CZ^\mu \overset{\fq}\longrightarrow \BunNb$$
equals the map
$$\ol{S}{}^0\cap S^{-,\mu}\hookrightarrow \ol{S}{}^0 \overset{\pi}\longrightarrow \BunNb.$$

\sssec{}

The following assertion was implicit in \cite[Sect. 3.4]{BFGM}. For completeness, we will supply a proof in 
\secref{ss:proof Zastava and Bun}: 

\begin{prop} \label{p:Zastava and Bun}  \hfill

\smallskip

\noindent{\em(a)}
We have a (canonical) isomorphism
$$\fq^!(\IC_{\on{glob}})[(g-1)\cdot \dim(N)]\simeq \IC_{\CZ^\mu}[-\langle \mu,2\check\rho\rangle],$$
extending the tautological isomorphism over $\oCZ{}^\mu$. 

\smallskip

\noindent{\em(b)} The map
$\jmath_!(\omega_{\oCZ{}^\mu})\to \fq^!\circ \jmath_!(\omega_{\Bun_N})$
is an isomorphism.

\end{prop} 

In addition, the following was established in \cite[Sect. 5]{BFGM}: 

\begin{prop} \label{p:calc on Zast} \hfill

\smallskip

\noindent{\em(a)}
The cohomology $H_{\{\fF^\mu\}}(\CZ^\mu,\IC_{\CZ^\mu})$
is concentrated in cohomological degree zero. 

\smallskip

\noindent{\em(b)}
The map
\begin{multline} \label{e:map on Zast}
H_{\{\fF^\mu\}}(\CZ^\mu,\IC_{\CZ^\mu})\to H_{\{\oF{}^\mu\}}(\oCZ{}^\mu,\IC_{\oCZ{}^\mu}))\simeq 
H_{\{\oF{}^\mu\}}(\oCZ{}^\mu,\omega_{\oCZ{}^\mu}))[\langle \mu,2\check\rho\rangle]\simeq \\
\simeq H(\ol{S}{}^0\cap S^{-,\mu},\omega_{\ol{S}{}^0\cap S^{-,\mu}})[\langle \mu,2\check\rho\rangle]
\end{multline}
induces an isomorphism in (the lowest) cohomological degree $0$.

\end{prop} 

\sssec{Proof of \propref{p:comp fibers}(b)}

We will now deduce \propref{p:comp fibers}(b) from 
\propref{p:Zastava and Bun}(b). 

\medskip

Indeed, from \secref{sss:central fiber},
we obtain that the expression 
$$H(S^{-,\mu},(\bi_{-,\mu})^!\circ \pi^!(\Delta^0_{\on{glob}}))$$
identifies with the !-fiber at $\mu\cdot x\in X^\mu$ of $\fs_*\circ \fq^!(\Delta^0_{\on{glob}})$. Now, 
using \propref{p:Zastava and Bun} (b) and base change along $\{\mu\cdot x\}\to X^\mu$, we obtain
that it suffices to show that
$$\fs_*\circ \jmath_!(\omega_{\oCZ{}^\mu})=0.$$

\medskip

However, this follows from the fact that there a $\BG_m$-action on $\CZ^\mu$ that preserves the projection
$\fs$ that contracts $\CZ^\mu$ to the canonical section
$$X^\mu \to \CZ^\mu$$
(see \cite[Sect. 5.1]{BFGM})
whose image lies outside of $\oCZ{}^\mu$, unless $\mu=0$. 

\ssec{Proof of \propref{p:comp fibers}(a)}

\sssec{}

For $\lambda\in \Lambda$ consider the map
$$t^{-\lambda}\cdot \IC_{\ol\Gr^\lambda_G}[\langle \lambda,2\check\rho\rangle]
\to \ICs\to \pi^!(\IC_{\on{glob}})[(g-1)\cdot \dim(N)],$$
and the corresponding map
\begin{equation} \label{e:map to check}
H(S^{-,\mu},(\bi_{-,\mu})^!(t^{-\lambda}\cdot \IC_{\ol\Gr^\lambda_G})[\langle \lambda,2\check\rho\rangle])\to 
H(S^{-,\mu},(\bi_{-,\mu})^!\circ \pi^!(\IC_{\on{glob}})[(g-1)\cdot \dim(N)]).
\end{equation} 

\medskip

By \secref{sss:Satake}(ii) and \lemref{l:Braden}, the left-hand side is concentrated in 
single cohomology degree $\langle \mu,2\check\rho\rangle$.

\medskip

The right-hand side is also is concentrated in single cohomology degree $\langle \mu,2\check\rho\rangle$, 
by Propositions \ref{p:Zastava and Bun}(a) and \ref{p:calc on Zast}(a). 

\medskip

Thus, it suffices to show that the map \eqref{e:map to check} induces
an isomorphism on the $\langle \mu,2\check\rho\rangle$ cohomology, after the taking the colimit of the left-hand side
over $\clambda$. 

\sssec{}

First, we note that the intersection
$$S^{-,\mu} \cap (t^{-\lambda}\cdot \ol\Gr^\lambda_G)$$ contains as an open subset
$$S^0\cap S^{-,\mu} \cap (t^{-\lambda}\cdot \Gr^\lambda_G),$$
which is dense in every irreducible component of the (top) dimension $-\langle \mu,2\check\rho\rangle$.  Moreover, for a fixed
$\mu$ and $\lambda$ large, the inclusion
$$S^0\cap S^{-,\mu} \cap (t^{-\lambda}\cdot \Gr^\lambda_G) \subset S^0\cap S^{-,\mu}$$
is an equality. 

\medskip

Hence, we obtain a map
$$H(S^{-,\mu},(\bi_{-,\mu})^!(t^{-\lambda}\cdot \IC_{\ol\Gr^\lambda_G})[\langle \lambda,2\check\rho\rangle])\to
H(S^{-,\mu}\cap S^0\cap (t^{-\lambda}\cdot \Gr^\lambda_G),\omega_{S^{-,\mu}\cap S^0\cap (t^{-\lambda}\cdot \Gr^\lambda_G)}),$$
which is injective at the level of the $\langle \mu,2\check\rho\rangle$ cohomology, and this injection is an isomorphism for
$\lambda$ large.

\sssec{}

Taking into account \propref{p:calc on Zast}(b), it suffices to show that the following diagram commutes
$$
\CD
H(S^{-,\mu},(\bi_{-,\mu})^!(t^{-\lambda}\cdot \IC_{\ol\Gr^\lambda_G})[\langle \lambda,2\check\rho\rangle])
@>>> H(S^0\cap S^{-,\mu}\cap (t^{-\lambda}\cdot \Gr^\lambda_G),\omega_{S^{-,\mu}\cap S^0\cap (t^{-\lambda}\cdot \Gr^\lambda_G)})  \\
@V{\text{\eqref{e:map to check}}}VV  @VVV  \\
H(S^{-,\mu},(\bi_{-,\mu})^!\circ \pi^!(\IC_{\on{glob}})[(g-1)\cdot \dim(N)]) @>>{\text{\eqref{e:map on Zast}}}>   
H(\ol{S}{}^0\cap S^{-,\mu},\omega_{\ol{S}{}^0\cap S^{-,\mu}}). 
\endCD
$$

\medskip

However, this follows from the description of the map
$$t^{-\lambda}\cdot \IC_{\ol\Gr^\lambda_G}[\langle \lambda,2\check\rho\rangle]
\to  \pi^!(\IC_{\on{glob}})[(g-1)\cdot \dim(N)]\to \pi^!(\nabla^0_{\on{glob}})\simeq (\bj_0)_*(\omega_{S^0})$$
in \secref{sss:map to nabla}. 

\ssec{Proof of equivariance}  \label{ss:proof of equiv}

In this subsection we will prove \propref{p:equiv}.

\sssec{}

Let $y$ be a point on $X$ different from $x$. In what follows we will use the subscript $x$ in $\Gr_{G,x}$, $S^\lambda_x$, $\pi_x$,
$\fL(N)_x$, $\fL^+(N)_x$ to emphasize the dependence on $x$. We will also consider the corresponding objects at $y$. 

\sssec{}

Let $(\BunNb)^{\on{good}_y}\subset \BunNb$ be the open substack,
where we forbid the maps $\kappa^{\check\lambda}$ to have a zero at $y$. Clearly, the map
$$\pi_x:\ol{S}{}^0_x\to \BunBb$$
has its image in $(\BunNb)^{\on{good}_y}$. 

\medskip

In addition, we can consider the map 
$$\pi_{x,y}:\ol{S}{}^0_x\times S^0_y\to (\BunNb)^{\on{good}_y}.$$
Its restriction to $\ol{S}{}^0_x\times 1\subset \ol{S}{}^0_x\times S^0_y$ is the original map $\pi_x$.

\medskip

We will prove:

\begin{prop}  \label{p:strong approx}
For $\CF=\IC_{\on{glob}}$ or $\CF=\Delta^0_{\on{glob}}$, the object $\pi_{x,y}^!(\CF)$ is equivariant with respect to
$\fL(N)_y$ acting on the second factor in $\ol{S}{}^0_x\times S^0_y$.
\end{prop}

Let us deduce \propref{p:equiv} from \propref{p:strong approx}:

\begin{proof}

Let $\CF\in \Shv((\BunNb)^{\on{good}_y})$ be such that $\pi_{x,y}^!(\CF)$ is $\fL(N)_y$-equivariant. We claim that
in this case $\pi_x^!(\CF)$ is $\fL(N)_x$-equivariant. 

\medskip

Let $N_{X-(x,y)}$ be the group ind-scheme of maps $(X-\{x,y\})\to N$. Laurent expansion defines
a closed embedding
$$N_{X-(x,y)}\to \fL(N)_x\times \fL(N)_y.$$

\medskip

We have three actions of $N_{X-(x,y)}$ on $\ol{S}{}^0_x\times S^0_y$. One is via the map $N_{X-(x,y)}\to \fL(N)_x$
and the first factor; another is via the map $N_{X-(x,y)}\to \fL(N)_y$ and the second factor; and the third is diagonal. 

\medskip

It is clear that the map $\pi_{x,y}$ is $N_{X-(x,y)}$-invariant with respect to the \emph{diagonal} action. 
In particular, for any $\CF\in \Shv((\BunNb)^{\on{good}_y})$, we have
$$\pi_{x,y}^!(\CF)\in \Shv(\ol{S}{}^0_x\times S^0_y)^{N_{X-(x,y)},\on{diag}}.$$

\medskip

Note, however, that the condition on $\CF$ implies that $\pi_{x,y}^!(\CF)$ is equivariant with respect to
the $N_{X-(x,y)}$-action via the second factor. Combined with diagonal equivariance, we obtain that $\pi_{x,y}^!(\CF)$
is $N_{X-(x,y)}$-equivariant with respect to the action on the first factor.  

\medskip

In particular, we obtain that $\pi_x^!(\CF)$ is $N_{X-(x,y)}$-equivariant with respect to the action on the first factor.  

\medskip

We claim that any object $\CF'\in \Shv(\ol{S}{}^0_x)$ with this property is $\fL(N)_x$-equivariant. Indeed, it is sufficient
to show that the !-restriction of $\CF'$ to any $S^\lambda_x$ is $\fL(N)_x$-equivariant. However, this follows from 
the fact that the action of $N_{X-(x,y)}$ on $S^\lambda_x$ is transitive. 

\end{proof} 

\sssec{}   \label{sss:jets}

We now prove \propref{p:strong approx}\footnote{In the proof below we will use sheaves on stacks and schemes of infinite-type.
As this may cause a feeling of discomfort, we note that everything can be rephrased by choosing finite level structures and thus
dealing only with stacks locally of finite type.}:

\begin{proof} 

The data of $\{\kappa^{\check\lambda}\}$ that does not have zeroes at $y$
defines a reduction of the $G$-bundle $\CP_G$ to $N$ on the formal neighborhood of $y$.
Thus, we can consider the $\fL^+(N)_y$-torsor over $(\BunNb)^{\on{good}_y}$, denoted 
$$(\BunNb)^{\on{level}_y}$$
that classifies the data $(\CP_G,\kappa,\epsilon)$, where $\epsilon$ is the datum of trivialization of the resulting $N$
bundle on the formal neighborhood of $y$. 

\medskip

The usual regluing procedure defines an action on $(\BunNb)^{\on{level}_y}$ of the group ind-scheme $\fL(N)_y$.
By functoriality, for $\CF=\IC_{\on{glob}}$ or $\CF=\Delta^0_{\on{glob}}$, the pullback of $\CF$ to $(\BunNb)^{\on{level}_y}$
is $\fL(N)_y$-equivariant.

\medskip

We have a commutative diagram
$$
\CD
\ol{S}{}^0_x\times \fL(N)_y  @>{\wt\pi_{x,y}}>>  (\BunNb)^{\on{level}_y} \\ 
@VVV   @VVV   \\
\ol{S}{}^0_x\times S^0_y  @>{\pi_{x,y}}>> (\BunNb)^{\on{good}_y},
\endCD
$$
where the map $\wt\pi_{x,y}$ is $\fL(N)_y$-equivariant. Hence, the pullback of $\pi_{x,y}^!(\CF)$ along
$$\ol{S}{}^0_x\times \fL(N)_y\to \ol{S}{}^0_x\times S^0_y$$
is $\fL(N)_y$-equivariant. This implies that $\pi_{x,y}^!(\CF)$ itself was $\fL(N)_y$-equivariant.

\end{proof}

\ssec{Proof of \propref{p:Zastava and Bun}}  \label{ss:proof Zastava and Bun}

\sssec{}  \label{sss:suff dom}

Let us first assume that $-\mu$ is \emph{sufficiently dominant}, by which we mean that
$\check\alpha(-\mu)>2(g-1)$ for every root $\check\alpha$ of $G$. In this case, the projection
$$\Bun_{B^-}^{-\mu}\to \Bun_G$$
is smooth. 

\medskip

Hence, in this case, the map 
$$\fq:\CZ^\mu\to \BunNb$$
is also smooth, and the assertion is evident.

\medskip

We will reduce the general case to the one above, using the factorization property of the Zastava
spaces over the configuration spaces. 

\sssec{}

For a pair of elements $\mu,\lambda\in -\Lambda^{\on{pos}}$, let
$$(X^\mu\times X^\lambda)_{\on{disj}}\subset X^\mu\times X^\lambda$$
be the open subset corresponding to the locus when the two $-\Lambda^{\on{pos}}$-valued divisors have
disjoint support.

\medskip

The \emph{factorization property} (see \cite[Proposition 2.4]{BFGM}) says that there is a canonical isomorphism
\begin{equation} \label{e:factorization}
\CZ^{\mu+\lambda}\underset{X^{\mu+\lambda}}\times (X^\mu\times X^\lambda)_{\on{disj}}\simeq
(\CZ^\mu\times \CZ^\lambda)\underset{X^\mu\times X^\lambda}\times (X^\mu\times X^\lambda)_{\on{disj}}.
\end{equation}

\sssec{}

In what follows we will need a particular property of the isomorphism \eqref{e:factorization} that follows from 
its definition. 

\medskip

Let $(\BunNb\times X^\lambda)^{\on{good}}$ be the open subset of the product
$$\BunNb\times X^\lambda,$$
where we forbid the maps $\kappa^{\check\lambda}$ to have a zero at the support of the point of $X^\lambda$. 

\medskip

As in \secref{sss:jets}, we consider the group-scheme $\fL^+(N)_{X^\lambda}$ (over $X^\lambda$), the group ind-scheme $\fL(N)_{X^\lambda}$, and 
the $\fL^+(N)_{X^\lambda}$-torsor 
$$(\BunNb\times X^\lambda)^{\on{level}}\to (\BunNb\times X^\lambda)^{\on{good}}.$$

The action of $\fL^+(N)_{X^\lambda}$ on $(\BunNb\times X^\lambda)^{\on{level}}$ extends to that of
$\fL(N)_{X^\lambda}$. We fix a group subscheme $N'_{X^\lambda}$
$$\fL^+(N)_{X^\lambda} \subset N'_{X^\lambda}\subset \fL(N)_{X^\lambda},$$
pro-smooth over $X^\lambda$, 
and consider the quotient stack $(\BunNb\times X^\lambda)^{\on{level}}/N'_{X^\lambda}$. It comes equipped with a smooth projection
$$(\BunNb\times X^\lambda)^{\on{good}}\to 
(\BunNb\times X^\lambda)^{\on{level}}/N'_{X^\lambda}.$$

\medskip

Then for $N'_{X^\lambda}$ large enough the following diagram is commutative:
\begin{equation} \label{e:Zast com diag}
\CD
(\CZ^\mu\times \oCZ{}^\lambda)\underset{X^\mu\times X^\lambda}\times (X^\mu\times X^\lambda)_{\on{disj}}   
@>{\text{\eqref{e:factorization}}}>> \CZ^{\mu+\lambda}\underset{X^{\mu+\lambda}}\times (X^\mu\times X^\lambda)_{\on{disj}}  \\
@V{\on{id}\times \fs}VV   @VV{\fq}V   \\
(\CZ^\mu\times X^\lambda)\underset{X^\mu\times X^\lambda}\times (X^\mu\times X^\lambda)_{\on{disj}} 
& &    (\BunNb\times X^\lambda)^{\on{good}}   \\
@V{\fq}VV   @VVV     \\
(\BunNb\times X^\lambda)^{\on{good}}   @>>>  (\BunNb\times X^\lambda)^{\on{level}}/N'_{X^\lambda}.
\endCD
\end{equation}

The idea is that a point $z\in \oCZ{}^\lambda$ modifies the (generalized) $N$-bundle at the points of the support
of the divisor $\fs(z)$. 

\sssec{}

Now, given any $\mu\in -\Lambda^{\on{pos}}$, we can find $\lambda\in -\Lambda^{\on{pos}}$ so that
$-(\mu+\lambda)$ is sufficiently dominant (as in \secref{sss:suff dom}), so that
$$\CZ^{\mu+\lambda}\to \BunNb$$
is smooth. 

\medskip

The assertion of \propref{p:Zastava and Bun} follows by chasing over the diagram \eqref{e:Zast com diag}. 
For example, point (a) is obtained as follows:

\medskip

It suffices to show that the pullback of the IC sheaf along the composite left vertical map is isomorphic to the IC sheaf
(up to a cohomological shift). 

\medskip

Since the bottom horizontal arrow in \eqref{e:Zast com diag} is smooth, it suffices to show that the pullback
of $\IC_{(\BunNb\times X^\lambda)^{\on{level}}/N'_{X^\lambda}}$ 
along the counter-clockwise circuit in \eqref{e:Zast com diag} is isomorphic to the IC sheaf
(up to a cohomological shift). 

\medskip

Since the diagram \eqref{e:Zast com diag} is commutative, this is equivalent to showing that the pullback
of $\IC_{(\BunNb\times X^\lambda)^{\on{level}}/N'_{X^\lambda}}$ 
along the clockwise circuit in \eqref{e:Zast com diag} is isomorphic to the IC sheaf
(up to a cohomological shift). 

\medskip

Since the top horizontal arrow and the lower right vertical arrows in \eqref{e:Zast com diag} are smooth, it suffices 
to show that the pullback of the IC sheaf along
$$\CZ^{\mu+\lambda}\underset{X^{\mu+\lambda}}\times (X^\mu\times X^\lambda)_{\on{disj}}\to 
(\BunNb\times X^\lambda)^{\on{good}}$$
is isomorphic to the IC sheaf
(up to a cohomological shift). 

\medskip

Since $X^\lambda$ is smooth, it suffices to show that the pullback of the IC sheaf along the composite map
$$\CZ^{\mu+\lambda}\underset{X^{\mu+\lambda}}\times (X^\mu\times X^\lambda)_{\on{disj}}\to 
(\BunNb\times X^\lambda)^{\on{good}}\to \BunNb$$
is isomorphic to the IC sheaf (up to a cohomological shift). But this follows from the fact that the map
$$\CZ^{\mu+\lambda} \overset{\fq}\longrightarrow \BunNb$$
is smooth.

\section{Digression: (dual) baby Verma objects}  

In this section we summarize the construction of (dual) baby Verma objects, following \cite{ABBGM} and \cite{FG}. 
These are objects of the category $\Shv(\Gr_G)^I$ that have a particular property (the Hecke property) with respect
to convolutions with objects of the form $\Sat(V)$, $V\in \Rep(\cG)$. 

\ssec{The Iwahori category on the affine Grassmannian}

\sssec{}

Consider the category
$$\Shv(\Gr_G)^I,$$
where $I\subset \fL^+(G)$ is the Iwahori 
subgroup\footnote{We will use a slightly renormalized version of $\Shv(\Gr_G)^I$, where we declare compact objects 
to be the ones that map to compact objects under the forgetful functor 
$\Shv(\Gr_G)^I\to \Shv(\Gr_G)$. The same applies to $\Shv(\Fl_G)^I$.
This is done in order to avoid the singular support condition on coherent sheaves on the spectral side.
We are grateful to J.~Campbell for catching this imprecision.}.

\medskip

This category carries an action by right convolutions by $\Sph(G)=\Shv(\Gr_G)^{\fL^+(G)}$ and
a commuting action of
$$\CH(G):=\Shv(\Fl_G)^I$$
by left convolutions. 

\sssec{}

For an element $\wt{w}$ of the extended affine Weyl group $W^{\on{aff}}$, we let
$$j_{\wt{w},!} \text{ and } j_{\wt{w},*}$$
denote the corresponding standard and costandard objects in $\CH(G)^\heartsuit$.

\medskip

The key fact is that there are canonical isomorphisms
\begin{equation} \label{e:j !*}
j_{\wt{w},!} \star j_{\wt{w}^{-1},*} \simeq \delta_{1,\Fl_G}\simeq  j_{\wt{w}^{-1},*}\star j_{\wt{w},!}.
\end{equation}

\sssec{}

Another crucial observation that there are symmetric monoidal functors
$$\Lambda\rightrightarrows \CH(G)^\heartsuit,$$
uniquely characterized by the property that they send 
$$(\lambda\in \Lambda^+)\mapsto j_{\lambda,*} \text{ and } (\lambda\in \Lambda^+)\mapsto j_{\lambda,!},$$
respectively.

\medskip

Using \eqref{e:j !*}, we obtain that the first functor sends $\lambda\in -\Lambda^+$ to $j_{\lambda,!}$ and
the second functor sends $\lambda\in -\Lambda^+$ to $j_{\lambda,*}$.

\medskip

The above two symmetric monoidal functors are intertwined by the automorphism induced by the action of $w_0\in W$
and the automorphism of $\CH(G)$ given by
$$\CF\mapsto j_{w_0,!}\star \CF\star j_{w_0,*}.$$

Indeed, for $\lambda\in \Lambda^+$ we have:
$$j_{w_0,!}\star j_{\lambda,* }\star j_{w_0,*}\simeq j_{w_0(\lambda),*}.$$

\medskip

In particular, the auto-equivalence of $\Shv(\Gr_G)^I$, given by
$$\CF\mapsto j_{w_0,!}\star \CF$$
intertwines the $\Rep(\cT)$-action on $\Shv(\Gr_G)^I$ given 
\begin{equation} \label{e:good action}
\sfe^\lambda\star \CF:=j_{\lambda,*}\star \CF, \quad \lambda\in \Lambda^+
\end{equation}
and the action given by
\begin{equation} \label{e:bad action}
\sfe^\lambda\star \CF:=j_{w_0(\lambda),*}\star \CF, \quad \lambda\in \Lambda^+.
\end{equation} 

\ssec{Recollections on the [ABG] theory}

For the remainder of this paper, we will change our conventions, and for a group $H$, we let $\Rep(H)$ denote
the symmetric monoidal DG category of its representation (rather than the corresponding abelian category). 

\sssec{}

Consider the derived stack
$$\cn^-\underset{\cg}\times \{0\}/\cB^-.$$

Note that its underlying classical stack is $\on{pt}/\cB^-$.

\sssec{}

The following theorem is established in \cite{ABG}:

\begin{thm} \label{t:ABG}
There exists a canonically defined equivalence of categories
$$\Sat^I:\IndCoh(\cn^-\underset{\cg}\times \{0\}/\cB^-)\simeq \Shv(\Gr_G)^I$$
with the following properties:

\smallskip

\noindent{\em(i)} The action of $\Rep(\cG)$ on $\IndCoh(\cn^-\underset{\cg}\times \{0\}/\cB^-)$
arising from the projection 
$$\cn^-\underset{\cg}\times \{0\}/\cB^-\to \on{pt}/\cB^-\to \on{pt}/\cG$$
corresponds to the action of $\Rep(\cG)$ on $\Shv(\Gr_G)^I$ via $\Sat:\Rep(\cG)\to \Sph(G)$
and right convolutions. 

\smallskip

\noindent{\em(ii)} The action of $\Rep(\cT)$ on $\IndCoh(\cn^-\underset{\cg}\times \{0\}/\cB^-)$
arising from the projection 
$$\cn^-\underset{\cg}\times \{0\}/\cB^-\to \on{pt}/\cB^-\to \on{pt}/\cT$$
corresponds to the action on $\Rep(\cG)$ on $\Shv(\Gr_G)^I$ given by \eqref{e:good action}.

\smallskip

\noindent{\em(iii)} The object
$$\CO_{\on{pt}/\cB^-}\in \IndCoh(\cn^-\underset{\cg}\times \{0\}/\cB^-)$$
corresponds under $\Sat^I$ to $\delta_{1,\Gr_G}\in \Shv(\Gr_G)^I$.

\smallskip

\noindent{\em(iv)}
For $\lambda\in \Lambda^+$ the morphism 
$$(v^\lambda)^*:V^\lambda\to \sfe^\lambda$$
in $\Rep(\cB^-)$ corresponds under $\Sat^I$ to the natural map of perverse sheaves
$$\IC_{\ol\Gr^\lambda_G}\to j_{\lambda,*}\star \delta_{1,\Gr_G}.$$

\end{thm}

\sssec{}  \label{sss:fibers on Gr}

Let us use \thmref{t:ABG} to supply the calculation for the proof of \propref{p:! fibers of IC} that
$$\underset{\lambda\in \Lambda^+}{\on{colim}}\, H_{\{t^{\lambda+\mu}\}}(\Gr_G,\IC_{\ol\Gr^\lambda_G})[\langle \lambda+\mu,2\check\rho\rangle]$$
identifies canonically with $\Sym^n(\cg/\cb^-[-2])(-\mu)$.

\begin{proof}

By Verdier duality, the !-fiber of $\Sat(V^\lambda)[\langle \lambda+\mu,2\check\rho\rangle]$ at $t^{\lambda+\mu}$ identifies with 
\begin{equation} \label{e:stable !-fibers equiv}
\CHom_{\Shv(\Gr_G)^I}(\IC_{\ol\Gr^\lambda_G},j_{\lambda+\mu,*}\star \delta_{1,\Gr_G}),
\end{equation}
tensored with $\sfe$ over $\Sym(\ct[-2])$, the latter being 
the $I$-equivariant cohomology of the point.

\medskip

Using \thmref{t:ABG}, we rewrite the expression in \eqref{e:stable !-fibers equiv} as 
\begin{equation} \label{e:stable !-fibers coh}
\CHom_{\IndCoh(\cn^-\underset{\cg}\times \{0\}/\cB^-)}(\Res^{\cG}_{\cB^-}(V^\lambda)\otimes \CO_{\on{pt}/\cB^-},
\Res^{\cT}_{\cB^-}(\sfe^{\lambda+\mu})\otimes \CO_{\on{pt}/\cB^-}).
\end{equation}

\medskip

Moreover, by unwinding the definitions, we obtain that the transition map for $\lambda_2=\lambda_1+\lambda$ 
$$H_{\{t^{\lambda_1+\mu}\}}(\Gr_G,\Sat(V^{\lambda_1})[\langle \lambda_1+\mu,2\check\rho\rangle]\to
H_{\{t^{\lambda_2+\mu}\}}(\Gr_G,\Sat(V^{\lambda_2})[\langle \lambda_2+\mu,2\check\rho\rangle]$$
is induced by the map
\begin{multline*}
\CHom_{\IndCoh(\cn^-\underset{\cg}\times \{0\}/\cB^-)}(\Res^{\cG}_{\cB^-}(V^{\lambda_1})\otimes \CO_{\on{pt}/\cB^-},
\Res^{\cT}_{\cB^-}(\sfe^{\lambda_1+\mu})\otimes \CO_{\on{pt}/\cB^-})\to \\
\to \CHom_{\IndCoh(\cn^-\underset{\cg}\times \{0\}/\cB^-)}(\Res^{\cG}_{\cB^-}(V^{\lambda_1}\otimes V^\lambda)
\otimes \CO_{\on{pt}/\cB^-}, \Res^{\cT}_{\cB^-}(\sfe^{\lambda_1+\mu})\otimes \Res^{\cG}_{\cB^-}(V^\lambda) \otimes \CO_{\on{pt}/\cB^-}) \to \\
\to
\CHom_{\IndCoh(\cn^-\underset{\cg}\times \{0\}/\cB^-)}(\Res^{\cG}_{\cB^-}(V^{\lambda_2})\otimes \CO_{\on{pt}/\cB^-},
\Res^{\cT}_{\cB^-}(\sfe^{\lambda_2})\otimes \CO_{\on{pt}/\cB^-}),
\end{multline*}
where second arrow is induced by the maps
$$V^{\lambda_2}\to V^{\lambda_1}\otimes V^\lambda$$
and
$$\Res^{\cG}_{\cB^-}(V^\lambda)\to \Res^{\cT}_{\cB^-}(\sfe^{\lambda}).$$

\medskip

We rewrite \eqref{e:stable !-fibers coh} as
\begin{equation} \label{e:stable !-fibers ind}
\CHom_{\Rep(\cG)}(V^\lambda,\on{coInd}_{\cB^-}^\cG(\sfe^{\lambda+\mu}\otimes \Sym(\cg/\cn^-[-2]))),
\end{equation}
and the transition maps are given by
\begin{multline*}
\CHom_{\Rep(\cG)}(V^{\lambda_1},\on{coInd}_{\cB^-}^\cG(\sfe^{\lambda_1+\mu}\otimes \Sym(\cg/\cn^-[-2]))) \to \\
\to \CHom_{\Rep(\cG)}(V^{\lambda_1}\otimes V^\lambda,\on{coInd}_{\cB^-}^\cG(\sfe^{\lambda_1+\mu}\otimes \Sym(\cg/\cn^-[-2]))\otimes V^\lambda)\simeq \\
\simeq 
\CHom_{\Rep(\cG)}(V^{\lambda_1}\otimes V^\lambda,\on{coInd}_{\cB^-}^\cG(\sfe^{\lambda_1+\mu}\otimes \Sym(\cg/\cn^-[-2])\otimes 
\Res^\cG_{\cB^-}(V)^\lambda)) \to \\
\to \CHom_{\Rep(\cG)}(V^{\lambda_2},\on{coInd}_{\cB^-}^\cG(\sfe^{\lambda_2+\mu}\otimes \Sym(\cg/\cn^-[-2]))).
\end{multline*}

Finally, we claim that the colimit of the expressions \eqref{e:stable !-fibers ind} over $\lambda\in \Lambda^+$ identifies with
$$\Sym(\cg/\cn^-[-2])(-\mu).$$

Indeed, for a given integer $m\geq 0$, let $\lambda_0$ be such that for every weight $\nu$ that appears in 
$\Sym^n(\cg/\cn^-)$, we have $\lambda_0+\nu\in \Lambda^+$. Then for all $\lambda\in \lambda_0+\Lambda^+$, we have
$$\CHom_{\Rep(\cG)}(V^\lambda,\on{coInd}_{\cB^-}^\cG(\sfe^{\lambda+\mu}\otimes \Sym^n(\cg/\cn^-[-2])))\simeq
\Sym^n(\cg/\cn^-[-2])(-\mu).$$

\end{proof}

\ssec{Hecke patterns}

\sssec{}

Let $\CC$ be a DG category, acted on by $\Rep(M)$, where $M$ as an algebraic group. We denote by
$\Hecke_M(\CC)$ the corresponding Hecke category
$$\Hecke_M(\CC):=\CC\underset{\Rep(M)}\otimes \Vect.$$

\medskip

We have a tautological functor
$$\ind_{\Hecke_M}:\CC\to \Hecke_M(\CC),$$
which admits a continuous right adjoint, denoted $\oblv_{\Hecke_M}$. The comonad
$\oblv_{\Hecke_M}\circ \ind_{\Hecke_M}$ on $\CC$ is given by the action of the left regular representation object $\CO(M)\in \Rep(M)$.

\sssec{}

Since $\Rep(M)$ is \emph{rigid} (see \cite[Chapter 1]{GR}), we an canonically identify
$$\Hecke_M(\CC):=\CC\underset{\Rep(M)}\otimes \Vect\simeq  \on{Funct}_{\Rep(M)}(\Vect,\CC),$$
see Proposition 9.4.8 in {\it loc.cit.}

\medskip

Under this identification, the functor $\oblv_{\Hecke_M}$ corresponds to the forgetful functor
$$\on{Funct}_{\Rep(M)}(\Vect,\CC)\to \on{Funct}_{\Rep(M)}(\Rep(M),\CC)\simeq \CC.$$

\medskip

This point of view allows to think of objects of $\Hecke_M(\CC)$ as ``Hecke eigen-objects'': these are objects $c\in \CC$ equipped
with a system of isomorphisms
$$c\star V\simeq \ul{V}\otimes c, \quad V\in \Rep(M)$$
(here $\ul{V}$ denotes the vector space underlying $V$) that are associative in the same sense as in \secref{sss:central}.

\medskip

The functor $\ind_{\Hecke_M}$ sends $c\in \CC$ to $c\star \CO(M)$, with the Hecke structure induced by that on $\CO(M)$:
$$\CO(M)\otimes V\simeq \ul{V}\otimes \CO(M).$$

\sssec{}

We apply this for $M$ being $\cG$ or $\cT$. In addition, we will use the following two variants.

\medskip

For $\CC$ equipped with an action of $\Rep(\cG)$ we will denote by $\bHecke_{\cG}(\CC)$ the category
$$\CC\underset{\Rep(\cG)}\otimes \Rep(\cT).$$

\medskip

The category $\bHecke_{\cG}(\CC)$
is acted on by $\Rep(\cT)$ and is related to $\Hecke_{\cG}(\CC)$ by the formula
$$\Hecke_{\cG}(\CC)=\Hecke_{\cT}(\CC)(\bHecke_{\cG}(\CC)).$$

We have the corresponding pair of adjoint functors 
$$\ind_{\Hecke_\cT}:\bHecke_{\cG}(\CC)\rightleftarrows \Hecke_{\cG}(\CC):\oblv_{\Hecke_\cT}.$$

\sssec{}

We can think of $\bHecke_{\cG}(\CC)$ as the category of $\Lambda$-graded Hecke eigen-objects. I.e., an object of 
$\bHecke_{\cG}(\CC)$ is a collection of objects $\{c_\lambda\in \CC, \lambda\in \Lambda\}$, equipped with a system 
of isomorphisms
\begin{equation} \label{e:graded Hecke data}
c_\lambda\star V\simeq \underset{\mu}\bigoplus\, V(\mu)\otimes c_{\lambda+\mu}, \quad V\in \Rep(\cG)
\end{equation}
that are associative in the same sense as in \secref{sss:central}.

\medskip

In terms of this description, the functor $\ind_{\Hecke_T}:\bHecke_{\cG}(\CC)\to\Hecke_{\cG}(\CC)$ sends
$$\{c_\lambda\}\mapsto \underset{\lambda}\oplus\, c_\lambda,$$
and the functor 
$$\oblv_{\Hecke_T}:\Hecke_{\cG}(\CC)\to \bHecke_{\cG}(\CC)$$ 
sends $c$ to
$$\{c_\lambda\},\,\, c_\lambda=c \text{ for all }\lambda.$$

\sssec{}  \label{sss:twisted Hecke}

Suppose now that $\CC$ carries an action of $\Rep(\cT)\otimes \Rep(\cG)$. We will consider the category
$$\Hecke_{\cG,\cT}(\CC):=\CC\underset{\Rep(\cT)\otimes \Rep(\cG)}\otimes \Rep(\cT).$$

\medskip

We can think of its objects as $c\in \CC$, equipped with a collection of isomorphisms 
\begin{equation} \label{e:twisted Hecke}
c\star V\simeq \Res^\cG_\cT(V)\star c, \quad V\in \Rep(\cG)
\end{equation}
(here convolution on the right denotes the action of $\Rep(\cG)$ and convolution on the left denotes the action of $\Rep(\cT)$),
that are associative in the same sense as in \secref{sss:central}.

\sssec{}

We have the tautological functor
$$\Phi:\bHecke_{\cG}(\CC)=\CC\underset{\Rep(\cG)}\otimes \Rep(\cT)\to 
\CC\underset{\Rep(\cT)\otimes \Rep(\cG)}\otimes \Rep(\cT)=\Hecke_{\cG,\cT}(\CC),$$
which admits a continuous right adjoint
\begin{equation} \label{e:right adj to Hecke}
\Psi:\Hecke_{\cG,\cT}(\CC)\to \bHecke_{\cG}(\CC).
\end{equation}

Explicitly, for an object $c$ of $\Hecke_{\cG,\cT}(\CC)$ as in \secref{sss:twisted Hecke}, the object 
$\Psi(c)\in \bHecke_{\cG}(\CC)$ consists of
$$c_\lambda=\sfe^\lambda\star c,$$
with the Hecke structure supplied by \eqref{e:twisted Hecke}.

\sssec{}

In addition, the functor
$$\ind_{\Hecke_{\cT}}:\CC\to \Hecke_{\cT}(\CC)$$
induces a functor
$$\Hecke_{\cG,\cT}(\ind_{\Hecke_{\cT}}):\Hecke_{\cG,\cT}(\CC) \to \Hecke_{\cG,\cT}(\Hecke_{\cT}(\CC))\simeq
\Hecke_\cG(\Hecke_{\cT}(\CC)).$$

Explicitly, for an object $c$ of $\Hecke_{\cG,\cT}(\CC)$, the corresponding underlying object of $\Hecke_{\cT}(\CC)$
is 
$$\ind_{\Hecke_{\cT}}(c):=\underset{\lambda\in \Lambda}\oplus\, \sfe^\lambda\star c,$$
with the Hecke structure with respect to $\cG$ given by \eqref{e:twisted Hecke}.

\sssec{}

We have a commutative diagram:
\begin{equation} \label{e:Hecke diagram}
\CD
\Hecke_{\cG,\cT}(\CC)     @>{\Psi}>> \bHecke_{\cG}(\CC)   \\
@V{\Hecke_{\cG,\cT}(\ind_{\Hecke_{\cT}})}VV    @VV{\ind_{\Hecke_{\cT}}}V   \\
\Hecke_\cG(\Hecke_{\cT}(\CC))   @>{\Hecke_\cG(\oblv_{\Hecke_{\cT}})}>>  \Hecke_{\cG}(\CC). 
\endCD
\end{equation} 

\ssec{The (dual) baby Verma object: coherent side}

\sssec{}

Consider $\IndCoh(\cn^-\underset{\cg}\times \{0\}/\cB^-)$ as a category endowed with a pair of commuting actions
of $\Rep(\cG)$ and $\Rep(\cT)$, obtained from the maps
$$\cn^-\underset{\cg}\times \{0\}/\cB^-\to \on{pt}/\cB^-$$
and 
$$\on{pt}/\cB^-\to \on{pt}/\cG \text{ and } \on{pt}/\cB^-\to \on{pt}/\cT,$$
respectively. 

\medskip

Consider the corresponding categories
$$\Hecke_\cG(\IndCoh(\cn^-\underset{\cg}\times \{0\}/\cB^-)), \,\,
\bHecke_\cG(\IndCoh(\cn^-\underset{\cg}\times \{0\}/\cB^-)),$$
$$\Hecke_\cT(\IndCoh(\cn^-\underset{\cg}\times \{0\}/\cB^-))
\text{ and }
\Hecke_{\cG,\cT}(\IndCoh(\cn^-\underset{\cg}\times \{0\}/\cB^-)).$$

Direct image along the closed embedding $\on{pt}/\cB^-\to \cn^-\underset{\cg}\times \{0\}/\cB^-$ defines functors
\begin{equation} \label{e:flags into usual}
\QCoh(\cG/\cB^-)\simeq \Hecke_\cG(\QCoh(\on{pt}/\cB^-))\to \Hecke_\cG(\IndCoh(\cn^-\underset{\cg}\times \{0\}/\cB^-));
\end{equation}
\begin{equation} \label{e:flags into graded}
\QCoh(\cT\backslash \cG/\cB^-)\simeq 
\bHecke_\cG(\QCoh(\on{pt}/\cB^-))\to \bHecke_\cG(\IndCoh(\cn^-\underset{\cg}\times \{0\}/\cB^-));
\end{equation}
\begin{equation} \label{e:flags into G/N}
\QCoh(\on{pt}/\cN^-)\simeq 
\Hecke_\cT(\QCoh(\on{pt}/\cB^-))\to \Hecke_\cT(\IndCoh(\cn^-\underset{\cg}\times \{0\}/\cB^-))
\end{equation}
and
\begin{equation} \label{e:flags into interesting}
\QCoh((\cG/\cN^-)/\on{Ad}_{\cT})\simeq \Hecke_{\cG,\cT}(\QCoh(\on{pt}/\cB^-))\to \Hecke_{\cG,\cT}(\IndCoh(\cn^-\underset{\cg}\times \{0\}/\cB^-)).
\end{equation}

\medskip

The diagram \eqref{e:Hecke diagram} is induced by the commutative diagram
$$
\CD
\QCoh((\cG/\cN^-)/\on{Ad}_{\cT}) @>>>  \QCoh(\cT\backslash \cG/\cB^-)  \\
@VVV  @VVV  \\
\QCoh(\cG/\cN^-) @>>> \QCoh(\cG/\cB^-),
\endCD
$$
where we take inverse images along the vertical arrows and direct images along the horizontal arrows.

\sssec{}  \label{sss:coherent Verma}

The inclusion $\cN^-\hookrightarrow \cG$ induces a map
\begin{equation} \label{e:key geom map}
\on{pt}/\cT\to (\cG/\cN^-)/\on{Ad}_{\cT}.
\end{equation} 

Taking the direct image of $\CO_{\on{pt}/\cT}$ along this map and applying the functor \eqref{e:flags into interesting} we obtain
an object that we denote
$$\CM_{\cG,\cT}\in \Hecke_{\cG,\cT}(\IndCoh(\cn^-\underset{\cg}\times \{0\}/\cB^-)).$$

\medskip

Note that its image under the functor $\Psi$ \eqref{e:right adj to Hecke} is the object of $\bHecke_\cG(\IndCoh(\cn^-\underset{\cg}\times \{0\}/\cB^-))$,
to be denoted $\bCM_\cG$, obtained by means of \eqref{e:flags into graded} from the direct image of $\CO_{\on{pt}/\cT}$ under the map
$$\on{pt}/\cT\to \cT\backslash \cG/\cB^-.$$

\medskip

Finally, the object
$$\CM_\cG:=\ind_{\Hecke_T}(\bCM)\in \Hecke_\cG(\IndCoh(\cn^-\underset{\cg}\times \{0\}/\cB^-))$$
is obtained by means of \eqref{e:flags into usual} from the sky-scraper
$$\sfe_{1,\cG/\cB^-}\in \QCoh(\cG/\cB^-).$$



\medskip

The objects $\CM_{\cG,\cT}$, $\CM_\cG$ and $\bCM_\cG$ are the various incarnations of what we call the ``(dual) baby Verma object'';
the origin of the name will be explained shortly.

\sssec{}

For future use, let us describe explicitly the object of $\Hecke_{\cG,\cT}(\QCoh(\on{pt}/\cB^-))$
from which we obtained $\CM_{\cG,\cT}$ as a direct image in terms of \secref{sss:twisted Hecke}. Unwinding the 
definition, we obtain that as an object of $\Rep(\cB^-)$ it identifies with $\CO(\cB^-/\cT)$. The isomorphisms \eqref{e:twisted Hecke}
are given as follows:

\medskip

For any $W\in \Rep(\cB^-)$, we have a canonical isomorphism
\begin{align*} 
&W\otimes \CO(\cB^-/\cT)=\\
&=(W\otimes \CO(\cB^-))^{\cT_{\on{right}}}\simeq \\
&\simeq (\CO(\cB^-) \otimes W)^{\cT_{\on{right-diag}}}\simeq \\
&\simeq \underset{\mu}\oplus\, \CO(\cB^-)(-\mu) \otimes W(\mu) \simeq \\
&\simeq \underset{\mu}\oplus\,  (\CO(\cB^-/\cT)\otimes \sfe^\mu) \otimes W(\mu)=\\
&=\CO(\cB^-/\cT)\otimes  \Res^{\cB^-}_\cT(W),
\end{align*}
where:

\smallskip

\noindent--In the second line we view $W\otimes \CO(\cB^-)$ as an object of $\Rep(\cB^-)$ 
diagonally via the given action on $W$ and an action on $\CO(\cB^-)$ by left translations, and as such 
acted on by $\cT$ via right translations on the $\CO(\cB^-)$-factor;

\smallskip

\noindent--In the third line we view $\CO(\cB^-) \otimes W$ as an object of $\Rep(\cB^-)$ 
via the action of $\cB^-$ by left translations on the $\CO(\cB^-)$-factor, and as such acted on 
diagonally by $\cT$ via $\cT\to \cB^-$ and right translations on the $\CO(\cB^-)$-factor and the
given action on $W$;

\smallskip

\noindent--In the fourth line $\CO(\cB^-)(-\mu)$ means the $(-\mu)$ weight space with respect to
the action of $\cT$ via $\cT\to \cB^-$ and right translations;
\smallskip

\noindent--The isomorphism between the fourth and the fifth lines is induced by the identification
$$\CO(\cB^-)(-\mu)\simeq \CO(\cB^-/\cT)\otimes \sfe^\mu,$$
given by multiplication by the character $\cB^-\to \cT \overset{\mu}\to \BG_m$.

\sssec{}  \label{sss:coherent b Verma expl}

We note that the above object of $\Hecke_{\cG,\cT}(\QCoh(\on{pt}/\cB^-))$ can be written
as
$$\underset{\lambda\in \Lambda^+}{\on{colim}}\, \sfe^{\lambda}\otimes \Res^{\cG}_{\cB^-}((V^\lambda)^*),$$
where the transition maps for $\lambda_2=\lambda_1+\lambda$ are given as follows
\begin{multline*} 
\sfe^{\lambda_1}\otimes \Res^{\cG}_{\cB^-}((V^{\lambda_1})^*)\simeq 
\sfe^{\lambda_1}\otimes \sfe^{\lambda}\otimes \sfe^{-\lambda}\otimes \Res^{\cG}_{\cB^-}((V^{\lambda_1})^*) \to \\
\to \sfe^{\lambda_2}\otimes \Res^{\cG}_{\cB^-}((V^\lambda)^*) \otimes \Res^{\cG}_{\cB^-}((V^{\lambda_1})^*) \to
\sfe^{\lambda_2}\otimes \Res^{\cG}_{\cB^-}((V^{\lambda_2})^*), 
\end{multline*}
where the second arrow is given by the identification $\sfe^{-\lambda}\simeq ((V^\lambda)^*)^{\cN^-}$, and the last 
arrow by \eqref{e:Plucker map dual}. 

\medskip

The Hecke structure is given as follows. For a given finite-dimensional $V\in \Rep(\cG)^\heartsuit$ and $\lambda\gg 0$,
we start with the isomorphism \eqref{e:razval initial}, from which we produce 
a canonical isomorphism
\begin{equation} \label{e:razval}
(V^\lambda)^*\otimes V\simeq \underset{\mu}\oplus\, (V^{\lambda+\mu})^*\otimes V(-\mu).
\end{equation} 

\medskip

Hence, for a fixed $V$ as above and $\lambda\gg 0$, we have
$$\sfe^{\lambda}\otimes \Res^{\cG}_{\cB^-}((V^\lambda)^*)\otimes \Res^{\cG}_{\cB^-}(V)\simeq
\underset{\mu}\oplus\, \left(\sfe^{-\mu}\otimes V(-\mu)\right)\otimes \left(\sfe^{\lambda+\mu}\otimes \Res^{\cG}_{\cB^-}((V^{\lambda+\mu})^*)\right),$$
as required.

\begin{rem}
In \secref{ss:Drinfeld-Plucker} we will give a more conceptual explanation of the above presentation the above object of 
$\Hecke_{\cG,\cT}(\QCoh(\on{pt}/\cB^-))$.
\end{rem} 

\ssec{The geometric (dual) baby Verma object}  \label{ss:geom Verma}

\sssec{}

We now consider the category $\Shv(\Gr_G)^I$, equipped with the action of $\Rep(\cG)$ via $\Sat$ and a commuting
action of $\Rep(\cT)$ given by \eqref{e:good action}. 

\medskip

Applying the equivalence $\Sat^I$ of \thmref{t:ABG}, from the objects $\CM_{\cG,\cT}$, $\CM_\cG$ and $\bCM_\cG$
we obtain the objects
$$\CF_{\cG,\cT}\in \Hecke_{\cG,\cT}(\Shv(\Gr_G)^I);$$
$$\bCF_\cG\in \bHecke_\cG(\Shv(\Gr_G)^I);$$
$$\CF_\cG\in \Hecke_\cG(\Shv(\Gr_G)^I).$$

\begin{rem}
The object $\bCF_\cG$ was introduced in \cite{ABBGM} so that under the equivalence between $\bHecke_\cG(\Shv(\Gr_G)^{I^0})$
and the (regular block of) category of modules over the small quantum group, it corresponds to the \emph{dual baby Verma module}.

\medskip

The object $\CF_\cG$ was introduced in \cite{FG}, so that the functor of global sections of critically twisted D-modules sends
it to the Verma module over affine algebra at the critical level (reduced modulo the ideal of the center that corresponds to
regular opers). 

\end{rem}

\sssec{}  \label{sss:geom Verma colim}

The object $\CF_{\cG,\cT}$ and its derivatives (i.e., $\bCF_\cG$, $\CF_\cG$) can be described explicitly
using \secref{sss:coherent b Verma expl}. Namely, the underlying object of $\Shv(\Gr_G)^I$ is 
$$\underset{\lambda\in \Lambda^+}{\on{colim}}\, j_{\lambda,*}\star \Sat((V^\lambda)^*),$$
where the transition maps for $\lambda_2=\lambda_1+\lambda$ are given by
\begin{multline*}
 j_{\lambda_1,*}\star \Sat((V^{\lambda_1})^*)\to  j_{\lambda_1,*}\star  j_{\lambda,*}\star  j_{-\lambda,!}\star 
\Sat((V^{\lambda_1})^*)\to \\
\to j_{\lambda_2,*}\star \Sat((V^\lambda)^*)\star \Sat((V^{\lambda_1})^*)\to
j_{\lambda_2,*}\star  \Sat((V^{\lambda_2})^*),
\end{multline*}
where the second arrow comes from the (monoidal) dual of the canonical map $\IC_{\ol\Gr^\lambda_G}\to j_{\lambda,*}\star \delta_{1,\Gr_G}$,
and the last arrow is given by \eqref{e:Plucker map dual}. 

\medskip

The Hecke structure is given by identifying for a fixed finite-dimensional $V\in \Rep(G)^\heartsuit$ and $\lambda\gg 0$ 
\begin{multline*} 
j_{\lambda,*}\star \Sat((V^\lambda)^*)\star \Sat(V)\simeq j_{\lambda,*}\star \Sat((V^\lambda)^*\otimes V)\overset{\text{\eqref{e:razval}}}\simeq \\
\simeq 
\underset{\mu}\oplus\, V(-\mu)\otimes j_{\lambda,*}\star \Sat((V^{\lambda+\mu})^*)
\simeq \underset{\mu}\oplus\, \left(J_{-\mu,!}\otimes V(-\mu)\right)\star \left(j_{\lambda+\mu,*}\star \Sat((V^{\lambda+\mu})^*)\right),
\end{multline*} 
where $J_{\nu}$ is the BMW object: for $\nu=\nu_1-\nu_2$ with $\nu_1,\nu_2\in \Lambda^+$,
$$J_\nu:=j_{\nu_1,*}\star j_{-\nu_2,!},$$
and this is canonically independent of the choice of $\nu_1,\nu_2$. 

\begin{rem}
In \secref{ss:Drinfeld-Plucker} we will present a point of view from which the presentation of 
(the object of $\Shv(\Gr_G)^I$ underlying) $\CF_{\cG,\cT}$ as a colimit is automatic, along with its Hecke structure. 
\end{rem}

\ssec{The twist by $w_0$}

\sssec{}

Note that the Weyl group $W$ acts on $\Rep(T)$, preserving the forgetful functor $\Rep(\cG)\to \Rep(\cT)$. 
Hence, its also acts on the category $\bHecke_{\cG}(\CC)$.

\medskip

Explicitly, given an object $\{c_\lambda\}\in \bHecke_{\cG}(\CC)$ and choosing a representative $g_w\in W$ of a given
element $w\in W$, define a new object $\{c^w_\lambda\} \in \bHecke_{\cG}(\CC)$ as follows:

\medskip

We set $c^w_\lambda:=c_{w(\lambda)}$. The Hecke data for $\{c^w_\lambda\}$ is given by the composing the
Hecke data \eqref{e:graded Hecke data} for $\{c_\lambda\}$ and the maps 
$$g_w:V(\mu)\to V(\w(\mu)), \quad V\in \Rep(\cG).$$

\medskip

If we modify the choice of $g_w$ by an element $t\in \cT$, the corresponding two objects $\{c^w_\lambda\}$ 
are isomorphed by means of acting by $\lambda(t)$ on the $\lambda$-component. 

\sssec{}  \label{sss:w conj}

Given an element $w\in W$ we can consider the category $\Hecke_{\cG,\cT,w}(\CC)$, which is defined in the same way
as $\Hecke_{\cG,\cT}(\CC)$, but where we modify the $\Rep(\cT)$-action by applying the automorphism $w$ 
(the restriction functor $\Rep(\cG)\to \Rep(\cT)$ stays intact).

\medskip

Choosing a representative $g_w\in G$ of $w$, there exists a canonical equivalence
$$\Hecke_{\cG,\cT}(\CC)\to \Hecke_{\cG,\cT,w}(\CC).$$

\medskip

Explicitly, given an object of $\Hecke_{\cG,\cT}(\CC)$, the corresponding object of $\Hecke_{\cG,\cT,w}(\CC)$ has the
same underlying object $c\in \CC$, and the new data of \eqref{e:twisted Hecke} is obtained by acting on the old
data of \eqref{e:twisted Hecke} by using $g_w$ to identify 
$$\Res^\cG_\cT \text{ and } \Res^\cG_{\cT,w},$$
where $\Res^\cG_{\cT,w}$ is the composite
$$\Rep(\cG)\overset{\Res^\cG_\cT}\longrightarrow \Rep(\cT) \overset{w}\longrightarrow \Rep(\cT).$$

\sssec{}

Note that the geometric Satake isomorphism provides a canonical representative $g_{w_0}$
of the element $w_0\in W$. Indeed, finding a representative for $w_0$ is equivalent to trivializing the
lower weight line in every $V^\lambda$ in a way compatible with the maps \eqref{e:Plucker map} and \eqref{e:Plucker map dual}.

\medskip

However, by \secref{sss:Satake}, we have
$$V^\lambda(w_0(\lambda))\simeq H_c(S^{w_0(\lambda)},\bi^*_{w_0(\lambda)}(\IC_{\ol\Gr^\lambda_G}))[\langle w_0(\lambda),2\rho\rangle].$$

Now, the intersection $S^{w_0(\lambda)}\cap \ol\Gr^\lambda_G$ consists of a single point $t^{w_0(\lambda)}$, hence
the above cohomology identifies canonically with $\sfe$.

\sssec{}

Note also that a choice of a representative of $w_0$ defines a system of identifications 
\begin{equation} \label{e:identify dual}
V^{-w_0(\lambda)}\simeq (V^\lambda)^*,
\end{equation} 
compatible with the maps \eqref{e:Plucker map} and \eqref{e:Plucker map dual}. The isomorphism is uniquely fixed by the condition that the trivialization
of the highest weight line in $V^{-w_0(\lambda)}$ corresponds to the trivialization of the lowest weight line in $V^\lambda$.

\medskip

The following is easy to verify: 

\begin{lem} \label{l:w_0} 
For any choice of a representative $g_{w_0}$ of $w_0$, a fixed finite-dimensional $V\in \Rep(\cG)$ and $\lambda\gg 0$, the
following diagram commutes:
$$
\CD
(V^\lambda)^*\otimes V @>{\text{\eqref{e:razval}}}>>  \underset{\mu}\oplus\, (V^{\lambda+\mu})^*\otimes V(-\mu) \\
@VVV   @VVV   \\ 
V^{-w_0(\lambda)} \otimes V @>{\sim}>> \underset{\mu}\oplus\, V^{-w_0(\lambda+\mu)}\otimes V(-w_0(\mu)),
\endCD
$$
where the right vertical arrow is given by the maps
$$V(-\mu) \overset{g_{w_0}}\longrightarrow V(-w_0(\mu)).$$
\end{lem}

\sssec{}

Applying the procedure of \secref{sss:w conj} to the object $\CF_{\cG,\cT}$ for the above choice of representative of $w_0$,
we obtain an object, denoted
$$\CF_{\cG,\cT,w_0}\in \Hecke_{\cG,\cT,w_0}(\Shv(\Gr_G)^I).$$

\sssec{}  \label{sss:baby w0}

From 
\lemref{l:w_0}, we obtain that $\CF_{\cG,\cT,w_0}$ can be explicitly described as follows. The underlying object of
$\Shv(\Gr_G)^I$ is given by

\begin{equation} \label{e:describe other b v}
\underset{\lambda\in \Lambda^+}{\on{colim}}\, j_{-w_0(\lambda),*}\star \Sat(V^{\lambda}).
\end{equation}

\medskip

The transition maps for $\lambda_2=\lambda_1+\lambda$ are given by
\begin{multline}  \label{e:Hecke other b v}
j_{-w_0(\lambda_1),*}\star \Sat(V^{\lambda_1})\simeq j_{-w_0(\lambda_1),*}\star j_{-w_0(\lambda),*}\star j_{w_0(\lambda),!}\star
\Sat(V^{\lambda_1}) \to \\
j_{-w_0(\lambda_2),*}\star \Sat(V^{\lambda})\star \Sat(V^{\lambda_1}) \to 
j_{-w_0(\lambda_2),*}\star\Sat(V^{\lambda_w}),
\end{multline} 
where the second arrow comes from the monoidal dual of the map $\IC_{\ol\Gr^\lambda_{-w_0(\lambda)}}\to j_{-w_0(\lambda),*}$
via the identification 
$$\Sat((V^{\lambda})^*)\simeq \Sat(V^{-w_0(\lambda)})\simeq \IC_{\ol\Gr^\lambda_{-w_0(\lambda)}}.$$

\medskip

The Hecke structure is given by the system isomorphisms
\begin{multline*} 
j_{-w_0(\lambda),*}\star \Sat(V^\lambda)\star \Sat(V) \simeq j_{-w_0(\lambda),*}\star \Sat(V^\lambda\otimes V) \simeq \\
\simeq \underset{\mu}\oplus\,  V(\mu)\otimes 
j_{-w_0(\lambda),*}\star \Sat(V^{\lambda+\mu}) \simeq 
\underset{\mu}\oplus\,  \left(j_{w_0(\mu),!}\times V(\mu)\right)\star 
\left(j_{-w_0(\lambda+\mu),*}\star \Sat(V^{\lambda+\mu})\right).
\end{multline*} 

\sssec{}

Finally, we perform one more modification. We replace the action of $\Rep(\cT)$ given by \eqref{e:good action}
by that given by \eqref{e:bad action}. Denote the resulting variant of $\Hecke_{\cG,\cT}(\Shv(\Gr_G)^I)$ by
$\Hecke'_{\cG,\cT}(\Shv(\Gr_G)^I)$. 

\medskip

We note that the functor
$$\CF\mapsto j_{w_0,!}\star \CF$$
defines an equivalence
$$\Hecke_{\cG,\cT,w_0}(\Shv(\Gr_G)^I)\to \Hecke'_{\cG,\cT}(\Shv(\Gr_G)^I).$$

\medskip

We denote by
$$\CF'_{\cG,\cT}:=j_{w_0,!}\star \CF_{\cG,\cT,w_0}$$
the resulting object of $\Hecke'_{\cG,\cT}(\Shv(\Gr_G)^I)$.

\section{The semi-infinite IC sheaf as a (dual) baby Verma object}  \label{s:baby Verma}

In this section we will relate our semi-infinite cohomology sheaf $\ICs$ with the (dual) baby Verma object
from the previous section (in its incarnation as $\CF'_{\cG,\cT}$).

\ssec{Hecke structure on the semi-infinite IC sheaf}  \label{ss:Hecke on IC}

\sssec{}

Let us return to our main object of study, namely, the object
$$\ICs\in \Shv(\Gr_G)^{\fL(N)}.$$

Note that it naturally belongs to the category $\Shv(\Gr_G)^{\fL(N)\cdot \fL^+(T)}$, and we will consider 
it as such\footnote{When forming $\Shv(\Gr_G)^{\fL(N)\cdot \fL^+(T)}$ starting from $\Shv(\Gr_G)^{\fL(N)}$, 
we apply the same renormalization procedure as we did for $\Shv(\Gr_G)^I$.}.  

\sssec{}

We note that the category $\Shv(\Gr_G)^{\fL(N)\cdot \fL^+(T)}$ is acted on by $\Rep(\cT)$ by translations
\begin{equation} \label{e:action by translations}
\sfe^\lambda\star \CF:= t^\lambda\cdot \CF[-\langle \lambda, 2\rho \rangle].
\end{equation} 

In addition, it is acted on by $\Rep(\cG)$ via Geometric Satake and convolutions on the right. So we are in the situation
of \secref{sss:twisted Hecke}. 

\sssec{}

We claim that our object $\ICs\in \Shv(\Gr_G)^{\fL(N)\cdot \fL^+(T)}$ naturally upgrades to an object of
$$\Hecke_{\cG,\cT}(\Shv(\Gr_G)^{\fL(N)\cdot \fL^+(T)}).$$

\sssec{}

In order to construct this structure, we recall that $\ICs$ lives in the heart of the t-structure, and so do the convolutions
$$t^\lambda\cdot \ICs[-\langle \lambda,2\check\rho\rangle] \text{ and } \ICs\star \Sat(V),\,\, V\in \Rep(\cG)^\heartsuit.$$

Hence, when constructing the Hecke structure, we are staying within the abelian category, and it is enough to make
the construction at the level of the homotopy categories.

\medskip

Now, the required Hecke structure on
$$\ICs:=\underset{\lambda\in \Lambda^+}{\on{colim}}\, t^{-\lambda}\cdot \Sat(V^\lambda)[\langle \lambda,2\check\rho\rangle]$$
is given by
\begin{multline}  \label{e:Hecke semiinf}
\left(\underset{\lambda\in \Lambda^+}{\on{colim}}\, t^{-\lambda}\cdot \Sat(V^\lambda)[\langle \lambda,2\check\rho\rangle]\right)\star \Sat(V)\simeq
\left(\underset{\lambda\in \Lambda^+}{\on{colim}}\, t^{-\lambda}\cdot  \Sat(V^\lambda\otimes V)[\langle \lambda,2\check\rho\rangle]\right) 
\overset{\text{\eqref{e:razval initial}}}\simeq \\
\simeq \underset{\mu}\bigoplus\, V(\mu)\otimes 
\left(\underset{\lambda\in \Lambda^+}{\on{colim}}\, t^{-\lambda}\cdot  \Sat(V^{\lambda+\mu})[\langle \lambda,2\check\rho\rangle]\right)
\simeq  \\
\simeq \underset{\mu}\bigoplus\, V(\mu)\otimes 
t^\mu \cdot 
\left(\underset{\lambda\in \Lambda^+}{\on{colim}}\, t^{-\lambda-\mu}\cdot  \Sat(V^{\lambda+\mu})[\langle \lambda+\mu,2\check\rho\rangle]\right)
[-\langle \mu,2\check\rho\rangle].
\end{multline}

The compatibility of these isomorphisms with tensor products in $\Rep(\cG)$ at the level of the homotopy categories is straightforward. 

\ssec{The semi-infinite vs Iwahori equivalence}

The results of this subsection are borrowed from \cite[Theorem 6.2.1]{Ras}. They are applicable for 
$\Shv(\Gr_G)$ replaced by any category acted on by $\fL(G)$. 

\sssec{}

Consider the forgetful functor
\begin{equation} \label{e:forget I equiv}
\Shv(\Gr_G)^I\to \Shv(\Gr_G)^{\fL^+(T)}.
\end{equation} 
 
It admits a right adjoint, given by *-averaging with respect to $I$ ``modulo" $\fL^+(T)$; we denote it by $\on{Av}^{I/\fL^+(T)}_*$. 

\begin{prop}  \label{p:Iwahori N}
The functor $\on{Av}^{I/\fL^+(T)}_*(\CF)$ restricted to
$$(\SI(\Gr_G))^{\fL^+(T)}:=
\Shv(\Gr_G)^{\fL(N)\cdot \fL^+(T)}\subset \Shv(\Gr_G)^{\fL^+(T)}$$
defines an equivalence 
$$(\SI(\Gr_G))^{\fL^+(T)}\to \Shv(\Gr_G)^I.$$
This equivalence intertwines the $\Rep(\cT)$-action on $(\SI(\Gr_G))^{\fL^+(T)}$, given by \eqref{e:action by translations}
and the $\Rep(\cT)$-action on $\Shv(\Gr_G)^I$, given by \eqref{e:bad action}.
\end{prop}

The rest of this subsection is devoted to the proof of this proposition. The argument essentially
mimics the proof of the corresponding statement in the theory of $\fp$-adic groups, due to 
J.~Bernstein. 

\sssec{}

We claim that the functor 
$$\on{Av}^{I/\fL^+(T)}_*(\CF):\Shv(\Gr_G)^{\fL(N)\cdot \fL^+(T)}\to \Shv(\Gr_G)^I$$
admits a left adjoint.

\medskip

More precisely, we claim that the partially defined functor 
$$\on{Av}^{\fL(N)}_!:\Shv(\Gr_G)^{\fL^+(T)}\to \Shv(\Gr_G)^{\fL(N)\cdot  \fL^+(T)}$$
(see \secref{sss:!-av}) is defined on objects lying in the essential image of the forgetful functor \eqref{e:forget I equiv}.
Then it automatically provides the required left adjoint. 

\medskip

To show this, we pick a particular sequence of subgroups $N_\alpha\subset \fL(N)$. Namely, we take the indexing
set to be $\Lambda^+$ (with the order relation from \secref{sss:bizarre order}) and we set
$$N_\lambda:=\on{Ad}_{t^{-\lambda}}(\fL^+(N)).$$

\medskip

Then it is easy to see that for $\CF\in \Shv(\Gr_G)^I$ we have a canonical identification 
$$\on{Av}^{\on{Ad}_{t^{-\lambda}}(\fL^+(N))}_!(\CF)\simeq t^{-\lambda}\cdot j_{\lambda,!}\star \CF[\langle \lambda,2\check\rho\rangle],$$
where the right-hand side is viewed as an object of $\Shv(\Gr_G)^{\fL^+(T)}$.

\sssec{}

Note that by construction, for $\lambda\in \Lambda^+$ and $\CF$ as above, we have a canonical identification 
$$t^{\lambda}\cdot \on{Av}^{\fL(N)}_!(\CF)[-\langle \lambda,2\check\rho\rangle]
\simeq \on{Av}^{\fL(N)}_!(j_{\lambda,!}\star \CF).$$

\medskip

This shows that the functor $\on{Av}^{\fL(N)}_!$ intertwines the $\Rep(\cT)$-action on $(\SI(\Gr_G))^{\fL^+(T)}$, given by
\eqref{e:action by translations} and the $\Rep(\cT)$-action on $\Shv(\Gr_G)^I$, given by \eqref{e:bad action}.

\sssec{}

We now check that the unit of the adjunction, namely, the natural transformation
\begin{equation} \label{e:unit for adj}
\CF\to \on{Av}^{I/\fL^+(T)}_*\circ \on{Av}^{\fL(N)}_!(\CF), \quad \CF\in \Shv(\Gr_G)^I
\end{equation}
is an isomorphism. 

\medskip

Indeed, for every $\lambda\in \Lambda^+$ and $\CF'\in \Shv(\Gr_G)^I$ we have a canonical identification 
$$\on{Av}^{I/\fL^+(T)}_*(t^{-\lambda}\cdot \CF')\simeq j_{-\lambda,*}\star \CF'[-\langle \lambda,2\check\rho\rangle].$$

So, 
$$\on{Av}^{I/\fL^+(T)}_*\circ \on{Av}^{\on{Ad}_{t^{-\lambda}}(\fL^+(N))}_!(\CF)\simeq  j_{-\lambda,*}\star j_{\lambda,!}\star \CF \simeq \CF,$$
and hence
$$\on{Av}^{I/\fL^+(T)}_*\circ \on{Av}^{\fL(N)}_!(\CF)\simeq 
\underset{\lambda\in \Lambda^+}{\on{colim}}\,  \on{Av}^{I/\fL^+(T)}_*\circ \on{Av}^{\on{Ad}_{t^{-\lambda}}(\fL^+(N))}_!(\CF)\simeq
\underset{\lambda\in \Lambda^+}{\on{colim}}\, \CF\simeq \CF.$$

\medskip

Furthermore, by unwinding the definitions it is easy to see that with respect to the above
identification $\on{Av}^{I/\fL^+(T)}_*\circ \on{Av}^{\fL(N)}_!(\CF)\simeq \CF$, 
the map \eqref{e:unit for adj} is the identity map $\CF\to \CF$. 

\sssec{}

Finally, let us show that the functor $\on{Av}^{I/\fL^+(T)}_*$ is conservative on 
$\Shv(\Gr_G)^{\fL(N)\cdot  \fL^+(T)}$. 

\medskip

Let $\CF$ be a non-zero object of $\Shv(\Gr_G)^{\fL(N)\cdot  \fL^+(T)}$. Since it is non-zero as an object of
$\Shv(\Gr_G)$, there exists $\CF'\in \Shv(\Gr_G)$ equivariant with respect to some congruence subgroup
of $\fL^+(G)$ such that
$$\CMaps_{\Shv(\Gr_G)}(\CF',\CF)\neq 0.$$

By assumption $\CF'$ is equivariant with respect to $\on{Ad}_{t^{-\lambda}}(\fL^+_1(N^-))$ for $\lambda\gg 0$
(here $\fL^+_1(N^-)$ denotes the first congruence subgroup of $\fL^+(N)$). Hence,
$$\on{Av}_*^{\on{Ad}_{t^{-\lambda}}(\fL^+(N^-))}(\CF)\neq 0.$$

Note, however that since $\CF$ belongs to $\Shv(\Gr_G)^{\fL(N)\cdot  \fL^+(T)}$ and
$$I=\fL^+_1(N^-)\cdot \fL^+(T)\cdot \fL^+(N),$$ 
we have
$$\on{Av}_*^{\on{Ad}_{t^{-\lambda}}(\fL^+_1(N^-))}(\CF)\in \Shv(\Gr_G)^{\on{Ad}_{t^{-\lambda}}(I)}.$$

\medskip

Now, for any $\CF''\in \Shv(\Gr_G)^{\on{Ad}_{t^{-\lambda}}(I)}$, we have
$$t^{\lambda}\cdot \CF''\in  \Shv(\Gr_G)^I$$
and 
$$\on{Av}^{\fL^+_1(N^-)}_*(\CF'') \simeq j_{-\lambda,*}\star (t^\lambda\cdot \CF'')[-\langle \lambda,2\check\rho\rangle].$$ 

Hence, for $\CF$ as above, 
\begin{multline*}
\on{Av}^{I/\fL^+(T)}_*(\CF)\simeq \on{Av}_*^{\fL^+_1(N^-)}(\CF)\simeq 
\on{Av}_*^{\fL^+_1(N^-)}\circ \on{Av}_*^{\on{Ad}_{t^{-\lambda}}(\fL^+_1(N^-))}(\CF)\simeq  \\ 
\simeq j_{-\lambda,*}\star \left(t^\lambda\cdot \on{Av}_*^{\on{Ad}_{t^{-\lambda}}(\fL^+_1(N^-))}(\CF)\right)[-\langle \lambda,2\check\rho\rangle],
\end{multline*}
which is non-zero, because the functor $j_{-\lambda,*}\star -$ is conservative 
(it is in fact a self-equivalence of $\Shv(\Gr_G)^I$).

\ssec{The semi-infinite IC sheaf as the baby Verma object}

We are now ready to state the second main result of this paper: 

\begin{thm} \label{t:other main}
Under the equivalence of \propref{p:Iwahori N}, the object 
$$\ICs\in \Hecke_{\cG,\cT}(\Shv(\Gr_G)^{\fL(N)\cdot \fL^+(T)})$$
goes over to $\CF'_{\cG,\cT}[\dim(G/B)]$.
\end{thm}

The rest of this subsection is devoted to the proof of \thmref{t:other main}. 

\sssec{}

Since the stated isomorphism takes place in the heart of $\Shv(\Gr_G)^I$ (with respect to the perverse t-structure),
and since the actions of $\Rep(\cG)$ and $\Rep(\cT)$ are given by t-exact functors, it is enough to establish it
at the level of the homotopy categories.

\sssec{}  \label{sss:F' expl}

According to \secref{sss:baby w0}, the object $\CF'_{\cG,\cT}$ can be described as follows. The underlying object of
$\Shv(\Gr_G)^I$ is
\begin{equation} \label{e:result one}
\underset{\lambda\in \Lambda^+}{\on{colim}}\, j_{-\lambda\cdot w_0,*}\star \Sat(V^{\lambda}).
\end{equation} 

Here we are using the fact that for $\lambda$ regular, we have
$$j_{w_0,!}\star j_{-w_0(\lambda),*}\simeq j_{-\lambda\cdot w_0,*},$$
which follows from the isomorphism
$$j_{-w_0(\lambda),*}\simeq j_{w_0,*}\star j_{-\lambda\cdot w_0,*},$$
which in turn follows from the fact that
$$-w_0(\lambda)=w_0\cdot (-\lambda\cdot w_0) \text{ and } \ell(-w_0(\lambda))=\ell(w_0) + \ell(-\lambda\cdot w_0)$$
for $\lambda$ dominant and regular.  

\medskip

The transition maps are induced by \eqref{e:describe other b v}, and the Hecke structure is induced by \eqref{e:Hecke other b v}.

\sssec{}

The object of $\Shv(\Gr_G)^I$, underlying $\on{Av}^{I/\fL^+(T)}_*(\ICs)$ is given by
$$\underset{\lambda\in \Lambda^+}{\on{colim}}\,  \on{Av}^{I/\fL^+(T)}_*\left(t^{-\lambda}\cdot \Sat(V^\lambda)\right)[\langle \lambda,2\check\rho\rangle].$$

\medskip

We note that for any $\CF\in \Shv(\Gr_G)^I$, we have a canonical isomorphism  
$$\on{Av}^{I/\fL^+(T)}_*(t^{-\lambda}\cdot \CF)[\langle \lambda,2\check\rho\rangle]\simeq 
j_{-\lambda,*}\star \CF.$$

Note also that for $\lambda$ regular we have
$$j_{-\lambda,*}\simeq j_{-\lambda\cdot w_0,*}\star j_{w_0,*}.$$

Hence, for any $\CF'\in \Shv(\Gr_G)^{\fL^+(G)}$, we have a canonical isomorphism
$$j_{-\lambda,*}\star \CF'\simeq j_{-\lambda\cdot w_0,*}\star \CF'[\dim(G/B)].$$

\medskip

Hence, we obtain that 
$\on{Av}^{I/\fL^+(T)}_*(\ICs)$ is given by
\begin{equation} \label{e:result two}
\underset{\lambda\in \Lambda^+}{\on{colim}}\,  j_{-\lambda\cdot w_0,*}\star \Sat(V^\lambda)[\dim(G/B)],
\end{equation} 
where the transition maps are induced by \eqref{e:transition map}, and the Hecke structure by \eqref{e:Hecke semiinf}. 

\sssec{}

However, by unwinding the definitions, we see that the transition maps in \eqref{e:result two} coincide with those
in \eqref{e:result one}, and that the Hecke structures match up. 

\ssec{Description of the abelian category} \label{ss:proof of ab}

In this subsection we will prove Propositions \ref{p:abelian categ} and \ref{p:semiinf as reg}. 

\sssec{}

Consider the functor
\begin{equation} \label{e:functor for ab}
(\SI(\Gr_G))^{\fL^+(T)} \overset{\on{Av}^{I/\fL^+(T)}_*}\longrightarrow \Shv(\Gr)^I \overset{j_{w_0,*}\star -[-\dim(G/B)]}
\longrightarrow \Shv(\Gr)^I  \overset{\Sat^I}\longrightarrow \IndCoh(\cn^-\underset{\cg}\times \{0\}/\cB^-).
\end{equation}

According to \thmref{t:ABG} and \propref{p:Iwahori N}, the above functor is an equivalence. Note that by construction, it sends 
the object $\Delta^\lambda\in (\SI(\Gr_G))^{\fL^+(T)}$ to the object 
$$\sfe^{w_0(\lambda)}\otimes \CO_{\on{pt}/\cB^-}\in \IndCoh(\cn^-\underset{\cg}\times \{0\}/\cB^-).$$

\medskip

We will prove:

\begin{prop} \label{p:abelian categ'}  
Under the equivalence \eqref{e:functor for ab}, the abelian category $((\SI(\Gr_G))^{\fL^+(T)})^\heartsuit$
goes over to
$$(\IndCoh(\on{pt}/\cB^-))^\heartsuit\simeq (\IndCoh(\cn^-\underset{\cg}\times \{0\}/\cB^-))^\heartsuit.$$
\end{prop} 

We will also prove:

\begin{prop} \label{p:monodromic}
The inclusion $((\SI(\Gr_G))^{\fL^+(T)})^\heartsuit\hookrightarrow (\SI(\Gr_G))^\heartsuit$
is an equivalence.
\end{prop}

Note that the combination of these two propositions implies \eqref{p:abelian categ}. Note also that
\propref{p:semiinf as reg} follows this and \thmref{t:other main}. 

\sssec{Proof of \propref{p:abelian categ'}}

To prove the inclusion $\supset$ it suffices to show that the objects 
$$\sfe^{w_0(\lambda)}\otimes \CO_{\on{pt}/\cB^-}\in \IndCoh(\cn^-\underset{\cg}\times \{0\}/\cB^-)$$
go over to objects in $((\SI(\Gr_G))^{\fL^+(T)})^\heartsuit$. But, as was noted above, they correspond to 
objects $\Delta^\lambda\in (\SI(\Gr_G))^{\fL^+(T)}$, which do belong to $((\SI(\Gr_G))^{\fL^+(T)})^\heartsuit$
by \corref{c:j is perv}.

\medskip

Vice versa, by \corref{c:j is perv}, every object of $((\SI(\Gr_G))^{\fL^+(T)})^\heartsuit$ admits a 
filtration with subquotients of the form $\Delta^\lambda$. Its image in 
$\CO_{\on{pt}/\cB^-}\in \IndCoh(\cn^-\underset{\cg}\times \{0\}/\cB^-)$ thus admits a filtration with
subquotients of the form $\sfe^{w_0(\lambda)}\otimes \CO_{\on{pt}/\cB^-}$, and hence belongs to 
$$(\IndCoh(\on{pt}/\cB^-))^\heartsuit\simeq (\IndCoh(\cn^-\underset{\cg}\times \{0\}/\cB^-))^\heartsuit.$$

\qed

\sssec{Proof of \propref{p:monodromic}}

Since $\SI(\Gr_G)$ is generated by $\fL^+(T)$-equivariant objects, every object $\CF$ of 
$(\SI(\Gr_G))^\heartsuit$ is $\fL^+(T)$-\emph{monodromic}. Hence, $\CF$ admits an action
of the algebra $\Sym(\ft)$ by endomorphisms, and $\CF$ is $\fL^+(T)$-equivariant if and only
if this action is trivial.

\medskip

The Cousin decomposition of $\CF$ with respect to the strata $S^\lambda$ defines on $\CF$ a 
functorial filtration, indexed by the poset $\Lambda$, with the subquotient corresponding
to a given $\lambda\in \Lambda$ isomorphic to a direct sum of copies of $\Delta^\lambda$. 
The action of $\Sym(\ft)$ an any such subquotient is trivial. However, since the $\Delta^\lambda$'s
are pairwise non-isomorphic, the action of $\Sym(\ft)$ on $\CF$ is therefore also trivial.

\qed

\section{The Drinfeld-Pl\"ucker formalism}

The contents of this section follow a suggestion of S.~Raskin. We will present another, in a sense more
functorial, point of view that leads to the construction of $\ICs$ and the baby Verma modules. 

\ssec{Drinfeld-Pl\"ucker structures}

\sssec{}

Consider the scheme $\cG/\cN^-$. It is known to be quasi-affine, and let $\ol{\cG/\cN^-}$ be its affine closure,
(i.e., the spectrum of the algebra of global functions), so that we have an open embedding
\begin{equation} \label{e:quasi-affine}
\cG/\cN^-\hookrightarrow \ol{\cG/\cN^-}.
\end{equation} 

\medskip

By transport of structure, $\ol{\cG/\cN^-}$ is acted on by $\cG\times \cT$. Thus, we can regard 
$\CO(\ol{\cG/\cN^-})$ as a commutative algebra object in the 
symmetric monoidal category $\Rep(\cT)\otimes \Rep(\cG)$.

\sssec{}

Note that 
$$\ol{\cG/\cN^-}\simeq \underset{\lambda\in \Lambda^+}\oplus\, V^\lambda\otimes \sfe^{-\lambda},$$
with the product given by 
$$\left(V^{\lambda_1}\otimes \sfe^{-\lambda_1}\right)\otimes 
\left(V^{\lambda_2}\otimes \sfe^{-\lambda_2}\right)\to V^{\lambda_1+\lambda_2}\otimes \sfe^{-\lambda_1-\lambda_2},$$
induced by \eqref{e:Plucker map}. 

\sssec{}

Let $\CC$ be a DG category equipped with an action of $\Rep(\cT)\otimes \Rep(\cG)$. We shall say that an object
$c\in \CC$ is equipped with a \emph{Drinfeld-Pl\"ucker structure} if it is equipped with an action of 
$\CO(\ol{\cG/\cN^-})$, viewed as an algebra object in $\Rep(\cT)\otimes \Rep(\cG)$.

\medskip

We denote the category of objects of $\CC$ equipped with a Drinfeld-Pl\"ucker structure by 
$$\CO(\ol{\cG/\cN^-})\mod(\CC).$$

\sssec{}

Using the above presentation of $\ol{\cG/\cN^-}$, we can think of a Drinfeld-Pl\"ucker structure on $c$ as a system of maps
\begin{equation} \label{e:Drinfeld-Plucker}
c\star V^\lambda\to \sfe^\lambda\star c,\quad \lambda\in \Lambda^+
\end{equation} 
that make the diagrams 
$$
\CD
(c\star  V^{\lambda_1})  \star V^{\lambda_2} @>{\sim}>>  c\star  (V^{\lambda_1} \star V^{\lambda_2})  \\
@VVV   @VVV   \\
(\sfe^{\lambda_1}\star c)  \star V^{\lambda_2}  & & c\star V^{\lambda_1+\lambda_2}   \\
@V{\sim}VV   @VV{\text{\eqref{e:Plucker map dual}}}V   \\
\sfe^{\lambda_1}\star (c  \star V^{\lambda_2}) & &  \sfe^{\lambda_1+\lambda_2}\star c \\
@VVV   @VV{\sim}V  \\
\sfe^{\lambda_1}\star (\sfe^{\lambda_2} \star c)  @>{\sim}>> \sfe^{\lambda_1+\lambda_2}\star c
\endCD
$$
satisfying a coherent system of higher compatibilities (note that the latter is automatic if the corresponding mapping
spaces are discrete). 

\begin{rem}
Note that we can rewrite the maps \eqref{e:Drinfeld-Plucker} as
$$\sfe^{-\lambda}\star c\to c\star (V^\lambda)^*, \quad \lambda\in \Lambda^+.$$
Then a Drinfeld-Pl\"ucker structure on $c$ is the same as the structure of \emph{lax central object}, 
in the terminology of \eqref{ss:infty}. Here our monoidal category is $\Lambda^+$, its left action on
$\CC$ is via the monoidal functor
$$\Lambda^+\to \Rep(\cT), \quad \lambda\mapsto \sfe^{-\lambda},$$
and the lax right action is via
$$\Lambda^+\to \Rep(\cG), \quad \lambda\mapsto (V^\lambda)^*.$$
\end{rem} 

\sssec{}

By a similar token, instead of $\cG/\cN^-$, we can consider $\cG/\cN$. We denote the corresponding category by 
$\CO(\ol{\cG/\cN})\mod(\CC)$.

\medskip

Let $_{w_0}\!\CC$ be obtained from $\CC$ by 
modifying the action of $\Rep(\cT)$ on $\CC$ by applying the automorphism $w_0$ to $\cT$. A choice of 
a representative $g_{w_0}\in \cG$ of $w_0$ defines an equivalence
\begin{equation} \label{e:N vs N^-}
\CO(\ol{\cG/\cN^-})\mod({}_{w_0}\!\CC)\simeq \CO(\ol{\cG/\cN})\mod(\CC).
\end{equation}

\begin{rem}
As above, objects of $\CO(\ol{\cG/\cN})\mod(\CC)$ can be thought of as objects $c\in \CC$ equipped with a system of maps
$$\sfe^{\lambda}\star c\to c\star V^\lambda, \quad \lambda\in \Lambda^+,$$
equipped with a system of compatibilities. I.e., they are the same as lax central objects in the terminology of\eqref{ss:infty},
where the left action of $\Lambda^+$ is given by 
$$\Lambda^+\to \Rep(\cT), \quad \lambda\mapsto \sfe^\lambda$$
and the lax right action is via
$$\Lambda^+\to \Rep(\cG), \quad \lambda\mapsto V^\lambda.$$
\end{rem} 

\ssec{From Drinfeld-Pl\"ucker structures to Hecke objects}

\sssec{}

The monad associated with the functor
\begin{equation}  \label{e:induce GT}
\CC\simeq \CC\underset{\Rep(\cT)}\otimes \Rep(\cT)\to \CC\underset{\Rep(\cT)\otimes \Rep(\cG)}\otimes \Rep(\cT)=:
\Hecke_{\cG,\cT}(\CC)
\end{equation}
and its right adjoint, is given by the action of $\CO(\cG)$, viewed as an object of $\Rep(\cT)\otimes \Rep(\cG)$.

\medskip

Hence, by the Barr-Beck-Lurie theorem, we can identify
$$\Hecke_{\cG,\cT}(\CC)\simeq \CO(\cG)\mod(\CC),$$
so that the right adjoint to \eqref{e:induce GT} corresponds to the forgetful functor 
$$\oblv_{\CO(\cG)}:\CO(\cG)\mod(\CC)\to \CC.$$

\sssec{}

The composite map
$$\cG\to \cG/\cN^-\to \ol{\cG/\cN^-}$$
defines a pair of adjoint functors
$$\CO(\ol{\cG/\cN^-})\mod(\CC)\rightleftarrows \CO(\cG)\mod(\CC).$$

We will now describe the left adjoint 
\begin{equation} \label{e:DrPl left adjoint}
\CO(\ol{\cG/\cN^-})\mod(\CC)\to \CO(\cG)\mod(\CC)=\Hecke_{\cG,\cT}(\CC)
\end{equation} 
explicitly.

\sssec{}

We will first describe the composition
\begin{equation} \label{e:to be colimit}
\CO(\ol{\cG/\cN^-})\mod(\CC)\to \CO(\cG)\mod(\CC) \overset{\oblv_{\CO(\cG)}}\longrightarrow \CC.
\end{equation}

By \secref{ss:infty}, for $c\in \CO(\ol{\cG/\cN^-})\mod(\CC)$, we have a well-defined functor
$$\Lambda^+\to \CC, \quad \lambda\mapsto \sfe^\lambda\star c\star (V^\lambda)^*,$$
where the transition maps for $\lambda_2=\lambda_1+\lambda$ are given by
\begin{multline*}
\sfe^{\lambda_1}\star c\star (V^{\lambda_1})^*\simeq 
\sfe^{\lambda_1+\lambda_2}\star \sfe^{-\lambda_2} \star c\star (V^{\lambda_1})^*\to \\
\to \sfe^{\lambda_1+\lambda_2}\star c\star (V^{\lambda_2})^* \star c\star (V^{\lambda_1})^*\to 
\sfe^{\lambda_1+\lambda_2}\star c\star (V^{\lambda_1+\lambda_2})^*.
\end{multline*}

Moreover, this construction is functorial in $c\in  \CO(\ol{\cG/\cN^-})\mod(\CC)$.  We claim:

\begin{prop}  \label{p:as colim}
The functor \eqref{e:to be colimit} identifies with
$$c\mapsto \underset{\lambda\in \Lambda^+}{\on{colim}}\, \sfe^\lambda\star c\star (V^\lambda)^*.$$
\end{prop}

\begin{proof}

Note that we have:
$$\CO(\ol{\cG/\cN^-})\mod(\CC)\simeq \CC\underset{\Rep(\cT)\otimes \Rep(\cG)}\otimes \CO(\ol{\cG/\cN^-})\mod(\Rep(\cT)\otimes \Rep(\cG))$$
and
$$\CO(\cG)\mod(\CC)\simeq \CC\underset{\Rep(\cT)\otimes \Rep(\cG)}\otimes \CO(\cG)\mod(\Rep(\cT)\otimes \Rep(\cG)).$$

Hence, the assertion of the proposition reduces to the universal situation of $\CC=\Rep(\cT)\otimes \Rep(\cG)$
and $c=\CO(\ol{\cG/\cN^-})$. Thus, we need to construct an isomorphism 
$$\underset{\lambda\in \Lambda^+}{\on{colim}}\, \sfe^\lambda\otimes \CO(\ol{\cG/\cN^-}) \otimes (V^\lambda)^*\simeq  \CO(\cG)$$
as objects of $\Rep(\cT)\otimes \Rep(\cG)$. 

\medskip

Since the index category $\Lambda^+$ is filtered, the left-hand side belongs to the heart of the t-structure. So, 
the computation we need to perform is inside the abelian category. 

\medskip

The map from the left-hand side to the right-hand side is given by embedding $\CO(\ol{\cG/\cN^-})$ into $\CO(G)$, and multiplying by the 
matrix coefficient function of $v^\lambda$ against an element of $(V^\lambda)^*$. The fact that the resulting map is an isomorphism
can be seen as follows:

\medskip

Fix $\mu\in \Lambda^+$, and let us calculate $\Hom_{\Rep(\cG)}(V^\mu,-)$ into both sides. In the right hand side, we
have 
$$\Hom_{\Rep(\cG)}(V^\mu,\CO(\cG))\simeq (V^\mu)^*,$$ viewed as a $\cT$-representation. 
In the left-hand side, we have the colimit over $\lambda\in \Lambda^+$ of
$$\Hom_{\Rep(\cG)}(V^\mu,\sfe^\lambda \otimes \CO(\ol{\cG/\cN^-})\otimes (V^\lambda)^*),$$
also viewed as a $\cT$-representation.

\medskip

Now, for $\lambda$ large relative to $\mu$, using \eqref{e:razval initial}, we have:
\begin{multline*}
\Hom_{\Rep(\cG)}(V^\mu,\sfe^\lambda \otimes \CO(\ol{\cG/\cN^-})\otimes (V^\lambda)^*) \simeq 
\Hom_{\Rep(\cG)}(V^\mu\otimes V^\lambda,\CO(\ol{\cG/\cN^-}))\otimes \sfe^\lambda\overset{\text{\eqref{e:razval initial}}}\simeq \\
\simeq \underset{\nu}\oplus\,
(V^\mu(\nu))^*\otimes \Hom_{\Rep(\cG)}(V^{\lambda+\nu},\CO(\ol{\cG/\cN^-}))\otimes \sfe^\lambda\simeq 
\underset{\nu}\oplus\, (V^\mu(\nu))^*\otimes \sfe^{-\lambda-\nu}\otimes \sfe^\lambda \simeq \\
\simeq \underset{\nu}\oplus\, (V^\mu(\nu))^*\otimes \sfe^{-\nu}\simeq (V^\mu)^*.
\end{multline*}

Moreover, for such $\lambda$, the map  
$$\Hom_{\Rep(\cG)}(V^\mu,\sfe^\lambda \otimes \CO(\ol{\cG/\cN^-})\otimes (V^\lambda)^*)  \to 
\Hom_{\Rep(\cG)}(V^\mu,\CO(\cG))$$
induced by $\sfe^\lambda \otimes \CO(\ol{\cG/\cN^-})\otimes (V^\lambda)^*\to \CO(\cG)$,
is the identity map $(V^\mu)^*\to (V^\mu)^*$.

\medskip

This establishes the desired isomorphism.

\end{proof}

\sssec{}

Next, we wish to describe the structure of an object of $\Hecke_{\cG,\cT}(\CC)\simeq \CO(\cG)\mod(\CC)$ on the colimit
$\underset{\lambda\in \Lambda^+}{\on{colim}}\, \sfe^\lambda\star c\star (V^\lambda)^*$. We claim that for an individual
finite-dimensional $V\in \Rep(\cG)^\heartsuit$, the corresponding isomorphism
$$\left(\underset{\lambda\in \Lambda^+}{\on{colim}}\, \sfe^\lambda\star c\star (V^\lambda)^*\right)\star V \simeq 
\Res^\cG_\cT(V)\star \left(\underset{\lambda\in \Lambda^+}{\on{colim}}\, \sfe^\lambda\star c\star (V^\lambda)^*\right)$$
is given as follows:

\medskip

Let us take $\lambda$ large enough so that \eqref{e:razval initial} applies. Then we have:
\begin{multline*}
\left(\sfe^\lambda\star c\star (V^\lambda)^*\right)\star V \overset{\text{\eqref{e:razval}}}\simeq \\
\simeq \underset{\mu}\oplus\, \sfe^\lambda\star c\star (V^{\lambda+\mu})^*\otimes V(-\mu)
\simeq \underset{\mu}\oplus\, (\sfe^{-\mu}\otimes V(-\mu))\star \left(\sfe^{\lambda+\mu}\star c\star (V^{\lambda+\mu})^*\right)\simeq \\
\simeq \Res^\cG_\cT(V) \star \left(\sfe^{\lambda+\mu}\star c\star (V^{\lambda+\mu})^*\right),
\end{multline*}

To establish this description, as in the proof of \propref{p:as colim}, it is enough to do in the universal case of
$\CC=\Rep(\cT)\otimes \Rep(\cG)$ and $c=\CO(\ol{\cG/\cN^-})$, in which case it becomes an elementary verification. 

\sssec{}

We have an obvious analog of \propref{p:as colim} when we replace $\cG/\cN^-$ by $\cG/\cN$. Namely, the corresponding functor
$$\CO(\ol{\cG/\cN})\mod(\CC)\to \CO(\cG)\mod(\CC) \overset{\oblv_{\CO(\cG)}}\longrightarrow \CC$$
is given by
$$c\mapsto \underset{\lambda\in \Lambda^+}{\on{colim}}\, \sfe^{-\lambda}\star c\star V^\lambda,$$
and the Hecke structure is given by
\begin{multline*}
\left(\sfe^{-\lambda}\star c\star V^\lambda\right)\star V \overset{\text{\eqref{e:razval initial}}}\simeq \\
\simeq \underset{\mu}\oplus\, \sfe^{-\lambda}\star c\star V^{\lambda+\mu}\otimes V(\mu)
\simeq \underset{\mu}\oplus\, (\sfe^{\mu}\otimes V(\mu))\star \left(\sfe^{-\lambda-\mu}\star c\star V^{\lambda+\mu}\right)\simeq \\
\simeq \Res^\cG_\cT(V) \star \left(\sfe^{-\lambda-\mu}\star c\star V^{\lambda+\mu}\right),
\end{multline*}

\ssec{Construction of $\ICs$ and other objects from the Drinfeld-Pl\"ucker picture}   \label{ss:Drinfeld-Plucker}

We will now apply the above discussion and give a conceptual explanation of the colimit formations
that we encountered earlier in the paper.

\sssec{}

Let $\CC$ be a category equipped with an action of $\Rep(\cB^-)$. We endow it with 
actions of $\Rep(\cG)$ and $\Rep(\cT)$
coming from the restriction functors
$$\Rep(\cG)\to \Rep(\cB^-)\leftarrow \Rep(\cT).$$

Then any object of $\CC$ is automatically equipped with a Drinfeld-Plucker structure:
indeed, this structure comes from the canonical system of maps
$$\Res^{\cG}_{\cB^-}(V^\lambda)\to \Res^{\cT}_{\cB^-}(\sfe^\lambda),$$
given by the covectors $(v^\lambda)^*$.

\sssec{}

First, take $\CC=\Rep(\cB^-)$. Applying the above construction to $c=\sfe$,  and applying the functor \eqref{e:DrPl left adjoint} 
we obtain an object of 
$$\Hecke_{\cG,\cT}(\Rep(\cB^-))\simeq \QCoh((\cG/\cN^-)/\on{Ad}_{\cT})$$
equal to the image of $\CO_{\on{pt}/\cT}$ along the map
$$\on{pt}/\cT\to \QCoh((\cG/\cN^-)/\on{Ad}_{\cT}).$$

This is the object, considered in \secref{sss:coherent Verma}. Its description as a colimit in \secref{sss:coherent b Verma expl}
coincides with one given by \propref{p:as colim}.

\sssec{}  \label{sss:DrPl coh}

Next we take $\CC=\IndCoh(\cn^-\underset{\cg}\times \{0\}/\cB^-)$, equipped with an action of $\Rep(\cB^-)$,
coming from the projection
$$\cn^-\underset{\cg}\times \{0\}/\cB^-\to \on{pt}/\cB^-.$$

Take $c$ equal to the direct image of the structure sheaf along the closed embedding
$$\{0\}/\cB^- \to \cn^-\underset{\cg}\times \{0\}/\cB^-.$$

The corresponding object of 
$$\Hecke_{\cG,\cT}(\IndCoh(\cn^-\underset{\cg}\times \{0\}/\cB^-))$$
is the object that we denoted $\CM_{\cG,\cT}$. 

\sssec{}

Next, we take $\CC=\Shv(\Gr_G)^I$, equipped with an action of $\Rep(\cG)$ coming right convolutions by 
$\Sph(G)$ and $\Sat:\Rep(\cG)\to \Sph(G)$. We equip it with the action of $\Rep(\cT)$ given by \eqref{e:good action}.
(Note that according to \thmref{t:ABG}, this situation is equivalent to the one in \secref{sss:DrPl coh}.)

\medskip

We take $c=\delta_{1,\Gr_G}$. It is equipped with a Drinfeld-Plucker structure by means of
$$\IC_{\ol\Gr_G^\lambda}\to j_{\lambda,*}\star \delta_{1,\Gr_G}.$$

\medskip

The corresponding object of $\Hecke_{\cG,\cT}(\Shv(\Gr_G)^I)$ is the object that in \secref{ss:geom Verma}
we denoted by $\CF_{\cG,\cT}$.  Its presentation as a colimit in \secref{sss:geom Verma colim} is the one given by
\propref{p:as colim}.

\sssec{}

We again take $\CC=\Shv(\Gr_G)^I$, but equip it with the $\Rep(\cT)$-action, given by \eqref{e:bad action}.
We claim that the object $\delta_{1,\Gr_G}$ canonically lifts to an object
of
$$\CO(\ol{\cG/\cN})\mod(\Shv(\Gr_G)^I).$$

The corresponding maps are given by
$$j_{\lambda,!}\star \delta_{1,\Gr_G}\to \IC_{\ol\Gr_G^\lambda}.$$

\medskip

The resulting object of 
$$\Hecke'_{\cG,\cT}(\Shv(\Gr_G)^I)$$
is what we denoted by $\CF'_{\cG,\cT}[\dim(G/B)]$. The cohomological shift comes from the identification
$$j_{w_0,!}\star  \delta_{1,\Gr_G}\simeq  \delta_{1,\Gr_G}[-\dim(G/B)].$$

The description of $\CF'_{\cG,\cT}$, given in \secref{sss:F' expl}, coincides with the one given by \propref{p:as colim}.

\sssec{}

Let us now take $\CC=\Shv(\Gr_G)^{\fL^+(T)}$. It is equipped with an action of $\Rep(\cG)$ coming right convolutions by 
$\Sph(G)$ and $\Sat:\Rep(\cG)\to \Sph(G)$.  It is equipped with an action of $\Rep(\cT)$ coming from he action
of $\Lambda\simeq \fL(T)/\fl^+(T)$ by left translations with a cohomological shift:
$$\sfe^\lambda\cdot \CF:=t^\lambda \cdot \CF[-\langle \lambda,2\check\rho\rangle].$$

\medskip

Take $c=\delta_{1,\Gr_G}$; it upgrades to an object of $\CO(\ol{\cG/\cN})\mod(\Shv(\Gr_G)^{\fl^+(T)})$
via the maps
$$\delta_{t^\lambda,\Gr_G}[-\langle \lambda,2\check\rho\rangle]\to \IC_{\ol\Gr_G^\lambda}.$$

The resulting object of $\Hecke_{\cG,\cT}(\Shv(\Gr_G)^{\fl^+(T)})$ is our $\ICs$, equipped with the Hecke
structure as in \secref{ss:Hecke on IC}. 

\sssec{}

Finally, let us take $\CC=\Shv(\Gr_G)^{\fL(N)\cdot \fL^+(T)}$. The above $\Rep(\cT)\otimes \Rep(\cG)$
action on $\Shv(\Gr_G)^{\fL^+(T)}$ induces one on $\Shv(\Gr_G)^{\fL(N)\cdot \fL^+(T)}$.

\medskip

We take $c=\Delta_0\simeq \on{Av}^{\fL(N)}_!(\delta_{1,\Gr_G})$. The structure on $\delta_{1,\Gr_G}$
of object in the category $\CO(\ol{\cG/\cN})\mod(\Shv(\Gr_G)^{\fl^+(T)})$ induces a structure on $\Delta_0$ of object in 
$\CO(\ol{\cG/\cN})\mod(\Shv(\Gr_G)^{\fl(N)\cdot \fl^+(T)})$. 

\medskip

The resulting object of $\Hecke_{\cG,\cT}(\Shv(\Gr_G)^{\fl(N)\cdot \fl^+(T)})$ is again $\ICs$.

\sssec{}

It is clear from the construction that the structure of objects of $\CO(\ol{\cG/\cN})\mod(\Shv(\Gr_G)^{\fl(N)\cdot \fl^+(T)})$
and $\CO(\ol{\cG/\cN})\mod(\Shv(\Gr_G)^I)$ on $\Delta_0$ and $\delta_{1,\Gr_G}$, respectively, match up under the
equivalence of \propref{p:Iwahori N}.  

\medskip

This gives a ``formula-free'' proof of \thmref{t:other main}.

\end{document}